\documentclass[11pt]{article}

\usepackage[a4paper,left=20mm,top=20mm,right=20mm,bottom=25mm]{geometry}

\usepackage{amsmath, amsfonts, amssymb, amsthm}
\usepackage[noend]{algorithmic}
\usepackage[linesnumbered,ruled,vlined]{algorithm2e}

\usepackage{mathtools}
\usepackage[mathscr]{euscript}
\usepackage{wasysym}

\usepackage{graphicx}
\usepackage{xcolor}
\usepackage{enumitem}
\usepackage{authblk}

\usepackage[authoryear]{natbib}
\setcitestyle{authoryear,open={(},close={)}}
\let\cite\citep

\usepackage{color}

\usepackage[colorlinks]{hyperref}
\hypersetup{
  citecolor=blue,
  linkcolor=blue,
}

\usepackage[font=footnotesize,width=.85\textwidth,labelfont=bf]{caption}

\newtheorem{theorem}{Theorem} 
\newtheorem{lemma}{Lemma}
\newtheorem{corollary}{Corollary}
\newtheorem{definition}{Definition}
\newtheorem{proposition}{Proposition}
\newtheorem{fact}{Observation}
\newtheorem{problem}{Problem}

\newtheorem{innercustomthm}{Theorem}
\newenvironment{ctheorem}[1]
{\renewcommand\theinnercustomthm{#1}\innercustomthm}
{\endinnercustomthm}

\newtheorem{innercustomlem}{Lemma}
\newenvironment{clemma}[1]
{\renewcommand\theinnercustomlem{#1}\innercustomlem}
{\endinnercustomlem}

\newtheorem{inner-claim}{Claim}
\newenvironment{claim-proof}%
{\begin{description}[leftmargin = 0.2cm, labelsep = 0.2cm]
    \item \emph{Proof of Claim:}}{\hfill$\diamond$\end{description}}

\newcommand{\PROBLEM}[1]{\textsc{#1}}
\DeclareMathOperator{\symdiff}{\triangle}

\def\arrowedvec{\mathaccent"017E}
\newcommand{\G}{\arrowedvec{G}}

\newcommand{\child}{\mathsf{child}}

\newcommand{\ur}{\textit{UR}\xspace}

\DeclareMathOperator{\lca}{lca}
\DeclareMathOperator{\Aho}{Aho}
\newcommand{\Rbin}{\mathop{\mathscr{R}^{\textrm{B}}}}

\newcommand{\cupdot}{\charfusion[\mathbin]{\cup}{\cdot}}
\newcommand{\bigcupdot}{\charfusion[\mathbin]{\bigcup}{\cdot}}
\makeatletter
\def\moverlay{\mathpalette\mov@rlay}
\def\mov@rlay#1#2{\leavevmode\vtop{%
    \baselineskip\z@skip \lineskiplimit-\maxdimen
    \ialign{\hfil$\m@th#1##$\hfil\cr#2\crcr}}}
\newcommand{\charfusion}[3][\mathord]{
  #1{\ifx#1\mathop\vphantom{#2}\fi
    \mathpalette\mov@rlay{#2\cr#3}
  }
  \ifx#1\mathop\expandafter\displaylimits\fi}
\DeclareRobustCommand\bigop[1]{%
  \mathop{\vphantom{\sum}\mathpalette\bigop@{#1}}\slimits@
}
\newcommand{\bigop@}[2]{%
  \vcenter{%
    \sbox\z@{$#1\sum$}%
    
\hbox{\resizebox{\ifx#1\displaystyle.9\fi\dimexpr\ht\z@+\dp\z@}{!}{$\m@th#2$}}%
  }%
}
\makeatother

\providecommand{\keywords}[1]{\textbf{\textit{Keywords: }} #1}

\title{Heuristic Algorithms for Best Match Graph Editing}

\author[1,2]{David Schaller}
\author[3]{Manuela Gei{\ss}}
\author[4]{Marc Hellmuth}
\author[1,2,5-7]{Peter F.\ Stadler}

\affil[1]{Max Planck Institute for Mathematics in the Sciences,
          Inselstra{\ss}e 22, D-04103 Leipzig, Germany
          \authorcr \texttt{sdavid@bioinf.uni-leipzig.de}}

\affil[2]{Bioinformatics Group, Department of Computer Science \&
          Interdisciplinary Center for Bioinformatics, Universit{\"a}t Leipzig,
          H{\"a}rtelstra{\ss}e~16--18, D-04107 Leipzig, Germany
          \authorcr \texttt{studla@bioinf.uni-leipzig.de}}

\affil[3]{Software Competence Center Hagenberg GmbH, Hagenberg, Austria
          \authorcr \texttt{manuela.geiss@scch.at}}

\affil[4]{Department of Mathematics, Faculty of Science,
          Stockholm University, SE-10691 Stockholm, Sweden
          \authorcr \texttt{marc.hellmuth@math.su.se}}

\affil[5]{Institute for Theoretical Chemistry, University of Vienna,
          W{\"a}hringerstrasse 17, A-1090 Wien, Austria}

\affil[6]{Facultad de Ciencias, Universidad National de Colombia, Sede
          Bogot{\'a}, Colombia}

\affil[7]{Santa Fe Insitute, 1399 Hyde Park Rd., Santa Fe NM 87501, USA}

\date{\ }

\begin{document}

\maketitle 

\abstract{  
  Best match graphs (BMGs) are a class of colored digraphs that
  naturally appear in mathematical phylogenetics and can be approximated
  with the help of similarity measures between gene sequences, albeit not
  without errors. The corresponding graph editing problem can be used as
  a means of error correction. Since the arc set modification problems
  for BMGs are NP-complete, efficient heuristics are needed if BMGs are
  to be used for the practical analysis of biological sequence data. Since
  BMGs have a characterization in terms of consistency of a certain set
  of rooted triples, we consider heuristics that operate on triple sets.
  As an alternative, we show that there is a close connection to a set
  partitioning problem that leads to a class of top-down recursive
  algorithms that are similar to Aho's supertree algorithm and give rise
  to BMG editing algorithms that are consistent in the sense that they
  leave BMGs invariant. Extensive benchmarking shows that
  community detection algorithms for the partitioning steps perform best
  for BMG editing.
}

\bigskip
\noindent

\sloppy

\section{Introduction}

A wide range of tasks in computational biology start by determining, for a
given ``query gene'' $x$ in a species $A$, one or all genes $y$ in another
species $B$ that are most similar to $x$. Conceptually, these \emph{best
  matches} $y$ of the query genes $x$ are meant to approximate the set of
genes in $B$ that are evolutionary most closely related to $x$. Best
matches can be identified by comparing evolutionary distances
\cite{Nei:06}, which in turn are usually obtained from sequence alignments
\cite{Chatzou:16}. In practice, fast approximation algorithms such as
\texttt{blast} and its more modern successors are often used for this
purpose \cite{MorenoHagelsieb:08,HernandezSalmeron:20}.  Even if sequence
similarity is measured perfectly, deviations from a common molecular clock,
i.e., differences in the evolutionary rates of different genes, cause
discrepancies between \emph{best hits} (most similar sequences) and
\emph{best matches} (evolutionary most closely related sequences), see
\cite{Stadler:20a} for a detailed discussion.

The idea of best matches in the sense of closest evolutionary relatedness
pre-supposes an underlying tree $T$ that describes the phylogenetic
relationships among the genes, which correspond to the leaves of $T$, and a
map $\sigma$ assigning to each gene $x$ the species $\sigma(x)$ in which it
resides. Given such a leaf-colored tree $(T,\sigma)$, the best match graph
$\G(T,\sigma)$ has as its vertex set the leaves of $T$, i.e., the set of
genes, and as (directed) arcs the best matches. The latter are defined as
the pairs $(x,y)$ for which the last common ancestor of $x$ and $y$ is at
least as close to $x$ as the last common ancestor of $x$ and any other gene
$y'$ from the same spaces $\sigma(y')=\sigma(y)$. Best match graphs (BMGs),
i.e., graphs that are derived from a leaf-labeled tree $(T,\sigma)$ in this
manner, form a very restrictive class of colored digraphs
\cite{Geiss:19a,BMG-corrigendum}. Empirically determined best hit data
therefore will in general not satisfy the defining properties of BMGs. They
can be corrected in part, however, by considering the problem of editing a
given graph to the closest BMG. \citet{Schaller:20y} showed that the
arc modification problems for BMGs are NP-hard, but can be formulated as
integer linear programs (ILPs) allowing practical solutions for small
instances. However, in computational biology, applications to large gene
families would be of particular interest, creating the need for faster,
approximate solutions for BMG editing. Before embarking to develop software
for a BMG-based analysis of large sequence data sets, we need to understand
whether the editing problem for BMGs is tractable in practice with
sufficient accuracy and for interestingly large instances.  The purpose of
this contributions is to establish that this is indeed the case.

Motivated by both theoretical and practical considerations, we are mainly
interested in heuristics that are consistent in the following sense: Let
$\mathbb{A}$ be an algorithm that takes an arbitrary vertex-colored digraph
$(\G,\sigma)$ as input and outputs a BMG $\mathbb{A}(\G,\sigma)$. Then
$\mathbb{A}$ is consistent if $\mathbb{A}(\G,\sigma)=(\G,\sigma)$ whenever
the input graph $(\G,\sigma)$ is a BMG.

A vertex-colored digraph $(\G=(V,E),\sigma)$ is a BMG if and only if (a) a
certain set of informative triples $\mathscr{R}(\G,\sigma)$, which can
easily be read off the input graph, is consistent and (b) the BMG
$\G(\hat T, \sigma)$ of the corresponding so-called Aho tree
$\hat T \coloneqq \Aho(\mathscr{R}(\G,\sigma),V)$ coincides with $(\G,\sigma)$
\cite{BMG-corrigendum}. In general, the Aho tree $\Aho(\mathscr{R}, V)$ of
a compatible set of triples $\mathscr{R}$ on a set $V$ is a least resolved
super-tree of all the triples in $\mathscr{R}$. For a BMG,
$\hat T \coloneqq \Aho(\mathscr{R}(\G,\sigma), V)$ is the unique least
resolved tree (LRT) for $(\G,\sigma)=\G(\hat T,\sigma)$. These close
connections between recognizing BMGs and constructing super-trees suggest
to adapt ideas from heuristic algorithms for triple consistency problems
and super-tree construction for BMG editing.

The simplest approach, therefore, is to extract a maximal consistent subset
$\mathscr{R}^*$ from $\mathscr{R}(\G,\sigma)$ and to use
$\G(\Aho(\mathscr{R}^*,V),\sigma)$ as an approximation, see
Sec.~\ref{sect:simple}. A more detailed analysis of arcs in $(\G,\sigma)$
that violate the property of being a BMG in Sec.~\ref{sect:UR}, however,
leads a notion of ``unsatisfiable relations'' (UR), which can be used to
count the arc modifications associated with a partition $\mathscr{V}$ of
the vertex set $V$ of $\G$. It also gives raise to a top-down algorithm in
which the vertex set of $\G$ is recursively edited and partitioned. A large
class of heuristics for BMG editing can be constructed depending on the
construction of the partition $\mathscr{V}$ in each recursion step. We
shall see that the arc edit sets in different steps of the recursion are
disjoint. A main result of this contribution,
Thm.~\ref{thm:algo-consistent}, links the partitions $\mathscr{V}$
appearing in BMG editing algorithms to the auxiliary graphs appearing in the
\texttt{BUILD} algorithm for supertree construction \cite{Aho:81}, see
Sec.~\ref{sect:notation}.  This provides a guarantee that the BMG editing
algorithms are consistent provided the choice of $\mathscr{V}$ is such that it
does not enforce edits whenever an alternative partition with and empty UR
is available. For BMGs, this is in particular the case for the partitions
appearing in the \texttt{BUILD} algorithm.  In Sec.~\ref{sect:UR-NP}, we
proceed to show by reduction from \PROBLEM{Set Splitting} that finding
a partition with a minimal number of unsatisfiable relations is NP-hard.

The theoretical results are complemented by computational experiments on
BMGs with randomly perturbed arc sets in
Sec.~\ref{sec:heur-computational-exp}.  We focus on a comparison of
different algorithms to construct the partitions $\mathscr{V}$. Somewhat
surprisingly, we find that minimizing the cardinality of the UR alone is
not the best approach, since this tends to produce very unbalanced
partitions and thus requires a large number of steps in the recursions
whose costs add up. Instead, certain types of clustering or community
detection approaches that favor more balanced partitions tend to perform well.

\section{Notation and Preliminaries}
\label{sect:notation}

\paragraph{Partitions.} $\mathscr{V}=\{V_1,V_2,\dots,V_k\}$ is a partition
of a set $V$ if (i) $V_i\ne\emptyset$, (ii) $\bigcup_{i=1}^k V_i = V$ and
(iii) $V_i\cap V_j=\emptyset$ for $i\ne j$. A partition is non-trivial if
$|\mathscr{V}|\ge2$. Consider two partitions
$\mathscr{V}=\{V_1,\dots,V_k\}$ and $\mathscr{V}'=\{V'_1,\dots,V'_l\}$ of
$V$. If for every $1\le j'\le l$ there is a $j$ such that
$V'_{j'}\subseteq V_j$, i.e., if every set in $\mathscr{V}'$ is
completely contained in a set in $\mathscr{V}$, then $\mathscr{V}'$ is a
\emph{refinement} of $\mathscr{V}$, and $\mathscr{V}$ is a
\emph{coarse-graining} of $\mathscr{V}'$.

\paragraph{Graphs.} 
Mostly, we consider simple directed graphs (digraphs) $\G=(V,E)$ with
vertex set $V$ and arc set
$E\subseteq V\times V \setminus \{(v,v)\mid v\in V\}$. We will frequently
write $V(\G)$ and $E(\G)$ to explicitly refer to the graph $\G$. For a
vertex $x\in V$, we say that $(y,x)$ is an \emph{in-arc} and $(x,z)$ is an
\emph{out-arc}. The subgraph induced by a subset $W\subseteq V$ is denoted
by $\G[W]$. Undirected graphs can be identified with symmetric digraphs,
i.e., the undirected graph $G$ underlying a digraph $\G$ is obtained by
dropping the direction of all arcs, or by symmetrizing the digraph, i.e.,
adding the arc $(y,x)$ to $E(\G)$ for every arc $(x,y)\in E(\G)$.  When
referring to an undirected graph $G$, we write $xy$ for
$(x,y),(y,x)\in E(G)$ and call $xy$ an \emph{edge}.  The (weakly) connected
components of $\G$ are the maximal connected subgraphs of the undirected
graph underlying $\G$ or, equivalently, the maximal strongly connected
components of the symmetrized digraph.  Whenever the context is clear, we
will also refer to the partition of $V$ formed by the vertex sets of the
maximal connected subgraphs as the set of connected components.

A vertex coloring is a map $\sigma:V\to S$, where $S$ is a non-empty set of
colors.  A vertex coloring of $\G$ is \emph{proper} if
$\sigma(x)\ne\sigma(y)$ whenever $(x,y)\in E(\G)$.  We write $(\G,\sigma)$
for a vertex-colored digraph and denote by $V[r]$ the subset of vertices of
a graph $(\G=(V,E),\sigma)$ that have color $r$.  Moreover, we define
$\sigma(W)\coloneqq \{\sigma(x) \mid x\in W\}$ for the subset of colors
present in a set $W\subseteq V$.

We write $N(x)$ for the set of out-neighbors of a vertex $x\in V(\G)$ and
$N^-(x)$ for the set of in-neighbors of $x$. A graph $\G$ is called
\emph{sink-free} if $N(x)\ne\emptyset$ holds for all $x\in V(\G)$.
We write
$A\symdiff B \coloneqq (A\setminus B)\cup(B\setminus A)$ for the symmetric
difference of two sets $A$ and $B$, and define, for a graph $\G=(V,E)$ and
arc set $F\subseteq (V\times V)\setminus\{(v,v)\mid v\in V\}$, the graph
$\G\symdiff F\coloneqq(V, E\symdiff F)$.
Analogously, we write $\G+ F\coloneqq(V, E\cup F)$ and $\G- F\coloneqq(V, 
E\setminus F)$.

\paragraph{Phylogenetic trees.} Consider an undirected, rooted tree $T$
with leaf set $L(T)\subseteq V(T)$ and root $\rho_T\in V(T)$. Its inner
vertices are given by the set $V^0(T) = V(T) \setminus L(T)$. The 
\emph{ancestor order} on $V(T)$ is defined such that
$u\preceq_T v$ if $v$ lies on the unique path from $u$ to the root
$\rho_T$, i.e., if $v$ is an ancestor of $v$. For brevity we set
$u \prec_T v$ if $u \preceq_{T} v$ and $u \neq v$. If $xy$ is an edge in
$T$ such that $y \prec_{T} x$, then $x$ is the \emph{parent} of $y$ and $y$
the \emph{child} of $x$. The set of children of a vertex $x\in V(T)$ is
denoted by $\child_T(x)$. 
A tree is \emph{phylogenetic} if all its inner vertices have at least two 
children. All trees considered in this contribution will be phylogenetic.
For a non-empty subset $A\subseteq V(T)$, we
define $\lca_T(A)$, the \emph{last common ancestor of $A$}, to be the
unique $\preceq_T$-minimal vertex of $T$ that is an ancestor of every
$u\in A$.  Following e.g.\ \citet{Bryant:95}, we denote by $T_{L'}$ the
\emph{restriction} of $T$ to a subset $L'\subseteq L(T)$, i.e.\ $T_{L'}$ is
obtained by identifying the (unique) minimal subtree of $T$ that connects
all leaves in $L'$, and suppressing all vertices with degree two except
possibly the root $\rho_{T_{L'}}=\lca_T(L')$.  We say that $T$
\emph{displays} or \emph{is a refinement of} a tree $T'$, in symbols
$T'\le T$, if $T'$ can be obtained from a restriction $T_{L'}$ of $T$ by a
series of inner edge contractions.  $(T,\sigma)$ is a leaf-colored tree if
$\sigma: L(T)\to S$ is a map from the leaves of $T$ into a non-empty set of
colors.  We say that $(T',\sigma')$ is displayed by $(T,\sigma)$ if
$T'\le T$ and $\sigma(v)=\sigma'(v)$ for all $v\in L(T')$.

\paragraph{Rooted triples.} A \emph{(rooted) triple} is a tree on three
leaves and with two inner vertices.  We write $xy|z$ for the triple on the
leaves $x,y$ and $z$ if the path from $x$ to $y$ does not intersect the
path from $z$ to the root in $T$, i.e., if
$\lca_T(x,y)\prec_T \lca_T(x,z)=\lca_T(y,z)$. In this case we say that $T$
displays $xy|z$. We write
$\mathscr{R}_{|L'} \coloneqq\left\{ xy|z \in \mathscr{R} \,\colon x,y,z\in
L' \right\}$ for the restriction of a triple set $\mathscr{R}$ to a set
$L'$ of leaves.  A set $\mathscr{R}$ of triples is \emph{consistent} if
there is a tree $T$ with leaf set $L:=\bigcup_{T'\in\mathscr{R}} L(T')$
that displays every triple in $\mathscr{R}$. The polynomial-time algorithm
\texttt{BUILD} decides for every triple set $\mathscr{R}$ whether it is
consistent, and if so, constructs a particular tree, the \emph{Aho tree}
$\Aho(\mathscr{R}, L)$, that displays every triple in $\mathscr{R}$
\cite{Aho:81}.  The algorithm relies on the construction of an (undirected)
auxiliary graph, the \emph{Aho graph}, for a given triple set $\mathscr{R}$
on a set of leaves $L$. This graph, denoted by
$[\mathscr{R}, L]$, contains an edge $xy$ if and only if
$xy|z\in\mathscr{R}$ for some $z\in L$.

\section{Best match graphs}
\label{sect:BMG}

\begin{definition}
  Let $(T,\sigma)$ be a leaf-colored tree. A leaf $y\in L(T)$ is a
  \emph{best match} of the leaf $x\in L(T)$ if $\sigma(x)\neq\sigma(y)$ and
  $\lca(x,y)\preceq_T \lca(x,y')$ holds for all leaves $y'$ of color
  $\sigma(y')=\sigma(y)$.
  \label{def:BMG}
\end{definition}
Given $(T,\sigma)$, the graph $\G(T,\sigma) = (V,E)$ with vertex set
$V=L(T)$, vertex-coloring $\sigma$, and with arcs $(x,y)\in E$ if and only
if $y$ is a best match of $x$ w.r.t.\ $(T,\sigma)$ is called the \emph{best
  match graph} (BMG) of $(T,\sigma)$ \cite{Geiss:19a}.

\begin{definition}\label{def:BestMatchGraph}
  An arbitrary vertex-colored graph $(\G,\sigma)$ is a \emph{best match
    graph (BMG)} if there exists a leaf-colored tree $(T,\sigma)$ such that
  $(\G,\sigma) = \G(T,\sigma)$. In this case, we say that $(T,\sigma)$
  \emph{explains} $(\G,\sigma)$.
\end{definition}
We say that $(\G=(V,E),\sigma)$ is an $\ell$-BMG if
$|\sigma(V)|=\ell$.  By construction, there is at least one best match of
$x$ for every color $s\in\sigma(V)\setminus\{\sigma(x)\}$:
\begin{fact}
  For every vertex $x$ and every color $s\neq \sigma(x)$ in a BMG
  $(\G,\sigma)$ there is some vertex $y\in N(x)$ with $\sigma(y)=s$.
  Equivalently, the subgraph induced by every pair of colors is sink-free.
  \label{fact:allcolors-out}
\end{fact}
In particular, therefore, BMGs are sink-free whenever they contain at least
two colors.  We formalize this basic property of BMGs for colored graphs in
general:
\begin{definition}
  Let $(G=(V,E),\sigma)$ be a colored graph. The coloring $\sigma$ is
  \emph{sink-free} if it is proper and, for every vertex $x\in V$ and
  every color $s\ne\sigma(x)$ in $\sigma(V)$, there is a vertex $y\in N(x)$
  with $\sigma(y)=s$.  A graph with a sink-free coloring is called
  \emph{sf-colored}.
\end{definition}

Given a tree $T$ and an edge $e$, we denote by $T_e$ the tree obtained from
$T$ by contracting the edge $e$. We call an edge in $(T,\sigma)$
\emph{redundant (w.r.t.\ $(\G,\sigma)$)} if both $(T,\sigma)$ and
$(T_e,\sigma)$ explain $(\G,\sigma)$.
\begin{definition}
  A tree $(T,\sigma)$ is \emph{least resolved} for a BMG $(\G,\sigma)$ if
  (i) it explains $(\G,\sigma)$ and (ii) it does not contain any redundant
  edges w.r.t.\ $(\G,\sigma)$.
  \label{def:LRT}
\end{definition}
\citet[Thm.~8]{Geiss:19a} showed that every BMG has a unique least resolved
tree (LRT).

\citet{Geiss:19a} gave a characterization of BMGs that makes use of
a set of informative triples, which can be defined compactly as
follows \cite{Schaller:20x}:
\begin{definition}\label{def:informative_triples}
  Let $(\G,\sigma)$ be a vertex-colored digraph. Then the set of
  \emph{informative triples} is
  \begin{align*}
    \mathscr{R}(\G,\sigma) &\coloneqq
    \left\{ab|b' \colon
    \sigma(a)\neq\sigma(b)=\sigma(b'),\,
    (a,b)\in E(\G), \text{ and }
    (a,b')\notin E(\G) \right\},
    \intertext{and the set of \emph{forbidden triples} is}
    \mathscr{F}(\G,\sigma) &\coloneqq
    \left\{ab|b' \colon
    \sigma(a)\neq\sigma(b)=\sigma(b'),\,
    b\ne b',\, \text{ and }
    (a,b),(a,b')\in E(\G) \right\}.
    \intertext{For the subclass of BMGs that can be explained by binary 
      trees, we will furthermore need} 
    \Rbin(\G,\sigma) &\coloneqq
    \mathscr{R}(\G,\sigma) \cup \{bb'|a\colon ab|b'\in \mathscr{F}(\G,\sigma)
    \text{ and }\sigma(b)=\sigma(b')\}.
  \end{align*}
\end{definition}
By definition, $a,b,b'$ must be pairwise distinct whenever
$ab|b'\in\mathscr{R}(\G,\sigma)$, $ab|b'\in\mathscr{F}(\G,\sigma)$, or
$ab|b'\in\Rbin(\G,\sigma)$.

We extend the notion of consistency to pairs of triple sets in
\begin{definition}
  Let $\mathscr{R}$ and $\mathscr{F}$ be sets of triples. The pair
  $(\mathscr{R},\mathscr{F})$ is called \emph{consistent} if there is a
  tree $T$ that displays all triples in $\mathscr{R}$ but none of the
  triples in $\mathscr{F}$. In this case, we also say that $T$
  \emph{agrees with} $(\mathscr{R},\mathscr{F})$.
\end{definition}
It can be decided in polynomial time whether such a pair
$(\mathscr{R},\mathscr{F})$ is consistent using the algorithm \texttt{MTT},
which was named after the corresponding \emph{mixed triplets problem
  restricted to trees} and described by \citet{He:06}. 

We continue with two simple observations concerning the restriction of
triple sets.  Since informative and forbidden triples $xy|z$ are only
defined by the presence and absence of arcs in the subgraph of $\G$ induced
by $\{x,y,z\}$, this leads to the following
\begin{fact} \cite{Schaller:20p}
  \label{obs:R-restriction}
  Let $(\G,\sigma)$ be a vertex-colored digraph and $V'\subseteq V(\G)$.
  Then $R(\G,\sigma)_{|V'}=R(\G[V'],\sigma_{|V'})$ holds for every
  $R\in\{\mathscr{R},\mathscr{F}, \Rbin \}$.
\end{fact}
Moreover, any pair of triples $(\mathscr{R}',\mathscr{F}')$ such that
$\mathscr{R}'\subseteq\mathscr{R}$ and $\mathscr{F}'\subseteq\mathscr{F}$
for a consistent pair $(\mathscr{R},\mathscr{F})$ remains consistent since
any tree that agrees with $(\mathscr{R},\mathscr{F})$ clearly displays all
triples in $\mathscr{R}'$ and none of the triples in $\mathscr{F}'$.
Hence, we have
\begin{fact}
  \label{obs:R-F-subsets}
  Let $\mathscr{R}'\subseteq\mathscr{R}$ and
  $\mathscr{F}'\subseteq\mathscr{F}$ for a consistent pair of triple sets
  $(\mathscr{R},\mathscr{F})$. Then $(\mathscr{R}',\mathscr{F}')$ is
  consistent.
\end{fact}

We summarize two characterizations of BMGs given in 
\cite[Thm.~15]{BMG-corrigendum} and
\cite[Lemma~3.4 and Thm.~3.5]{Schaller:20y} in the following
\begin{proposition}\label{prop:BMG-charac}
  Let $(\G,\sigma)$ be a properly colored digraph with vertex set $L$.
  Then the following three statements are equivalent:
  \begin{enumerate}
    \item $(\G,\sigma)$ is a BMG.
    \item $\mathscr{R}(\G,\sigma)$ is consistent and
    $\G(\Aho(\mathscr{R}(\G,\sigma),L), \sigma) = (\G,\sigma)$.
    \item $(\G,\sigma)$ is sf-colored and
    $(\mathscr{R}(\G,\sigma),\mathscr{F}(\G,\sigma))$ is consistent.
  \end{enumerate}
  In this case, $(\Aho(\mathscr{R}(\G,\sigma),L),\sigma)$ is the unique
  LRT for $(\G,\sigma)$, and a leaf-colored tree
  $(T,\sigma)$ on $L$ explains $(\G,\sigma)$ if and only if it agrees with
  $(\mathscr{R}(\G,\sigma), \mathscr{F}(\G,\sigma))$.
\end{proposition}

\begin{figure}[htb]
  \centering
  \includegraphics[width=0.85\textwidth]{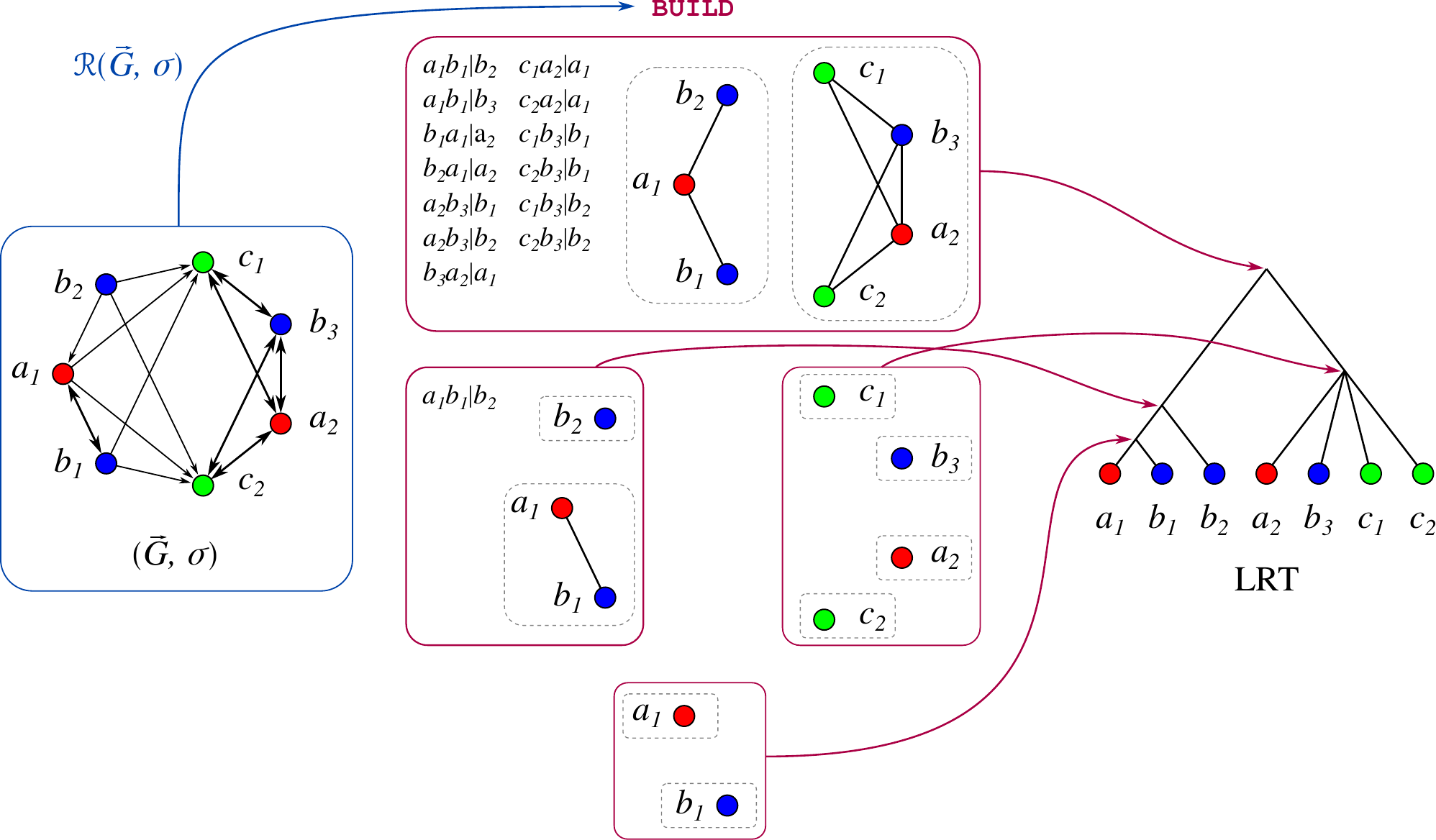}
  \caption[]{Construction of the LRT for a BMG $(\G,\sigma)$ using the
    \texttt{BUILD} algorithm. Each recursion step (pink boxes) corresponds
    to a vertex of the resulting tree (trivial steps on single vertices are
    omitted in the drawing).  The algorithm recurses on the connected
    components (gray dashed boxes) of the Aho graphs and the corresponding
    subsets of triples.}
  \label{fig:bmg-aho-example}
\end{figure}
Prop.~\ref{prop:BMG-charac} states that the set of informative triples
$\mathscr{R}(\G,\sigma)$ of a BMG $(\G,\sigma)$ is consistent. Therefore,
it can be used to construct its LRT by means of the \texttt{BUILD}
algorithm, see Fig.~\ref{fig:bmg-aho-example} for an example.
\begin{figure}[htb]
  \centering
  \includegraphics[width=0.85\textwidth]{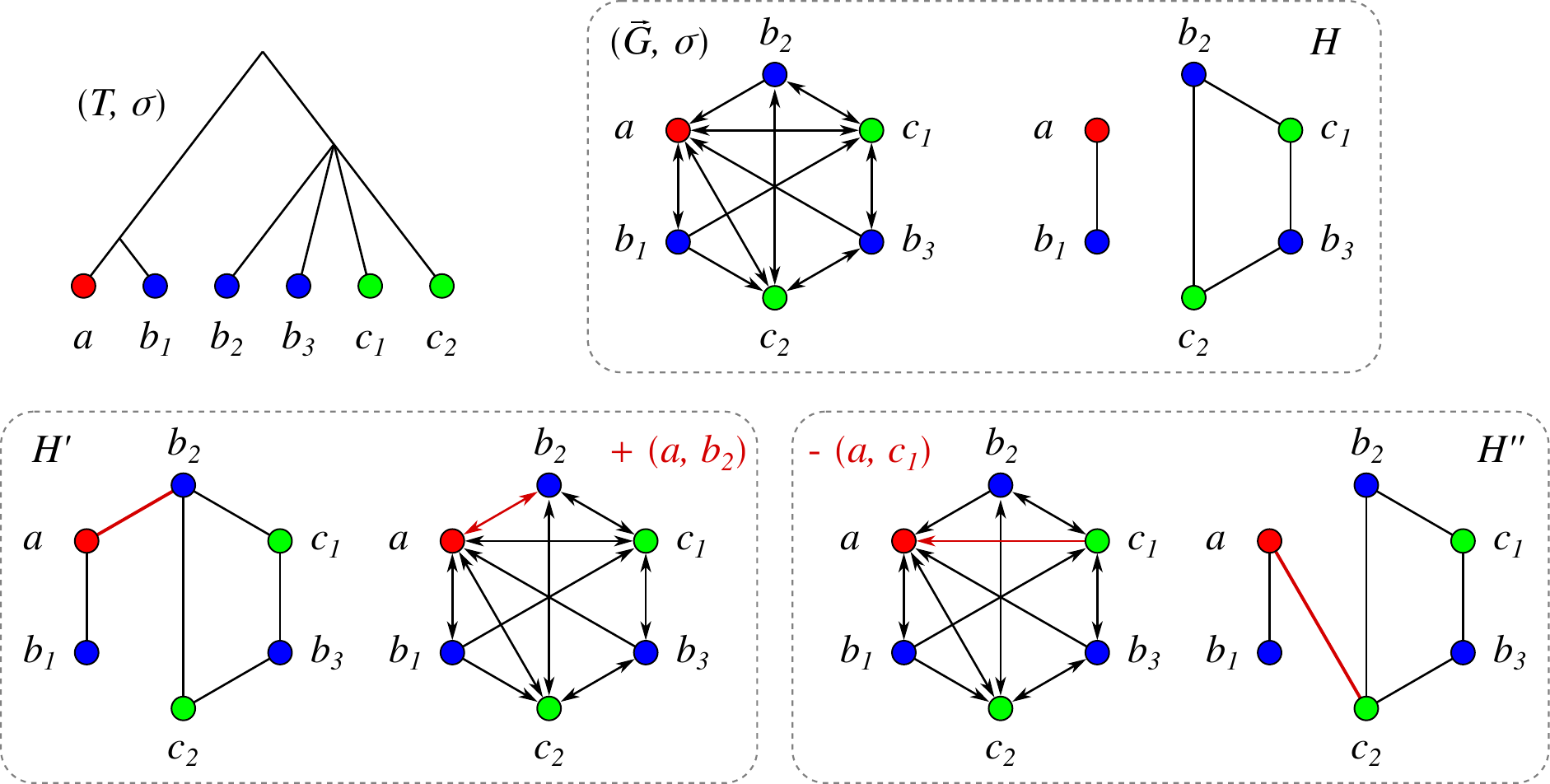}
  \caption[]{Both arc insertions and deletions into a BMG
    $(\G=(V,E),\sigma)$ can introduce inconsistencies into the set of
    informative triples.  Top row: Leaf-colored tree $(T,\sigma)$ explaining
    the BMG $(\G,\sigma)$.  Its set of informative triples is
    $\mathscr{R}(\G,\sigma)=\{ab_1|b_2,\, ab_1|b_3,\, c_1b_2|b_1,\,
    c_1b_3|b_1,\, c_2b_2|b_1,\, c_2b_3|b_1\}$ giving the Aho graph
    $H=[\mathscr{R}(\G,\sigma), V]$.  Bottom left: Insertion of the arc
    $(a, b_2)$ creates a new informative triple $ab_2|b_3$ ($ab_1|b_2$ gets
    lost) resulting in a connected Aho graph $H'$.  Bottom right: Deletion
    of the arc $(a, c_1)$ creates a new triple $ac_2|c_1$ resulting in a
    connected Aho graph $H''$.}
  \label{fig:introduce-inconsistency}
\end{figure}
It is important to note that both arc insertions and deletions
may lead to creation and loss of informative triples.  In particular, when
starting from a BMG, both types of modifications have the potential to make
the triple set inconsistent as the example in
Fig.~\ref{fig:introduce-inconsistency} shows.  This is indeed often the
case even for moderate disturbances of a BMG as we shall see in
Sec.~\ref{sec:heur-computational-exp}.

We expect that empirically estimated best match relations will typically
contain errors that correspond to both arc insertions and deletions w.r.t.\
the unknown underlying ``true'' best match graph.  This motivates the
problem of editing a given vertex-colored digraph to a BMG:

\begin{problem}[\PROBLEM{$\ell$-BMG Editing}]\ \\
  \label{prblm:ell-bmg-edit}
  \begin{tabular}{ll}
    \emph{Input:}    & A properly $\ell$-colored digraph $(\G =(V,E),\sigma)$
    and an integer $k$.\\
    \emph{Question:} & Is there a subset $F\subseteq V\times V \setminus
    \{(v,v)\mid v\in V\}$ such
    that $|F|\leq k$ and\\ & $(\G\symdiff F,\sigma)$
    is an $\ell$-BMG?
  \end{tabular}
\end{problem}
Natural variants are \PROBLEM{$\ell$-BMG Completion} and
\PROBLEM{$\ell$-BMG Deletion} where $\G\symdiff F$ is replaced by
$\G+F$ and $\G-F$, respectively, i.e., only addition or
deletion of arcs is allowed. Both \PROBLEM{$\ell$-BMG Editing} and its
variants are NP-complete \cite{Schaller:20y}.

The heuristic algorithms considered in this contribution can be thought of
as maps $\mathbb{A}$ on the set of finite vertex-colored digraphs such that
$\mathbb{A}(\G,\sigma)$ is a BMG for every vertex-colored input graph
$(\G,\sigma)$.  In particular, the following property of such algorithms is
desirable:
\begin{definition}
  A (BMG-editing) algorithm is \emph{consistent} if
  $\mathbb{A}(\G,\sigma)=(\G,\sigma)$ whenever $(\G,\sigma)$ is a BMG.
\end{definition}

\section{A simple, triple-based heuristic}
\label{sect:simple}

The triple-based characterization summarized by Prop.~\ref{prop:BMG-charac}
suggests a simple heuristic for BMG editing that relies on replacing the
consistency checks for triple sets by the extraction of maximal sets of
consistent triples (see Alg.~\ref{alg:simple}). Both \PROBLEM{MaxRTC}, the
problem of extracting from a given set $\mathscr{R}$ of rooted triples a
maximum-size consistent subset, and \PROBLEM{MinRTI}, the problem of
finding a minimum-size subset $\mathscr{I}$ such that
$\mathscr{R}\setminus\mathscr{I}$ is consistent, are NP-hard
\cite{Jansson:01}.  Furthermore, \PROBLEM{MaxRTC} is APX-hard and
\PROBLEM{MinRTI} is $\Omega(\ln n)$-inapproximable
\cite{Byrka:10}. However, because of their practical importance in
phylogenetics, a large number of practically useful heuristics have been
devised, see e.g.\ \cite{Gasieniec:99,Wu:04,Tazehkand:13}.

\begin{algorithm}
  \caption{Simple Heuristic for BMG editing.}
  \label{alg:simple}
  \KwIn{Properly colored digraph $(\G,\sigma)$}
  \KwOut{BMG $(\G^*,\sigma)$}
  $\mathscr{R}^*\leftarrow
  \PROBLEM{MaxRTC}(\mathscr{R}(\G,\sigma))$\\
  \Return{$\G(\Aho(\mathscr{R}^*, V(\G)),\sigma)$}
\end{algorithm}

As a consequence of Prop.~\ref{prop:BMG-charac}, Alg.~\ref{alg:simple}
is consistent, i.e., $(\G^*,\sigma)=(\G,\sigma)$ if and only if the input
graph $(G,\sigma)$ is a BMG, if a consistent heuristic is
employed to solve \PROBLEM{MaxRTG}/\PROBLEM{MinRTI}, i.e., if consistent
triple sets remain unchanged by the method approximating
\PROBLEM{MaxRTG}/\PROBLEM{MinRTI}.

\begin{figure}[htb]
  \centering
  \includegraphics[width=0.65\textwidth]{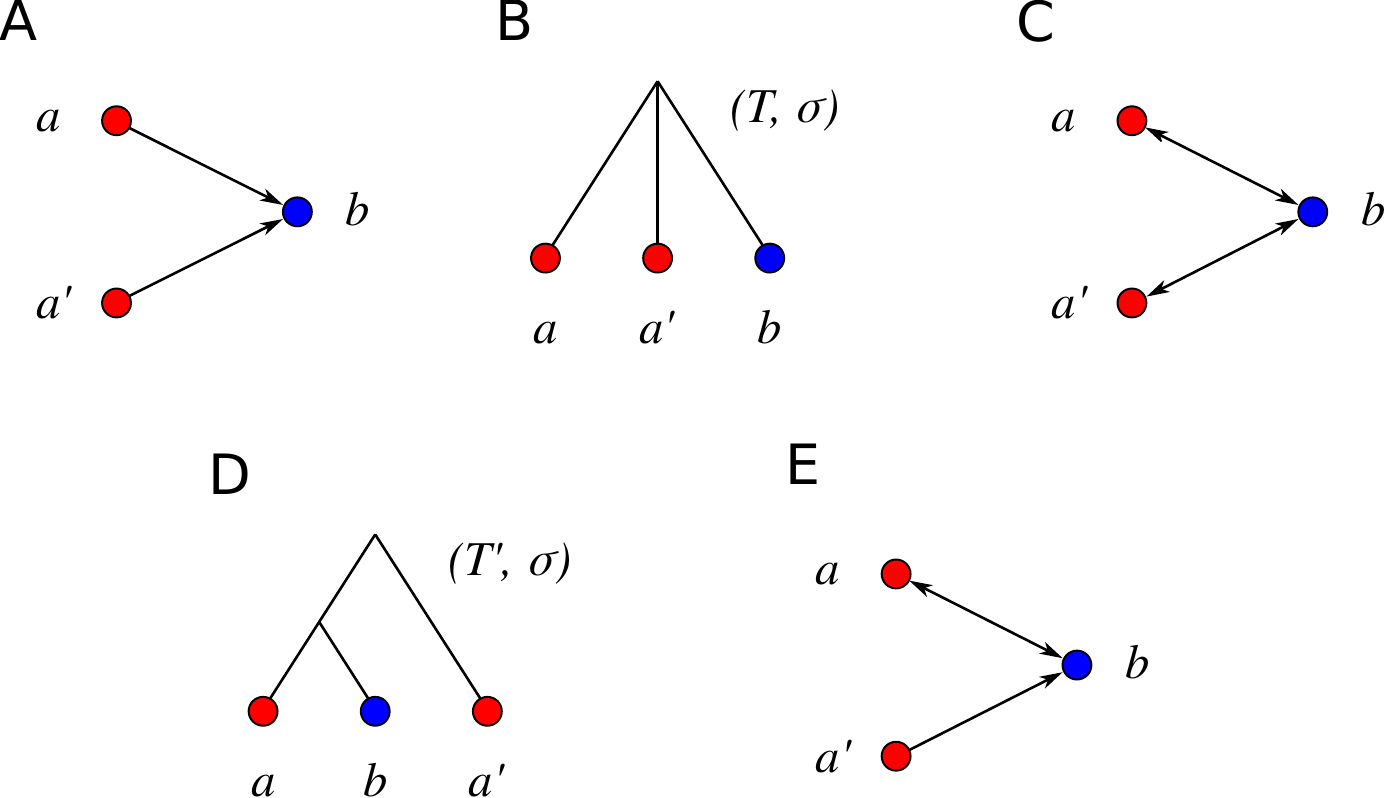}
  \caption[]{Example for a graph (A) where Alg.~\ref{alg:simple} does not
    lead to an optimal BMG editing.  The set $\mathscr{R}(\G,\sigma)$ is
    empty and thus consistent.  (B) The tree
    $T=\Aho(\mathscr{R}(\G,\sigma), V(\G))$ and (C) its corresponding BMG. The 
    two arcs $(b,a)$ and $(b,a')$ have been inserted.  (D) A tree $(T',\sigma)$
    and (E) its corresponding BMG $\G(T',\sigma)$ in which only the arc
    $(b,a)$ has been inserted.}
  \label{fig:heuristic_not_optimal}
\end{figure}

The heuristic Alg.~\ref{alg:simple} is not always optimal, even if
\PROBLEM{MaxRTC}/\PROBLEM{MinRTI} is solved optimally.
Fig.~\ref{fig:heuristic_not_optimal} shows an unconnected 2-colored graph
$(\G,\sigma)$ on three vertices that is not a BMG and does not contain
informative triples. The BMG $(\G^*,\sigma)$ produced by
Alg.~\ref{alg:simple} introduces two arcs into $(\G,\sigma)$. However,
$(\G,\sigma)$ can also be edited to a BMG by inserting only one arc.

A simple improvement is to start by enforcing obvious arcs: If $v$ is the
only vertex with color $\sigma(v)$, then by definition there must be an arc
$(x,v)$ for every vertex $x\ne v$. The computation then starts from the
sets of informative triples of the modified graph. We shall see below that
these are the only arcs that can safely be added to $\G$ without other
additional knowledge or constraints (cf.\ Thm.~\ref{thm:trivial-unsat-rel}
below).

\section{Locally optimal splits}
\label{sect:UR}

Finding an optimal BMG editing of a digraph $(\G=(V,E),\sigma)$ is equivalent 
to finding a tree $(T,\sigma)$ on $V$ that minimizes the cardinality of
\begin{equation}
  \begin{split}
    U(\G,T) \coloneqq \{ (x,y)\in V\times V \mid\ &
    (x,y)\in E \textrm{ and } (x,y)\notin E(\G(T,\sigma))
    \text{, or} \\
    & (x,y)\notin E \textrm{ and } (x,y)\in E(\G(T,\sigma)) \}.
  \end{split}
\end{equation}
Clearly, $U(\G,T) =\emptyset$ implies that $(\G,\sigma) = \G(T,\sigma)$ is
a BMG.  However, finding a tree $(T,\sigma)$ that minimizes $|U(\G,T)|$ is
intractable (unless $P=NP$) since \PROBLEM{$\ell$-BMG Editing},
Problem~\ref{prblm:ell-bmg-edit} above, is NP-complete
\cite{Schaller:20y}.

We may ask, nevertheless, if trees $(T,\sigma)$ on $V$ contain information
about arcs and non-arcs in $(\G,\sigma)$ that are ``unambiguously false''
in the sense that they are contained in every edit set that converts
$(\G,\sigma)$ into a BMG. Denote by $\mathscr{T}_V$ the set of all
phylogenetic trees on $V$. The set of these ``unambiguously false''
\hbox{(non-)}arcs can then be expressed as
\begin{equation}
  \label{eq:U-unambiguous}
  U^*(\G)\coloneqq \bigcap_{T\in \mathscr{T}_V} U(\G,T).
\end{equation}
Since there are in general exponentially many trees on $V$ and thus, the
problem of determining $U^*(\G)$ seems to be quite challenging on a first
glance. We shall see in Thm.~\ref{thm:trivial-unsat-rel}, however, that
$U^*(\G)$ can be computed efficiently. We start with a conceptually
simpler construction.

\begin{definition}
  \label{def:unsat-relations}
  Let $(\G=(V,E),\sigma)$ be a properly vertex-colored digraph and
  $\mathscr{V}$ a partition of $V$ with $|\mathscr{V}|\ge2$.  Moreover, let
  $\mathscr{T}(\mathscr{V})$ be the set of trees $T$ on $V$ that satisfy
  $\mathscr{V} = \{L(T(v)) \mid v\in\child_{T}(\rho_T) \}$.  The set of
  \emph{unsatisfiable relations} (\ur), denoted by $U(\G,\mathscr{V})$, is 
  defined as
  \begin{equation}
    \label{eq:U(G,V)}
    U(\G,\mathscr{V}) \coloneqq \bigcap_{T\in \mathscr{T}(\mathscr{V})} 
    U(\G,T).
  \end{equation}
  The associated \ur-cost is $c(\G,\mathscr{V})\coloneqq |U(\G,\mathscr{V})|$.
\end{definition}

The set of (phylogenetic) trees $\mathscr{T}(\mathscr{V})$ is non-empty
since $|\mathscr{V}|\ge2$ in Def.~\ref{def:unsat-relations}.  Moreover, by
construction, $(x,y) \in U(\G,\mathscr{V})$ if and only if
\begin{align*}
  & (x,y)\in E \textrm{ and } (x,y)\notin E(\G(T,\sigma)) \textrm{ for all } 
  T\in \mathscr{T}(\mathscr{V}) \textrm{, or}\\
  & (x,y)\notin E \textrm{ and } (x,y)\in E(\G(T,\sigma)) \textrm{ for all } 
  T\in \mathscr{T}(\mathscr{V}).
\end{align*}

Intriguingly, the set $U(\G,\mathscr{V})$, and thus the \ur-cost
$c(\G,\mathscr{V})$, can be computed in polynomial time without any
explicit knowledge of the possible trees to determine the set
$U(\G,\mathscr{V})$. To this end, we define the three sets
\begin{align*}
  U_1(\G,\mathscr{V}) &= \bigcup_{V_{i}\in\mathscr{V}} \{(x,y) \mid
  (x,y)\in E,\; x\in V_{i},\; 
  y\in V\setminus V_{i},\;
  \sigma(y)\in \sigma(V_{i})\},\\
  U_2(\G,\mathscr{V}) &= \bigcup_{V_{i}\in\mathscr{V}} \{(x,y) \mid
  (x,y)\notin E,\; x\in V_{i},\; 
  y\in V\setminus V_{i},\;
  \sigma(y)\notin \sigma(V_{i})\}, \textrm{ and}\\
  U_3(\G,\mathscr{V}) &= \bigcup_{V_{i}\in\mathscr{V}} \{(x,y) \mid
  (x,y)\notin E,\; \textrm{distinct }x,y\in V_{i},\;
  V_{i}[\sigma(y)]=\{y\}\}.
\end{align*}
\begin{lemma}
  \label{lem:unsat-rel-charac}
  Let $(\G=(V,E),\sigma)$ be a properly vertex-colored digraph and let
  $\mathscr{V}=\{V_1,\dots,V_k\}$ be a partition of $V$ with
  $|\mathscr{V}|=k\ge2$.  Then
  \begin{equation*}
    U(\G,\mathscr{V}) = U_1(\G,\mathscr{V}) \;\cupdot\;
    U_2(\G,\mathscr{V}) \;\cupdot\; U_3(\G,\mathscr{V})\,.
  \end{equation*}
\end{lemma}

\begin{figure}[htb]
  \centering
  \includegraphics[width=0.75\textwidth]{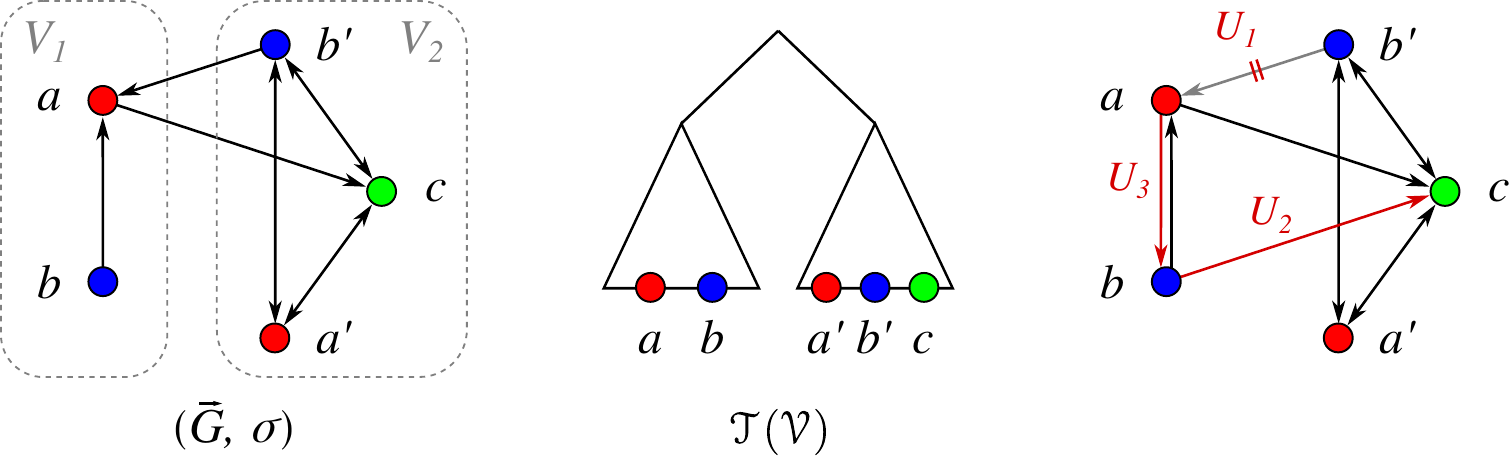}
  \caption[]{Example for unsatisfiable relations $U(\G,\mathscr{V})$ of a
    vertex-colored digraph $(\G=(V,E),\sigma)$ w.r.t.\ a partition
    $\mathscr{V}=\{V_1,V_2\}$ (indicated by the gray boxes).  In the
    middle, the set of trees $\mathscr{T}(\mathscr{V})$ is illustrated,
    i.e., the triangles represent all possible phylogenetic trees on the
    respective subset of leaves.  On the right, the arc modifications
    implied by $\mathscr{V}$ (i.e., $U(\G,\mathscr{V})$) are illustrated
    where $U_1$, $U_2$, and $U_3$ indicate the type according to
    Lemma~\ref{lem:unsat-rel-charac}.}
  \label{fig:unsat-rel-example}
\end{figure}

The proof of Lemma~\ref{lem:unsat-rel-charac} relates the possible cases
between $\mathscr{V}$ and the tree set $\mathscr{T}(\mathscr{V})$ in a
straightforward manner. Since it is rather lengthy it is relegated to
Appendix.  Fig.~\ref{fig:unsat-rel-example} gives examples for all three
types of unsatisfiable relations, i.e., for $U_1(\G,\mathscr{V})$,
$U_2(\G,\mathscr{V})$, and $U_3(\G,\mathscr{V})$.  In particular, we have
$(b', a)\in U_1(\G,\mathscr{V})$ since it is an arc in $\G$ but $V_2$
contains another red vertex $a'$. Moreover, $(b,c)\in U_2(\G,\mathscr{V})$
since it is not an arc in $\G$ but $V_1$ does not contain another green
vertex.  Finally, we have $(a,b)\in U_3(\G,\mathscr{V})$ since it is not an
arc in $\G$ but $b$ is the only blue vertex in $V_1$.  In the example, the
graph $(\G\symdiff U(\G,\mathscr{V}))$ is already a BMG which, however, is
not true in general.

\begin{corollary}
  \label{cor:U-polytime}
  The set $U(\G,\mathscr{V})$ can be computed in quadratic time.
\end{corollary}
\begin{proof}
  We first compute all numbers $n_{i,A}$ of vertices in $V_i$ with a given
  color $A$. This can be done in $O(|V|)$ if we do not explicitly store the
  zero-entries. Now, $\sigma(y)\in\sigma(V_i)$, i.e.\ $n_{i,\sigma(y)}>0$,
  can be checked in constant time, and thus, it can also be decided in
  constant time whether or not a pair $(x,y)$ is contained in
  $U_1(\G,\mathscr{V})$ or $U_2(\G,\mathscr{V})$.  Since, given $y\in V_i$,
  the condition $V_i[\sigma(y)]=\{y\}$ is equivalent to
  $n_{i,\sigma(y)}=1$, membership in $U_3(\G,\mathscr{V})$ can also be
  decided in constant time.  Checking all ordered pairs $x,y\in V$ thus
  requires a total effort of $O(|V|^2)$.
  \qed
\end{proof}

\begin{algorithm}[t]
  \caption{General BMG editing via locally optimal steps.\newline The parts
    highlighted in color produce a tree $T$ explaining the edited graph
    $(\vec G^*,\sigma)$. If the tree is not needed, these steps can be
    omitted. The method for choosing the partition $\mathscr{V}$ (framed
    box) determines different variants of the algorithm.}
  \label{alg:general-local-optimal}
  
  \newcommand{\treeline}[1]{\begingroup\color{blue}#1\endgroup}
  
  \DontPrintSemicolon
  \SetKwFunction{FRecurs}{void FnRecursive}%
  \SetKwFunction{FRecurs}{Edit}
  \SetKwProg{Fn}{Function}{}{}
  
  \KwIn{Properly colored digraph $(\G=(V, E),\sigma)$}
  \KwOut{BMG $(\G^*,\sigma)$.}
  
  \BlankLine
  
  initialize $(\G^*,\sigma) \leftarrow (\G,\sigma)$ \label{line:init-G*}\;
  
  \BlankLine
  \Fn{\FRecurs{$V'$}}{
    
    \uIf{$|V'|>1$}{
      $\mathscr{V} \leftarrow$
      \framebox{suitably chosen partition of $V'$ with $|\mathscr{V}|\ge2$}
      \label{line:min-cost}\;
      $\G^{*} \leftarrow \G^{*} \, \triangle \,
      U(\G^{*}[V'],\mathscr{V})$ \label{line:apply-edits} \;
      \treeline{create a tree $T'$ with root $\rho'$\;}
      \ForEach{$V_i \in \mathscr{V}$}{
        \treeline{\SetNlSty{bfseries}{\color{black}}{} attach the tree} 
        \FRecurs{$V_i$} \treeline{to $\rho'$}\;
      }
      \treeline{\Return $T'$\;}
    }
    \treeline{
      \Else{
        \Return a tree with the single element in $V'$ as root\;
    }}
    
  }
  \BlankLine
  
  \treeline{\SetNlSty{bfseries}{\color{black}}{} $T \leftarrow$} 
  \FRecurs{$V(\G)$}\;
  \Return  $(\G^*,\sigma)$ \treeline{ and $T$} \label{line:return-edited}\;
\end{algorithm}

Our discussion so far suggests a recursive top-down approach, made precise
in Alg.~\ref{alg:general-local-optimal}.  In each step, one determines a
``suitably chosen'' partition $\mathscr{V}$ and then recurses on the
subgraphs of the edited graph $\G^* \triangle
U(\G^*[V'],\mathscr{V})$. More details on such suitable partitions
$\mathscr{V}$ will be given in Thm.~\ref{thm:algo-consistent} below.  The
parts in the algorithm highlighted in color can be omitted.  They are
useful, however, if one is also interested in a tree $(T,\sigma)$ that
explains the editing result $(\G^*,\sigma)$ and to show that
$(\G^*,\sigma)$ is indeed a BMG (see below).
Alg.~\ref{alg:general-local-optimal} is designed to accumulate the
edit sets in each step, Line~\ref{line:apply-edits}. In particular, the
total edit cost and the scores $c(\G^*[V'],\mathscr{V})$ are closely tied
together, which follows from the following result:
\begin{lemma}
  All edit sets $U(\G^*[V'],\mathscr{V})$ constructed in
  Alg.~\ref{alg:general-local-optimal} are pairwise disjoint.
  \label{lem:independent-U}
\end{lemma}
The proof of Lemma~\ref{lem:independent-U} and a technical result on which
it relies can be found in the Appendix.  As an immediate consequence of
Lemma~\ref{lem:independent-U}, we have
\begin{corollary}
  The edit cost of Alg.~\ref{alg:general-local-optimal} is the sum of
  the \ur-costs $c(\G^*[V'],\mathscr{V})$ in each recursion step.
  \label{cor:sum-c}
\end{corollary}

It is important to note that the edits $U(\G^*[V'],\mathscr{V})$ must be
applied immediately in each step (cf.\ Line~\ref{line:apply-edits} in
Alg.~\ref{alg:general-local-optimal}). In particular,
Lemma~\ref{lem:independent-U} and Cor.~\ref{cor:sum-c} pertain to the
partitioning of the edited graph $\G^*$, not to the original graph $\G$.

\begin{theorem}
  \label{thm:algo-tree}
  Every pair of edited graph $(\G^*,\sigma)$ and tree $T$ produced as
  output by Alg.~\ref{alg:general-local-optimal} satisfies
  $(\G^*,\sigma)=\G(T,\sigma)$.  In particular, $(\G^*,\sigma)$ is a BMG.
\end{theorem}
\begin{proof}
  By construction, the tree $T$ is phylogenetic and there is a
  one-to-one correspondence between the vertices $u\in V(T)$ and the
  recursion steps, which operate on the sets $V'=L(T(u))$.  If
  $|V'|\ge 2$ (or, equivalently, $u$ is an inner vertex of $T$), we
  furthermore have $\mathscr{V}=\{L(T(v)) \mid v\in\child_{T}(u)\}$ for the
  partition $\mathscr{V}$ of $V'$ chosen in that recursion step.  In the
  following, we denote by $(\G^*,\sigma)$ the graph during the editing
  process, and by $(\G,\sigma)$ the input graph, i.e., as in
  Alg.~\ref{alg:general-local-optimal}. For brevity, we write
  $E^*$ for the arc set of the final edited graph and
  $E^T\coloneqq E(\G(T,\sigma))$.
  
  Let us assume, for contradiction, that there exists (a)
  $(x,y)\in E^*\setminus E^T\ne\emptyset$, or (b)
  $(x,y)\in E^T\setminus E^*\ne\emptyset$.  In either case, we set
  $u\coloneqq \lca_T(x,y)$ and consider the recursion step on
  $V'\coloneqq L(T(u))$ with the corresponding partition
  $\mathscr{V}\coloneqq\{L(T(v)) \mid v\in\child_{T}(u)\}$ chosen for $V'$.
  Note that $x\ne y$, and thus $u\in V^0(T)$.  Moreover, let $v_x$ be the
  child of $u$ such that $x\preceq_{T} v_x$, and
  $V_x\coloneqq L(T(v_x))\in \mathscr{V}$.
  
  Case~(a): $(x,y)\in E^*\setminus E^T\ne\emptyset$.  \newline Since
  $(x,y)\notin E^T$ and by the definition of best matches, there must be a
  vertex $y'\in V_x$ of color $\sigma(y)$ such that
  $\lca_T(x,y')\prec_T\lca_{T}(x,y)=u$, and thus
  $\sigma(y)\in \sigma(V_x)$.  Moreover, we have $V_x\in \mathscr{V}$,
  $x\in V_x$ and $y\in V'\setminus V_x$.  Two subcases need to be
  considered, depending on whether or not $(x,y)$ is an arc in $\G^*$ at
  the beginning of the recursion step.  In the first case, the
  arguments above imply that $(x,y)\in U_1(\G^*[V'], \mathscr{V})$,
  and thus, $(x,y)\in U(\G^*[V'], \mathscr{V})$ by
  Lemma~\ref{lem:unsat-rel-charac}. Hence, we delete the arc $(x,y)$ in
  this step.  In the second case, it is an easy task to verify that
  none of the definitions of $U_1(\G^*[V'], \mathscr{V})$,
  $U_2(\G^*[V'], \mathscr{V})$, and $U_3(\G^*[V'], \mathscr{V})$ matches
  for $(x,y)$.  Since this step is clearly the last one in the
  recursion hierarchy that can affect the (non-)arc $(x,y)$, it follows
  for both subcases that $(x,y)\notin E^*$; a contradiction.
  
  Case~(b): $(x,y)\in E^T\setminus E^*\ne\emptyset$. \newline Since
  $(x,y)\in E^T$ and by the definition of best matches, there cannot be a
  vertex $y'\in V_x$ of color $\sigma(y)$ such that
  $\lca_T(x,y')\prec_T\lca_{T}(x,y)=u$, and thus
  $\sigma(y)\notin \sigma(V_x)$.  Moreover, we have $V_x\in \mathscr{V}$,
  $x\in V_x$ and $y\in V'\setminus V_x$.  Again, two subcases need to
  be distinguished depending on whether or not $(x,y)$ is an arc in
  $\G^*$ at the beginning of the recursion step. In the first case,
  the arguments above make it easy to verify that none of the
  definitions of $U_1(\G^*[V'], \mathscr{V})$,
  $U_2(\G^*[V'], \mathscr{V})$, and $U_3(\G^*[V'], \mathscr{V})$ matches
  for $(x,y)$.  In the second case, we obtain
  $(x,y)\in U_2(\G^*[V'], \mathscr{V})$, and thus,
  $(x,y)\in U(\G^*[V'], \mathscr{V})$ by
  Lemma~\ref{lem:unsat-rel-charac}. Hence, we insert the arc $(x,y)$ in
  this step.  As before, the (non-)arc $(x,y)$ remains unaffected in any
  deeper recursion step. Therefore, we have $(x,y)\in E^*$ in both
  subcases; a contradiction.
  
  Finally, $(\G^*,\sigma)=\G(T,\sigma)$ immediately implies that 
  $(\G^*,\sigma)$ is a BMG.
  \qed
\end{proof}

\begin{figure}[htb]
  \centering
  \includegraphics[width=0.85\textwidth]{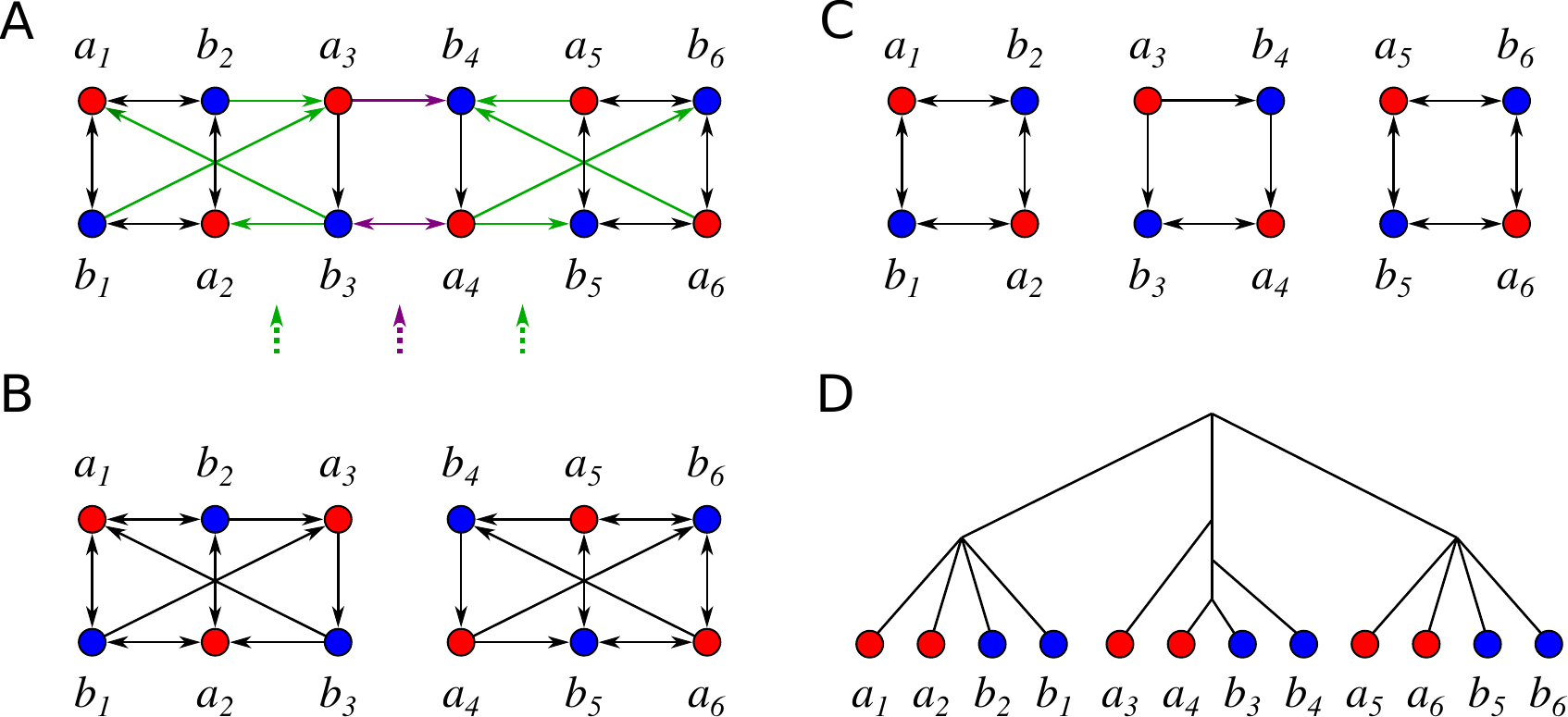}
  \caption[]{(A) Example for a colored digraph $(\G,\sigma)$ in which the
    ``locally'' optimal (first) split does not result in a global optimal
    BMG editing.  The minimal \ur-cost equals $3$ and is
    attained only for the partition
    $\mathscr{V}=\{ \{a_1,a_2,a_3,b_1,b_2,b_3\},
    \{a_4,a_5,a_6,b_4,b_5,b_6\}\}$, which was verified by full enumeration
    of all partitions and Lemma~\ref{lem:unsat-rel-charac}.  For this
    partition, $U(\G,\mathscr{V})$ comprises the three purple arcs.  (B)
    The two (isomorphic) induced subgraphs obtained by applying the locally
    optimal partition $\mathscr{V}$. Each of them has a (global) optimal
    BMG editing cost of $4$.  Therefore, the overall symmetric difference
    of an edited graph (using the initial split $\mathscr{V}$ as specified)
    comprises at least $c(\G,\mathscr{V})+2\cdot 4=11$ arcs.  (C) An
    optimal editing removes the 8 green arcs and results in a digraph that
    is explained by the tree in (D). The optimality of this solution was
    verified using an implementation of the ILP formulation for BMG editing
    given in \cite{Schaller:20y}.}
  \label{fig:global-vs-local-optimum}
\end{figure}

Cor.~\ref{cor:sum-c} suggests a greedy-like ``local'' approach. In each
step, the partition $\mathscr{V}$ is chosen to minimize the score
$c(\G, \mathscr{V})$ in Line~\ref{line:min-cost}. The example in
Fig.~\ref{fig:global-vs-local-optimum} shows, however, that the greedy-like
choice of $\mathscr{V}$ does not necessarily yield a globally optimal
edit set.

In order to identify arcs that must be contained in every edit set, we first
clarify the relationship between the partitions $\mathfrak{P}_{\ge 2}$ on
$V$ and the partitions defined by the phylogenetic trees on $V$.
\begin{lemma}
  \label{lem:tree-sets-equal}
  Let $V$ be a set with $|V|\ge 2$. Let $\mathfrak{P}_{\ge 2}$ be the set
  of all partitions $\mathscr{V}$ of $V$ with $|\mathscr{V}|\ge2$.  Then
  the set $\mathscr{T}_V$ of all phylogenetic trees with leaf set $V$
  satisfies
  $\mathscr{T}_V= \bigcup_{\mathscr{V}\in\mathfrak{P}_{\ge 2}}
  \mathscr{T}(\mathscr{V})$.
\end{lemma}
\begin{proof}
  For every $\mathscr{V}\in\mathfrak{P}_{\ge 2}$,
  $\mathscr{T}(\mathscr{V})$ is a set of phylogenetic trees on $V$.  Hence,
  we conclude
  $\bigcup_{\mathscr{V}\in\mathfrak{P}_{\ge 2}} \mathscr{T}(\mathscr{V})
  \subseteq \mathscr{T}_V$. Conversely, assume that $T\in \mathscr{T}_V$.
  Since $T$ (with root $\rho_T$) is a phylogenetic tree and has at least
  two leaves, we have $|\child_{T}(\rho_T)|\ge 2$. Together with
  $L(T(\rho_T))=L(T)=V$, this implies
  $\mathscr{V}^*\coloneqq\{L(T(v)) \mid v\in \child_{T}(\rho_T)\} \in
  \mathfrak{P}_{\ge 2}$.  In particular, $T$ satisfies
  $T\in\mathscr{T}(\mathscr{V}^*)$ for some
  $\mathscr{V}^* \in \mathfrak{P}_{\ge 2}$, and is therefore contained in
  $\bigcup_{\mathscr{V}\in\mathfrak{P}_{\ge 2}} \mathscr{T}(\mathscr{V})$.
  \qed
\end{proof}

Using Lemma~\ref{lem:tree-sets-equal} and given that $|V|\ge 2$, we can 
express the set of relations
that are unsatisfiable for every partition as follows
\begin{equation}
  \begin{split}
    \bigcap_{\mathscr{V}\in\mathfrak{P}_{\ge2}} U(\G,\mathscr{V})
    &= \bigcap_{\mathscr{V}\in\mathfrak{P}_{\ge2}}
    \left( \bigcap_{T\in \mathscr{T}(\mathscr{V})} U(\G,T) \right) 
    = \bigcap_{T\in\bigcup_{\mathscr{V}\in\mathfrak{P}_{\ge 2}}
      \mathscr{T}(\mathscr{V})} U(\G,T)\\
    &= \bigcap_{T\in\mathscr{T}_V} U(\G,T) = U^*(\G)\;,
  \end{split}
  \label{eq:U*U-P}
\end{equation}
i.e., it coincides with the set of relations that are unsatisfiable for
every phylogenetic tree, and thus part of every edit set. 
Note that $U^*(\G)$ is trivially empty if $|V|<2$.
We next show that
$U^*(\G)$ can be computed without considering the partitions of $V$
explicitly.

\begin{theorem}
  \label{thm:trivial-unsat-rel}
  Let $(\G=(V,E),\sigma)$ be a properly vertex-colored digraph with
  $|V|\ge 2$ then
  \begin{equation}
    U^*(\G) = \left\{ (x,y)  \mid 
    (x,y)\notin E,\; x\ne y,\; V[\sigma(y)]=\{y\} \right\}.
    \label{eq:U-equalities}
  \end{equation}
\end{theorem}
\begin{proof}
  First note that $|V|\ge 2$ ensures that
  $\mathfrak{P}_{\ge 2}\ne\emptyset$.  Moreover, since $|\mathscr{V}|\ge2$
  for any $\mathscr{V}\in\mathfrak{P}_{\ge 2}$, the sets
  $\mathscr{T}(\mathscr{V})$ are all non-empty as well. With the
  abbreviation $\hat U(\G)$ for the right-hand side of
  Eq.~(\ref{eq:U-equalities}), we show that
  $\hat U(\G)= \bigcap_{\mathscr{V}\in\mathfrak{P}_{\ge2}}
  U(\G,\mathscr{V})$ which by Eq.~(\ref{eq:U*U-P}) equals $U^*(\G)$.
  
  Suppose first that $(x,y)\in \hat U(\G)$. Then $x\ne y$ and
  $V[\sigma(y)]=\{y\}$ imply that $\sigma(x)\ne\sigma(y)$.  This together
  with the facts that (i) $y$ is the only vertex of its color in $V$, and
  (ii) $L(T)=V$ for each $T\in \mathscr{T}(\mathscr{V})$ and any
  $\mathscr{V}\in\mathfrak{P}_{\ge 2}$ implies that $y$ is a best match of
  $x$ in every such tree $T$, i.e.\ $(x,y)\in E(\G(T,\sigma))$.  Since in
  addition $(x,y)\notin E$ by assumption, we conclude that
  $(x,y)\in U^{*}(\G)$.
  
  Now suppose that $(x,y)\in U^{*}(\G)$. Observe that
  $\sigma(x)\ne\sigma(y)$ (and thus $x\ne y$) as a consequence of
  Def.~\ref{def:unsat-relations} and the fact that $(\G,\sigma)$ and all
  BMGs are properly colored.  If $V=\{x,y\}$ and thus $\{\{x\},\{y\}\}$ is
  the only partition in $\mathfrak{P}_{\ge 2}$, the corresponding unique
  tree $T$ consists of $x$ and $y$ connected to the root. In this case, we
  clearly have $(x,y)\in E(\G(T,\sigma))$ since $\sigma(x)\ne\sigma(y)$.
  On the other hand, if $\{x,y\}\subsetneq V$, then we can find a partition
  $\mathscr{V}\in \mathfrak{P}_{\ge 2}$ such that $V_{i}=\{x,y\}$ for some
  $V_{i}\in\mathscr{V}$.  In this case, every tree
  $T\in\mathscr{T}(\mathscr{V})$ has a vertex $v_{i}\in\child_{T}(\rho_T)$
  with the leaves $x$ and $y$ as its single two children. Clearly,
  $(x,y)\in E(\G(T,\sigma))$ holds for any such tree.  In summary, there
  always exists a partition $\mathscr{V}\in\mathfrak{P}_{\ge 2}$ such that
  $(x,y)\in E(\G(T,\sigma))$ for some tree $T\in\mathscr{T}(\mathscr{V})$.
  Therefore, by
  $(x,y)\in\bigcap_{\mathscr{V}\in\mathfrak{P}_{\ge 2}} U(\G,\mathscr{V})$
  and Def.~\ref{def:unsat-relations}, we conclude that $(x,y)\notin E$.  In
  order to obtain $(x,y)\in \hat U(\G)$, it remains to show that
  $V[\sigma(y)]=\{y\}$.  Since $(x,y)\notin E$ and
  $(x,y)\in\bigcap_{\mathscr{V}\in\mathfrak{P}_{\ge 2}} U(\G,\mathscr{V})$,
  it must hold that $(x,y)\in E(\G(T,\sigma))$ for all
  $T\in\mathscr{T}(\mathscr{V})$ and all
  $\mathscr{V}\in\mathfrak{P}_{\ge 2}$.  Now assume, for contradiction,
  that there is a vertex $y'\ne y$ of color $\sigma(y')=\sigma(y)$.  Since
  $\sigma(x)\ne\sigma(y)$, the vertices $x,y,y'$ must be pairwise distinct.
  Hence, we can find a partition $\mathscr{V}\in \mathfrak{P}_{\ge 2}$ such
  that $V_{i}=\{x,y'\}$ for some $V_{i}\in\mathscr{V}$.  In this case,
  every tree $T\in\mathscr{T}(\mathscr{V})$ has a vertex
  $v_{i}\in\child_{T}(\rho_T)$ with only the leaves $x$ and $y'$ as its
  children. Clearly, $\lca_T(x,y')=v_{i}\prec_{T}\rho_{T}=\lca_{T}(x,y)$,
  and thus $(x,y)\notin E(\G(T,\sigma))$; a contradiction.  Therefore, we
  conclude that $y$ is the only vertex of its color in $V$, and hence,
  $(x,y)\in \hat U(\G)$. In summary, therefore, we have
  $U^{*}(\G)=\hat U(\G)$.  \qed
\end{proof}
As a consequence of Thm.~\ref{thm:trivial-unsat-rel} and by similar
arguments as in the proof of Cor.~\ref{cor:U-polytime}, we observe
\begin{corollary}
  The set $U^*(\G)$ can be computed in quadratic time.
\end{corollary}
By Thm.~\ref{thm:trivial-unsat-rel}, $U^*(\G)$ contains only non-arcs,
more precisely, missing arcs pointing towards a vertex that is the
only one of its color and thus, by definition, a best match of every
other vertex irrespective of the details of the gene tree. By definition,
furthermore, $U^*(\G)$ is a subset of every edit set for 
$(\G,\sigma)$. We
therefore have the lower bound
\begin{equation}
  |U^*(\G)| \le c(\G,\mathscr{V}) 
\end{equation}
for every $\mathscr{V}\in\mathfrak{P}_{\ge2}$. 

The following result shows that if $(\G,\sigma)$ is a BMG, then a suitable
partition $\mathscr{V}$ can be chosen such that
$c(\G,\mathscr{V})=|U^*(\G)|=0$.

\begin{lemma}
  \label{lem:cost-0-BMG}
  Let $(\G=(V,E),\sigma)$ be a BMG with $|V|\ge 2$ and $\mathscr{V}$ be the
  connected components of the Aho graph $[\mathscr{R}(\G,\sigma), V]$. Then
  the partition $\mathscr{V}$ of $V$ satisfies $|\mathscr{V}|\ge 2$ and
  $c(\G,\mathscr{V})=0$.
\end{lemma}
\begin{proof}
  Since $(\G,\sigma)$ is a BMG, we can apply Prop.~\ref{prop:BMG-charac} to
  conclude that $\mathscr{R}\coloneqq\mathscr{R}(\G,\sigma)$ is consistent
  and that $(T,\sigma)\coloneqq (\Aho(\mathscr{R}, V),\sigma)$ explains
  $(\G,\sigma)$, i.e., $\G(T,\sigma)=(\G,\sigma)$. Hence,
  $U(\G,T)=\emptyset$.  From $|V|\ge 2$ and consistency of $\mathscr{R}$,
  it follows that $[\mathscr{R}, V]$ has at least two connected components
  \cite{Aho:81}, and thus, by construction, $|\mathscr{V}|\ge 2$.
  Moreover, we clearly have $T\in \mathscr{T}(\mathscr{V})$ by the
  construction of $T$ via \texttt{BUILD}. Together with
  $U(\G,T)=\emptyset$, the latter implies $U(\G,\mathscr{V})=\emptyset$,
  and thus $c(\G,\mathscr{V})=0$.  \qed
\end{proof}

\begin{lemma}
  \label{lem:cost-0-subgraphs-BMG}
  Let $(\G=(V,E),\sigma)$ be a BMG, and $\mathscr{V}$ a partition of $V$
  such that $c(\G,\mathscr{V})=0$. Then the induced subgraph
  $(\G[V'],\sigma_{|V'})$ is a BMG for every $V'\in\mathscr{V}$.
\end{lemma}
\begin{proof}
  Set $\mathscr{R}\coloneqq\mathscr{R}(\G,\sigma)$ and
  $\mathscr{F}\coloneqq\mathscr{F}(\G,\sigma)$ for the sets of informative
  and forbidden triples of $(\G,\sigma)$, respectively.  Since
  $(\G,\sigma)$ is a BMG, we can apply Prop.~\ref{prop:BMG-charac} to
  conclude that $(\mathscr{R},\mathscr{F})$ is consistent. Now we
  choose an arbitrary set $V'\in\mathscr{V}$ and set
  $(\G',\sigma')\coloneqq (\G[V'],\sigma_{|V'})$.  By
  Obs.~\ref{obs:R-restriction}, we obtain
  $\mathscr{R}(\G',\sigma')=\mathscr{R}_{|V'}$ and
  $\mathscr{F}(\G',\sigma')=\mathscr{F}_{|V'}$. This together with the fact
  that $\mathscr{R}_{|V'}\subseteq\mathscr{R}$ and
  $\mathscr{F}_{|V'}\subseteq\mathscr{F}$ and Obs.~\ref{obs:R-F-subsets}
  implies that
  $(\mathscr{R}_{|V'},\mathscr{F}_{|V'})=(\mathscr{R}(\G',\sigma'),
  \mathscr{F}(\G',\sigma'))$ is consistent.
  
  By Prop.~\ref{prop:BMG-charac}, it remains to show that $(\G',\sigma')$
  is sf-colored to prove that it is a BMG.  To this end, assume for
  contradiction that there is a vertex $x\in V'$ and a color
  $s\in\sigma(V')$ such that $x$ has no out-neighbor of color
  $s\ne\sigma(x)$ in $V'$.  However, since the color $s$ is contained in
  $\sigma(V)$ and $(\G,\sigma)$ is a BMG, and thus sf-colored, we conclude
  that there must be a vertex $y\in V\setminus V'$ of color $s$ such that
  $(x,y)\in E$.  In summary, we obtain $(x,y)\in E$, $x\in V'$,
  $y\in V\setminus V'$ and $\sigma(y)=s\in\sigma(V')$.  Thus, we have
  $(x,y)\in U_1(\G,\mathscr{V})$.  Hence, Lemma~\ref{lem:unsat-rel-charac}
  implies that $(x,y)\in U(\G,\mathscr{V})$ and, hence,
  $c(\G,\mathscr{V})>0$; a contradiction.  Therefore, $(\G',\sigma')$ must
  be sf-colored, which concludes the proof.  \qed
\end{proof}

Lemma \ref{lem:cost-0-BMG} and \ref{lem:cost-0-subgraphs-BMG} allow us to
choose the partition $\mathscr{V}$ in each step of
Alg.~\ref{alg:general-local-optimal} in such a way that
Alg.~\ref{alg:general-local-optimal} is consistent, i.e., BMGs remain
unchanged.
\begin{theorem}
  \label{thm:algo-consistent}
  Alg.~\ref{alg:general-local-optimal} is consistent if, in each step
  on $V'$ with $|V'|\ge 2$, the partition $\mathscr{V}$ in
  Line~\ref{line:min-cost} is chosen according to one of the following
  rules:
  \begin{enumerate}
    \item $\mathscr{V}$ has minimal \ur-cost among all possible partitions
    $\mathscr{V}'$ of $V'$ with $|\mathscr{V}'|\ge 2$.
    \item If the Aho graph $[\mathscr{R}(\G^*[V'],\sigma_{|V'}),V']$ is
    disconnected with the set of connected components $\mathscr{V}_{\Aho}$,
    and moreover $c(\G^*[V'],\mathscr{V}_{\Aho})=0$, then
    $\mathscr{V}=\mathscr{V}_{\Aho}$.
  \end{enumerate}
\end{theorem}
\begin{proof}
  We have to show that the final edited graph $(\G^*,\sigma)$ returned in
  Line~\ref{line:return-edited} equals the input graph $(\G=(V,E),\sigma)$
  whenever $(\G,\sigma)$ already is a BMG, i.e., nothing is edited.  Thus
  suppose that $(\G,\sigma)$ is a BMG and first consider the top-level
  recursion step on $V$ (where initially $\G^*=\G$ still holds at
  Line~\ref{line:init-G*}).  If $|V|=1$, neither $(\G,\sigma)$ nor
  $(\G^*,\sigma)$ contain any arcs, and thus, the edit cost is trivially
  zero. Now suppose $|V|\ge 2$.  Since $(\G,\sigma)$ is a BMG,
  Lemma~\ref{lem:cost-0-BMG} guarantees the existence of a partition
  $\mathscr{V}$ satisfying $c(\G,\mathscr{V})=0$, in particular, the
  connected components $\mathscr{V}_{\Aho}$ of the Aho graph
  $[\mathscr{R}(\G,\sigma), V]$ form such a partition.  Hence, for both
  rules~(1) and~(2), we choose a partition $\mathscr{V}$ with (minimal)
  \ur-cost $c(\G,\mathscr{V})=0$.  Now,
  Lemma~\ref{lem:cost-0-subgraphs-BMG} implies that the induced subgraph
  $(\G[V'],\sigma_{|V'})$ is a BMG for every $V'\in\mathscr{V}$.  Since we
  recurse on these subgraphs, we can repeat the arguments above along the
  recursion hierarchy to conclude that the \ur-cost
  $c(\G^*[V'],\mathscr{V}')$ vanishes in every recursion step.  By
  Cor.~\ref{cor:sum-c}, the total edit cost of
  Alg.~\ref{alg:general-local-optimal} is the sum of the \ur-costs
  $c(\G^*[V'],\mathscr{V}')$ in each recursion step, and thus, also zero.
  Therefore, we conclude that we still have $(\G^*,\sigma)=(\G,\sigma)$ in
  Line~\ref{line:return-edited}.  \qed
\end{proof}

By Thm.~\ref{thm:algo-consistent}, Alg.~\ref{alg:general-local-optimal} is
consistent whenever the choice of $\mathscr{V}$ minimizes the \ur-cost of
$\mathscr{V}$ in each step. We shall see in Sec.~\ref{sect:UR-NP} that
minimizing $c(\G,\mathscr{V})$ is a difficult optimization problem in
general. Therefore, a good heuristic will be required for this step. This,
however, may not guarantee consistency of
Alg.~\ref{alg:general-local-optimal} in general.  The second rule in
Thm.~\ref{thm:algo-consistent} provides a remedy: the Aho graph
$[\mathscr{R}(\G^*[V'],\sigma_{|V'}), V']$ can be computed
efficiently. Whenever $[\mathscr{R}(\G^*[V'],\sigma_{|V'}), V']$ is not
connected, the partition $\mathscr{V}_{\Aho}$ defined by the connected
components $[\mathscr{R}(\G^*[V'],\sigma_{|V'}), V']$ is chosen provided it
has \ur-cost zero.  This procedure is effectively a generalization of the
algorithm \texttt{BUILD} using as input the set of informative triples
$\mathscr{R}(\G,\sigma)$ of a properly vertex-colored graph
$(\G,\sigma)$. If $(\G,\sigma)$ is already a BMG, then the recursion in
Alg.~\ref{alg:general-local-optimal} is exactly the same as in
\texttt{BUILD}: it recurses on the connected components of the Aho graph
(cf.\ Prop.~\ref{prop:BMG-charac}). We can summarize this discussion as
\begin{corollary}
  $(\G,\sigma)$ is a BMG if and only if, in every step of the
  \texttt{BUILD} algorithm operating on $\mathscr{R}(\G,\sigma)_{|V'}$ and
  $V'$, either $|V'|=1$, or $c(\G^*[V'], \mathscr{V}_{\Aho})=0$ for the
  connected component partition $\mathscr{V}_{\Aho}$ of the disconnected
  Aho graph $[\mathscr{R}(\G^*[V'],\sigma_{|V'}), V']$.
\end{corollary}    
For recursion steps in which the Aho graph
$[\mathscr{R}(\G^*[V'],\sigma_{|V'}), V']$ is connected, and possibly also
in steps with non-zero \ur-cost, another (heuristic) rule has to be
employed. As a by-product, we obtain an approach for the case that
$\mathscr{R}(\G,\sigma)$ is consistent: Following \texttt{BUILD} yields the
approximation $\G(\Aho(\mathscr{R}(\G,\sigma),V(\G)),\sigma)$ as a natural
choice.

\section{Binary-explainable BMGs}
\label{sect:beBMG}

Phylogenetic trees are often binary. Multifurcations are in many cases --
but not always -- the consequence of insufficient data
\cite{DeSalle:94,Sayyari:18,Schaller:20p}. It is therefore of practical
interest to consider BMGs that can be explained by a binary tree:
\begin{definition}
  A properly colored digraph $(\G,\sigma)$ is a \emph{binary-explainable best
    match graph} (\emph{beBMG}) if there is a binary tree $T$ such that
  $\G(T,\sigma)=(\G,\sigma)$.
\end{definition}

Correspondingly, it is of interest to edit a properly colored digraph to a 
beBMG, which translates to the following decision problem:
\begin{problem}[\PROBLEM{$\ell$-BMG Editing restricted to 
    Binary-Explainable\newline
    Graphs (EBEG)}]\ \\
  \label{prblm:ell-bmg-EBEG}
  \begin{tabular}{ll}
    \emph{Input:}    & A properly $\ell$-colored digraph $(\G =(V,E),\sigma)$
    and an integer $k$.\\
    \emph{Question:} & Is there a subset $F\subseteq V\times V \setminus
    \{(v,v)\mid v\in V\}$ such
    that $|F|\leq k$ and\\ & $(\G\symdiff F,\sigma)$
    is a binary-explainable $\ell$-BMG?
  \end{tabular}
\end{problem}
We call the corresponding completion and deletion problem
\PROBLEM{$\ell$-BMG CBEG} and \PROBLEM{$\ell$-BMG DBEG}, respectively.  As
their more general counterparts, all three variants are NP-complete as well,
cf.\ \cite[Cor.~6.2]{Schaller:20y} and \cite[Thm.~5]{Schaller:20p}.

Since the recursive partitioning in Alg.~\ref{alg:general-local-optimal}
defines a tree that explains the edited BMG, see Thm.~\ref{thm:algo-tree},
it is reasonable to restrict the optimization of $\mathscr{V}$ in
Line~\ref{line:min-cost} to bipartitions. The problem still remains hard,
however, since the corresponding decision problem (problem \PROBLEM{BPURC}
in Sec.~\ref{sect:UR-NP}) is NP-complete as shown in
Thm.~\ref{thm:BPURC-NPc} below.  Similar to BMGs in general, beBMGs have a
characterization in terms of informative triples:
\begin{proposition}{\cite[Thm.~3.5]{Schaller:20p}}
  \label{prop:Rbin-MAIN}
  A properly vertex-colored graph $(\G,\sigma)$ with vertex set $V$ is
  binary-explainable if and only if (i) $(\G,\sigma)$ is sf-colored, and
  (ii) the triple set $\Rbin(\G,\sigma)$ is consistent. In this case, the
  BMG $(\G,\sigma)$ is explained by every refinement of the \emph{binary
    refinable tree} $(\Aho(\Rbin(\G,\sigma), V), \sigma)$.
\end{proposition}

Using Prop.~\ref{prop:Rbin-MAIN}, we can apply analogous arguments as in
the proof of Lemma~\ref{lem:cost-0-BMG} for $\Rbin(\G,\sigma)$ instead of
$\mathscr{R}(\G,\sigma)$ to obtain
\begin{corollary}
  \label{cor:cost-0-beBMG}
  Let $(\G=(V,E),\sigma)$ be a beBMG with $|V|\ge 2$ and $\mathscr{V}$ be
  the connected components of the Aho graph $[\Rbin(\G,\sigma), V]$. Then
  the partition $\mathscr{V}$ of $V$ satisfies $|\mathscr{V}|\ge 2$ and
  $c(\G,\mathscr{V})=0$.
\end{corollary}

Since a beBMG $(\G,\sigma)$ is explained by every refinement of the Aho
tree constructed from $\Rbin(\G,\sigma)$ (cf.\ Prop.~\ref{prop:Rbin-MAIN}),
we can obtain a slightly more general result.
\begin{lemma}
  \label{lem:cost-0-beBMG-coarsements}
  Let $(\G=(V,E),\sigma)$ be a beBMG with $|V|\ge 2$ and $\mathscr{V}$ be
  the connected components of the Aho graph $[\Rbin(\G,\sigma),V]$. Then,
  every coarse-graining $\mathscr{V}'$ of $\mathscr{V}$ with
  $|\mathscr{V}'|\ge 2$ satisfies $c(\G,\mathscr{V}')=0$.
\end{lemma}
\begin{proof}
  First note that $\Rbin(\G,\sigma)$ is consistent by
  Prop.~\ref{prop:Rbin-MAIN} since $(\G,\sigma)$ is a beBMG. Therefore,
  $|V|\ge 2$ implies $|\mathscr{V}|\ge 2$ \cite{Aho:81}.  For the trivial
  coarse-graining $\mathscr{V}'=\mathscr{V}$, Cor.~\ref{cor:cost-0-beBMG}
  already implies the statement.  Now assume $\mathscr{V}'\ne\mathscr{V}$.
  Observe that the tree
  $(T,\sigma) \coloneqq (\Aho(\Rbin(\G,\sigma), V), \sigma)$ exists and
  explains $(\G,\sigma)$ by Prop.~\ref{prop:Rbin-MAIN}.  Moreover, there
  is, by construction, a one-to-one correspondence between the children $v_i$
  of its root $\rho$ and the elements in $V_i\in\mathscr{V}$ given by
  $L(T(v_i))=V_i$.  We construct a refinement (tree) $T'$ of $T$ as follows:
  Whenever we have multiple sets $V_i\in\mathscr{V}$ that are subsets of
  the same set $V_j\in\mathscr{V}'$, we remove the edges $\rho v_i$ to the
  corresponding vertices $v_i\in\child_{T}(\rho)$ in $T$, and collectively
  connect these $v_i$ to a newly created vertex $w_j$. These vertices $w_j$ are
  then reattached to the root $\rho$.  Since $|\mathscr{V}'|\ge 2$ by 
  assumption, the
  so-constructed tree $T'$ is still phylogenetic. Moreover, it satisfies
  $\mathscr{V}'=\{L(T'(v)) \mid v\in\child_{T'}(\rho)\}$, and thus,
  $T'\in \mathscr{T}(\mathscr{V}')$.  It is a refinement of $T$ since
  contraction of the edges $\rho w_j$ again yields $T$.  Hence, we can
  apply Prop.~\ref{prop:Rbin-MAIN} to conclude that $(T',\sigma)$ also
  explains $(\G,\sigma)$.  It follows immediately that
  $U(\G,T')=\emptyset$. The latter together with
  $T'\in \mathscr{T}(\mathscr{V}')$ implies $U(\G,\mathscr{V}')=\emptyset$,
  and thus $c(\G,\mathscr{V}')=0$.  \qed
\end{proof}

We are now in the position to formulate an analogue of
Thm.~\ref{thm:algo-consistent} for variants of
Alg.~\ref{alg:general-local-optimal} that aim to edit a properly-colored
digraph $(\G,\sigma)$ to a beBMG.
\begin{theorem}
  \label{thm:algo-consistent-binary}
  Alg.~\ref{alg:general-local-optimal} is consistent for beBMGs
  $(\G,\sigma)$ if, in each step on $V'$ with $|V'|\ge 2$, a bipartition
  $\mathscr{V}$ in Line~\ref{line:min-cost} is chosen according to one of
  the following rules:
  \begin{enumerate}
    \item $\mathscr{V}$ has minimal \ur-cost among all possible bipartitions
    $\mathscr{V}'$ of $V'$.
    \item If the Aho graph $[\Rbin(\G^*[V'],\sigma_{|V'}),V']$ is
    disconnected with the set of connected components $\mathscr{V}_{\Aho}$,
    and moreover $c(\G^*[V'],\mathscr{V}_{\Aho})=0$, then $\mathscr{V}$ is
    a coarse-graining of $\mathscr{V}_{\Aho}$.
  \end{enumerate}
\end{theorem}
\begin{proof}
  We have to show that the final edited graph $(\G^*,\sigma)$ returned in
  Line~\ref{line:return-edited} equals the input graph $(\G=(V,E),\sigma)$
  whenever $(\G,\sigma)$ already is a beBMG, i.e., nothing is edited.  Thus
  suppose that $(\G,\sigma)$ is a beBMG and first consider the top-level
  recursion step on $V$ (where initially $\G^*=\G$ still holds at
  Line~\ref{line:init-G*}).  If $|V|=1$, neither $(\G,\sigma)$ nor
  $(\G^*,\sigma)$ contain any arcs, and thus, the edit cost is trivially
  zero. Now suppose $|V|\ge 2$.  Since $(\G,\sigma)$ is a beBMG,
  $\Rbin\coloneqq\Rbin(\G,\sigma)$ is consistent, and thus, the set of
  connected components $\mathscr{V}_{\Aho}$ of the Aho graph $[\Rbin, V]$
  has a cardinality of at least two. If $|\mathscr{V}_{\Aho}|=2$,
  $\mathscr{V}\coloneqq \mathscr{V}_{\Aho}$ is a bipartition satisfying
  $c(\G,\mathscr{V})=0$ by Cor.~\ref{cor:cost-0-beBMG}.  If
  $|\mathscr{V}_{\Aho}|>2$, we can find an arbitrary bipartition
  $\mathscr{V}$ that is a coarsement of $\mathscr{V}_{\Aho}$.  By
  Lemma~\ref{lem:cost-0-beBMG-coarsements}, $\mathscr{V}$ also satisfies
  $c(\G,\mathscr{V})=0$ in this case.  Hence, for both rules~(1) and~(2),
  we choose a bipartition $\mathscr{V}$ with (minimal) \ur-cost
  $c(\G,\mathscr{V})=0$.  Now, Lemma~\ref{lem:cost-0-subgraphs-BMG} implies
  that the induced subgraph $(\G[V'],\sigma_{|V'})$ is a BMG for every
  $V'\in\mathscr{V}$.  To see that $(\G[V'],\sigma_{|V'})$ is also
  binary-explainable, first note that
  $\Rbin(\G[V'],\sigma_{|V'})=\Rbin_{|V'}$ by Obs.~\ref{obs:R-restriction}.
  This together with the fact that $\Rbin_{|V'}\subseteq\Rbin$ and
  Obs.~\ref{obs:R-F-subsets} implies that $\Rbin(\G[V'],\sigma_{|V'})$ is
  consistent.  Moreover, Prop.~\ref{prop:BMG-charac} and
  $(\G[V'],\sigma_{|V'})$ being a BMG together imply that
  $(\G[V'],\sigma_{|V'})$ is sf-colored.  Hence, we can apply
  Prop.~\ref{prop:Rbin-MAIN} to conclude that $(\G[V'],\sigma_{|V'})$ is a
  beBMG.
  
  Since we recurse on the subgraphs $(\G[V'],\sigma_{|V'})$, which are
  again beBMGs, we can repeat the arguments above along the recursion
  hierarchy to conclude that the \ur-cost $c(\G^*[V'],\mathscr{V}')$
  vanishes in every recursion step.  By Cor.~\ref{cor:sum-c}, the total
  edit cost of Alg.~\ref{alg:general-local-optimal} is the sum of the
  \ur-costs $c(\G^*[V'],\mathscr{V}')$ in each recursion step, and thus,
  also zero.  Therefore, we conclude that we still have
  $(\G^*,\sigma)=(\G,\sigma)$ in Line~\ref{line:return-edited}.  \qed
\end{proof}

\section{Minimizing the \ur-cost $c(\G,\mathscr{V})$}
\label{sect:UR-NP}

The problem of minimizing $c(\G,\mathscr{V})$ for a given properly
colored graph $(\G,\sigma)$ corresponds to the following decision
problem.

\begin{problem}[\PROBLEM{(Bi)Partition with \ur-Cost ((B)PURC)}]\ \\
  \begin{tabular}{ll}
    \emph{Input:}    & A properly $\ell$-colored digraph $(\G =(V,E),\sigma)$
    and an integer $k\geq 0$.\\
    \emph{Question:} & Is there a (bi)partition $\mathscr{V}$ of $V$ 
    such that $c(\G,\mathscr{V})\leq k$?
  \end{tabular}
\end{problem}
In the Appendix, we show that \PROBLEM{(B)PURC} is NP-hard by reduction 
from \PROBLEM{Set Splitting}, one of \citeauthor{Garey:79}'s
\citeyearpar{Garey:79} classical NP-complete problems.

\begin{theorem}
  \label{thm:BPURC-NPc}
  \PROBLEM{BPURC} is NP-complete.
\end{theorem}

Thm.~\ref{thm:algo-consistent} suggests to consider heuristics
for \PROBLEM{(B)PURC} that make use of the Aho graph in the following
manner:
\begin{enumerate}
  \item Construct the Aho graph $H\coloneqq [\mathscr{R}(\G,\sigma),V]$
  based on the set of informative triples $\mathscr{R}(\G,\sigma)$.
  \item If $H$ has more than one connected component, we use the set of
  connected components as the partition $\mathscr{V}$.
  \item If $H$ is connected, a heuristic that operates on the Aho graph $H$
  is used to find a partition $\mathscr{V}$ with small \ur-cost
  $c(\G,\mathscr{V})$.
\end{enumerate}
Plugging any algorithm of this type into Line~\ref{line:min-cost} of
Alg.~\ref{alg:general-local-optimal} reduces
the algorithm to \texttt{BUILD} if a BMG is used as
input and thus guarantees consistency (cf.\ Prop.~\ref{prop:BMG-charac}).
We note, however, that the connected components of a disconnected Aho graph
are not guaranteed to correspond to an optimal solution for
\PROBLEM{(B)PURC} in the general case.

\section{Computational experiments}
\label{sec:heur-computational-exp}

In this section, we compare different heuristics for the \PROBLEM{(B)PURC}
Problem and their performance in the context of BMG editing. Somewhat
unexpectedly, but in accordance with
Fig.~\ref{fig:global-vs-local-optimum}, our results suggest that a good (or
bad) performance of \PROBLEM{(B)PURC} is not directly linked to a good (or
bad) performance for BMG editing. Moreover, we find that, even for noisy
data, all analyzed methods are able to capture the tree structure of the
underlying ``true'' BMG at least to some extent. As we shall see, community
detection approaches in combination with the \ur-cost appear to be more
promising for BMG editing than optimal solutions of \PROBLEM{(B)PURC} alone.

\subsection{Heuristics for \PROBLEM{(B)PURC}}\label{subsec:exp_1}

\PROBLEM{(B)PURC} is a variation on graph partitioning problems. It seems
reasonable, therefore, to adapt graph partitioning algorithms for our
purposes. 

\paragraph{MinCut.}
We solve the minimum edge cut problem for the connected undirected graph
$H$, i.e., we want to find a bipartition $\mathscr{V}=\{V_1,V_2\}$ such
that the number of edges between $V_1$ and $V_2$ is minimal in $H$.  The
problem can be solved exactly in polynomial time using the Stoer-Wagner
algorithm \cite{Stoer:97}. Note, however, that the minimum edge cut in $H$
will in general not deliver an optimal solution of \PROBLEM{(B)PURC}.

\paragraph{Karger's algorithm} is a randomized algorithm that, in its
original form, also aims to find a minimum edge cut \cite{Karger:93}.  In
brief, it merges vertices of the graph by randomly choosing and contracting
edges, until only two vertices remain, which induce a bipartition
$\mathscr{V}$ according to the vertices that were merged into them.  By
repeating this process a sufficient number of times, a minimum edge cut can
be found with high probability. Here, we use the \ur-cost
$c(\G,\mathscr{V})$ instead of the size of the edge cut as objective
function to select the best solution over multiple runs.

\paragraph{A simple greedy approach} starts with
$\mathscr{V}=\{V_1=\emptyset, V_2=V'\}$ and stepwise moves a vertex
$v\in V_2$ to $V_1$ such that $c(\G,\{V_1\cup\{v\},V_2\setminus\{v\}\})$ is
optimized. Ties are broken at random. This produces $|V|-1$ ``locally
optimal'' bipartitions, from which the best one is selected.

\paragraph{Gradient walks.}
Here we interpret the space of all bipartitions $\mathscr{V}$ endowed with
the objective function $c(\G,\mathscr{V})$ as a fitness landscape.  We
start with a random but balanced bipartition $\mathscr{V}=\{V_1, V_2\}$.
As move set we allow moving one vertex from $V_1$ to $V_2$ or \textit{vice
  versa}.  In each step, we execute the move that best improves the
objective function, and stop when we reach a local optimum. 

\paragraph{Louvain method.} This method for community detection in graphs
greedily optimizes the so-called modularity of a vertex partition
$\mathscr{V}$ \cite{Blondel:08}. Its objective function is
$q(\mathscr{V}) = \sum_{W\in\mathscr{V}}\sum_{u,v\in W}
(a_{uv}-d_ud_v/(2m))$, where $a_{uv}$ are the entries of the (possibly
weighted) adjacency matrix of a graph $H$, $d_u=\sum_v a_{uv}$ the vertex 
degrees,
and $m$ is the sum of all edge weights in the graph. This favors so-called
\emph{communities} or \emph{modules} $W$ that are highly connected
internally but have only few edges between them.  The Louvain method
operates in two phases starting from the discrete partition
$\mathscr{V}=\{\{u\}\,|\, u\in V\}$. In the first phase, it repeatedly
iterates over all vertices $x$ and moves $x$ into the community of one of
its neighbors that leads to the highest gain in modularity as long as a
move that increases $q(\mathscr{V})$ can be found. The second phase repeats
the first one on the weighted quotient graph $H/\mathscr{V}$ whose vertices
are the sets of $\mathscr{V}$ and whose edge weights are the sum of the
original weights between the communities.  In addition to maximizing the
modularity, we also investigate a variant of the Louvain method that moves
vertices into the community of one of their neighbors if this results in a
lower \ur-cost $c(\G,\mathscr{V})$, and otherwise proceeds analogously. We
exclude the merging of the last two vertices to ensure that a non-trivial
partition is returned. Since the Louvain method is sensitive to the order
in which the vertex set is traversed, we randomly permute the order of
vertices to allow multiple runs on the same input.

\smallskip With the exception of the Stoer-Wagner algorithm for solving the
minimum edge cut problem, all of these partitioning methods include random
decisions.  One may therefore run them multiple times and use the partition
corresponding to the best objective value, i.e., the lowest \ur-cost
$c(\G,\mathscr{V})$ or the highest modularity. If not stated otherwise, we
apply five runs for each of these methods in each recursion step (with a
connected Aho graph) in the following analyses.

\subsubsection*{Construction of test instances} 

Since we are interested in the \PROBLEM{(B)PURC} problem in the context of
BMG editing, we test the heuristics on ensembles of perturbed BMGs that were
constructed as follows: We first generate leaf-colored trees $(T,\sigma)$
with a predefined number of vertices $N$ and colors $\ell$ and then compute
their BMGs $\G(T,\sigma)$.  For each tree, we start from a single
vertex. We then repeatedly choose one of the existing vertices $v$
randomly, and, depending on whether $v$ is currently an inner vertex or a 
leaf, attach either a single or two new leaves to it, respectively.
Hence, the number
of leaves increases by exactly one and the tree remains phylogenetic in
each step.  We stop when the desired number $N$ of leaves is reached.  In
the next step, colors are assigned randomly to the leaves under the
constraint that each of the $\ell$ colors appears at least once.  We note
that trees created in this manner are usually not least resolved, and their
BMGs are in general not binary-explainable.  Finally, we disturb these BMGs
by inserting and deleting arcs according to a specified insertion and
deletion probability, respectively.  Since arcs between vertices of the
same color trivially cannot correspond to best matches, we do not insert
arcs between such vertices, i.e., the input graphs for the editing are all
properly vertex-colored digraphs.

For the purpose of benchmarking the heuristics for the \PROBLEM{(B)PURC}
problem, we only retain perturbed BMGs $(\G,\sigma)$ with a connected Aho
graph $H\coloneqq[\mathscr{R}(\G,\sigma), V(\G)]$ because the heuristics
are not applied to instances with a disconnected Aho graph~$H$.  Depending on
the insertion and deletion probabilities, we retained 93\% to
100\% of the initial sample, except in the case where arcs were only
inserted to obtain a disturbed graph. Here, the Aho graph $H$ was connected
in 60\% of the initial sample. Thus, even moderate perturbation of a BMG
introduces inconsistencies into the triple set $\mathscr{R}(\G,\sigma)$ and
results in a connected Aho graph $H$ in the majority of cases. As shown in
Fig.~\ref{fig:introduce-inconsistency}, both arc insertions and deletions
can cause triple inconsistencies.

\begin{figure}[!t]
  \centering
  \includegraphics[width=0.85\textwidth]{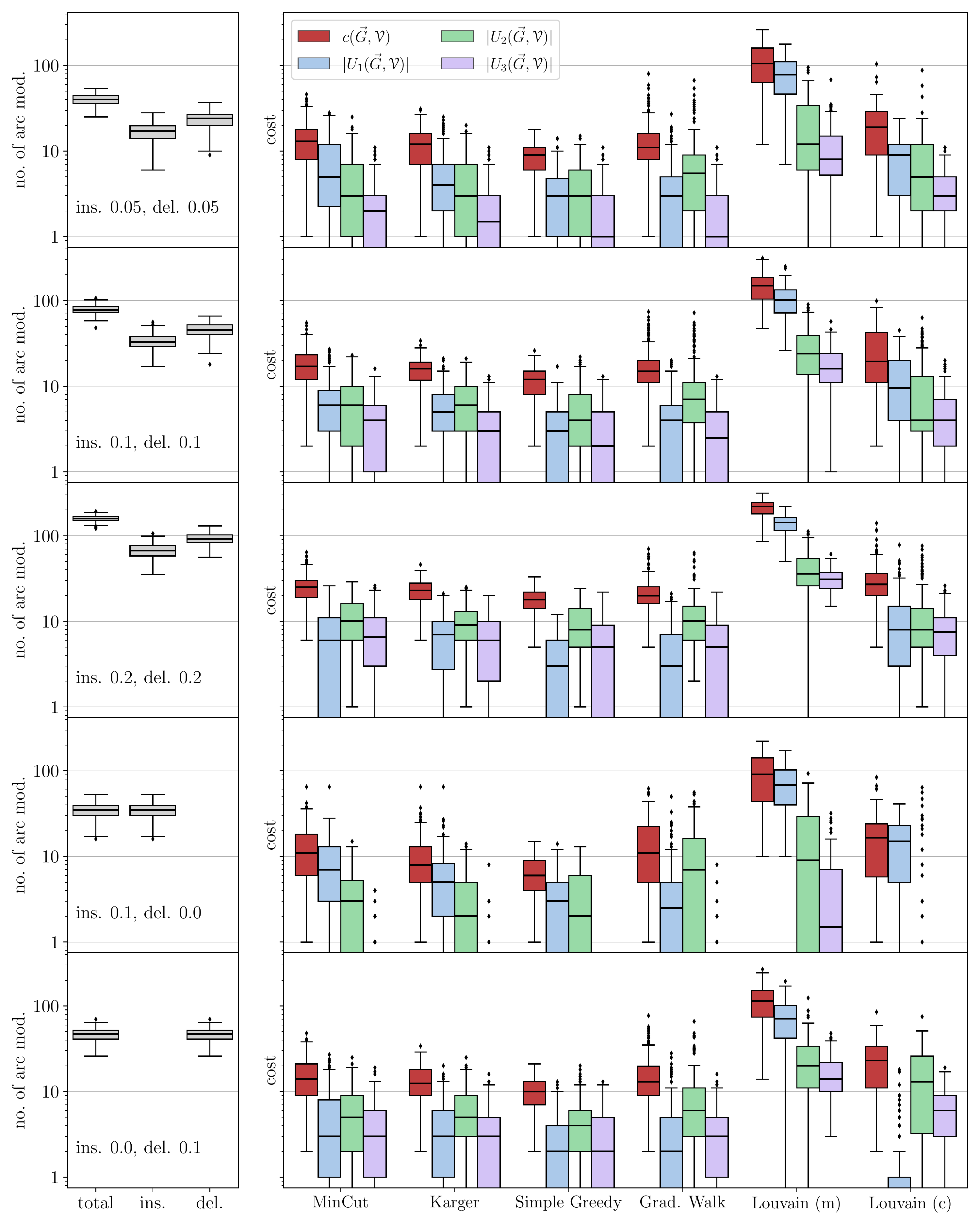}
  \caption{Performance of partitioning methods for minimizing
    $c(\G,\mathscr{V})$ on perturbed BMGs $(\G,\sigma)$.  The rows
    correspond to different insertion and deletion probabilities (indicated
    in the l.h.s.\ panels) used to disturb the original BMGs.  The l.h.s.\
    panels show the distribution of the no.\ of arc modifications in total,
    arc insertions and arc deletions of the disturbed graphs w.r.t.\ the
    original BMGs.  The r.h.s.\ panels show the distribution of \ur-costs
    $c(\G,\mathscr{V})$ (red) obtained for each method, and of the no.\ of
    arcs in $U_1(\G,\mathscr{V})$, $U_2(\G,\mathscr{V})$, and
    $U_3(\G,\mathscr{V})$ (i.e., the sets that contribute to the \ur-cost).
    Example plot for $|V|=30$ vertices and $|\sigma(V)|=10$ colors in each
    graph.  Among the 200 generated graphs, only those with a connected Aho
    graph $[\mathscr{R}(\G,\sigma), V(\G)]$ are included in each of the
    five rows (93\%, 100\%, 100\%, 60\%, 95\%).}
  \label{fig:part-quality}
\end{figure}

\subsubsection*{Benchmarking Results} 

Fig.~\ref{fig:part-quality} suggests that the \emph{Simple Greedy} approach
is best suitable for the minimization of the \ur-cost $c(\G,\mathscr{V})$
for any of the considered parameters for BMG disturbance.  The Louvain
method based on graph modularity (\emph{Louvain (m)}) appears to have by
far the worst performance which, moreover, quickly produces higher
\ur-costs with an increasing intensity of the perturbations.

In order to better understand the behavior of the repeated application
of the partitioning heuristics of Alg.~\ref{alg:general-local-optimal}, it
is instructive to consider not only the score but also the structure of
partitions. We observe a strong tendency of some of the partitioning
methods to produce \emph{single-leaf splits}, i.e.,
(bi)partitions $\mathscr{V}$ in which at least one set $W\in\mathscr{V}$
is a singleton (i.e., $|W|=1$). Single-leaf splits in general seem to have
relatively low \ur-costs. Further details on the propensity of the
partitioning heuristics to produce single-leaf splits are given in
Appendix~\ref{sect:app-single}.

\subsection{Heuristics for BMG Editing}

In this section, we explore the performance of several variants of
Alg.~\ref{alg:simple} and~\ref{alg:general-local-optimal} for BMG
editing. The variants of Alg.~\ref{alg:general-local-optimal} correspond to
using the heuristics for \PROBLEM{(B)PURC} discussed above for processing
a connected Aho graph
$H\coloneqq [\mathscr{R}(\G^*[V'],\sigma_{|V'}),V']$ for the informative
triples $\mathscr{R}(\G^*[V'],\sigma_{|V'})$ in each step of the recursion.
We note that Alg.~\ref{alg:general-local-optimal} in combination with any
of the heuristics for \PROBLEM{(B)PURC} also serves as a heuristic for
\PROBLEM{MaxRTC} because the choice of the partition $\mathscr{V}$ in each
recursion step determines a set of 
included triples $xy|z$, namely those
for which $x$ and $y$ are contained in one set of $\mathscr{V}$ while $z$
is contained in another. Another way of expressing that same fact is that
an approximation to \PROBLEM{MaxRTC} is given by the subset
$\mathscr{R}^*\subseteq \mathscr{R}(\G,\sigma)$ of the informative triples
of the input graph $(\G,\sigma)$ that are displayed by the tree $T$
constructed in Alg.~\ref{alg:general-local-optimal}. In particular,
Alg.~\ref{alg:general-local-optimal} together with the \emph{MinCut}
method has been described as a heuristic for \PROBLEM{MaxRTC} in earlier
work \cite{Gasieniec:99,Byrka:10}. For comparison, we will also consider
the following bottom-up approach as a component of Alg.~\ref{alg:simple}:

\paragraph{Best-Pair-Merge-First (BPMF)} was described by \citet{Wu:04},
and constructs a tree from a set of triples $\mathscr{R}$ in a bottom-up
fashion. We use here a modified version introduced by
\citet{Byrka:10}. BPMF operates similar to the well-known UPGMA clustering
algorithm \cite{Sokal:58}. Starting with each vertex $x\in V$ as its own
cluster, pairs of clusters are merged iteratively, thereby defining a
rooted binary tree with leaf set $V$.  The choice of the two clusters to
merge depends on a similarity score with the property that any triple
$xy|z$ with $x$, $y$, and $z$ lying in distinct clusters $S_x$, $S_y$, and
$S_z$ contributes positively to $\textrm{score}(S_x,S_y)$ and negatively to
$\textrm{score}(S_x,S_z)$ and $\textrm{score}(S_y,S_z)$. Since BPMF
constructs the tree $T$ from the bottom, it does not imply a vertex
partitioning scheme that could be plugged into the top-down procedure of
Alg.~\ref{alg:general-local-optimal}.  Importantly, BPMF is not a
consistent heuristic for \PROBLEM{MaxRTC}, i.e.\ it does not necessarily
recognize consistent triples sets. Hence, consistency in the application to
BMG editing is also not guaranteed, see Fig.~\ref{fig:BPMF-not-consistent}
in Appendix~\ref{sect:app-BPMF} for an example.

\smallskip
In summary, we have two distinct ways to obtain an edited BMG: We may
take either 
\begin{enumerate}
  \item $\G(T,\sigma)$, where $T$ is the output tree of
  Alg.~\ref{alg:general-local-optimal} or BPMF, respectively, or
  \item $\G(T^*, \sigma)$, where $T^*=\Aho(\mathscr{R}^*,V(\G))$ is
  constructed from the consistent triple subset of triples
  $\mathscr{R}^*$. This corresponds to Alg.~\ref{alg:simple}.
\end{enumerate}
Somewhat surprisingly, the results in Fig.~\ref{fig:edit_comparison1}
suggest that it is in general beneficial to extract the triple set
$\mathscr{R}^*$ and rerun the \texttt{BUILD} algorithm, i.e., to use
$\G(T^*,\sigma)$.

\subsubsection*{Benchmarking Results}

\begin{figure}[!t]
  \centering
  \includegraphics[width=0.85\textwidth]{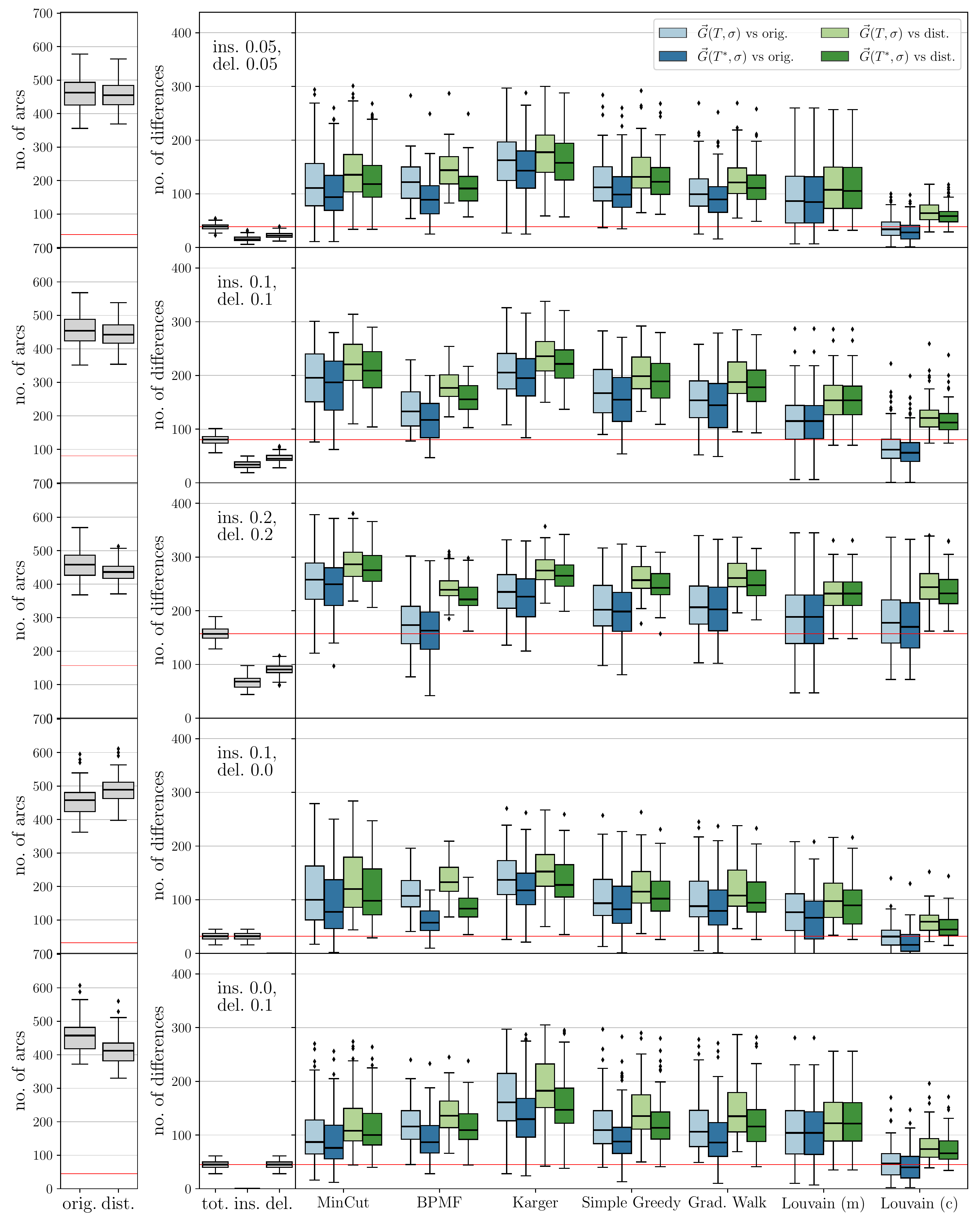}
  \caption{Performance comparison of several BMG editing heuristics based
    on the no.\ of arc differences. The rows correspond to different
    insertion and deletion probabilities (indicated in the second column
    panels) used to perturb the original BMGs.  The l.h.s.\ panels show the 
    distribution of the no.\ of arcs in the original BMG and in the
    perturbed graph.  The second column panels show the distribution of the
    no.\ of arc modifications in total, arc insertions and arc deletions of
    the perturbed graphs w.r.t.\ the original BMGs.  The red lines mark the
    median values of the total no.\ of modifications.  The r.h.s.\ panels
    show the total no.\ of arc differences w.r.t.\ the original random BMGs
    (blue) and the perturbed graphs (green).  The light colors indicate the
    ``direct'' performance of each method, i.e., the graph $\G(T,\sigma)$
    where $T$ is the tree that is directly constructed by each method.  The
    darker colors indicate the results if the methods are used as heuristic
    for \PROBLEM{MaxRTC} in Alg.~\ref{alg:simple}.  Example plot for
    $|V|=30$ vertices and $|\sigma(V)|=10$ colors in each graph, 100 graphs
    per row.}
  \label{fig:edit_comparison1}
\end{figure}

\begin{figure}[!t]
  \centering
  \includegraphics[width=0.85\textwidth]{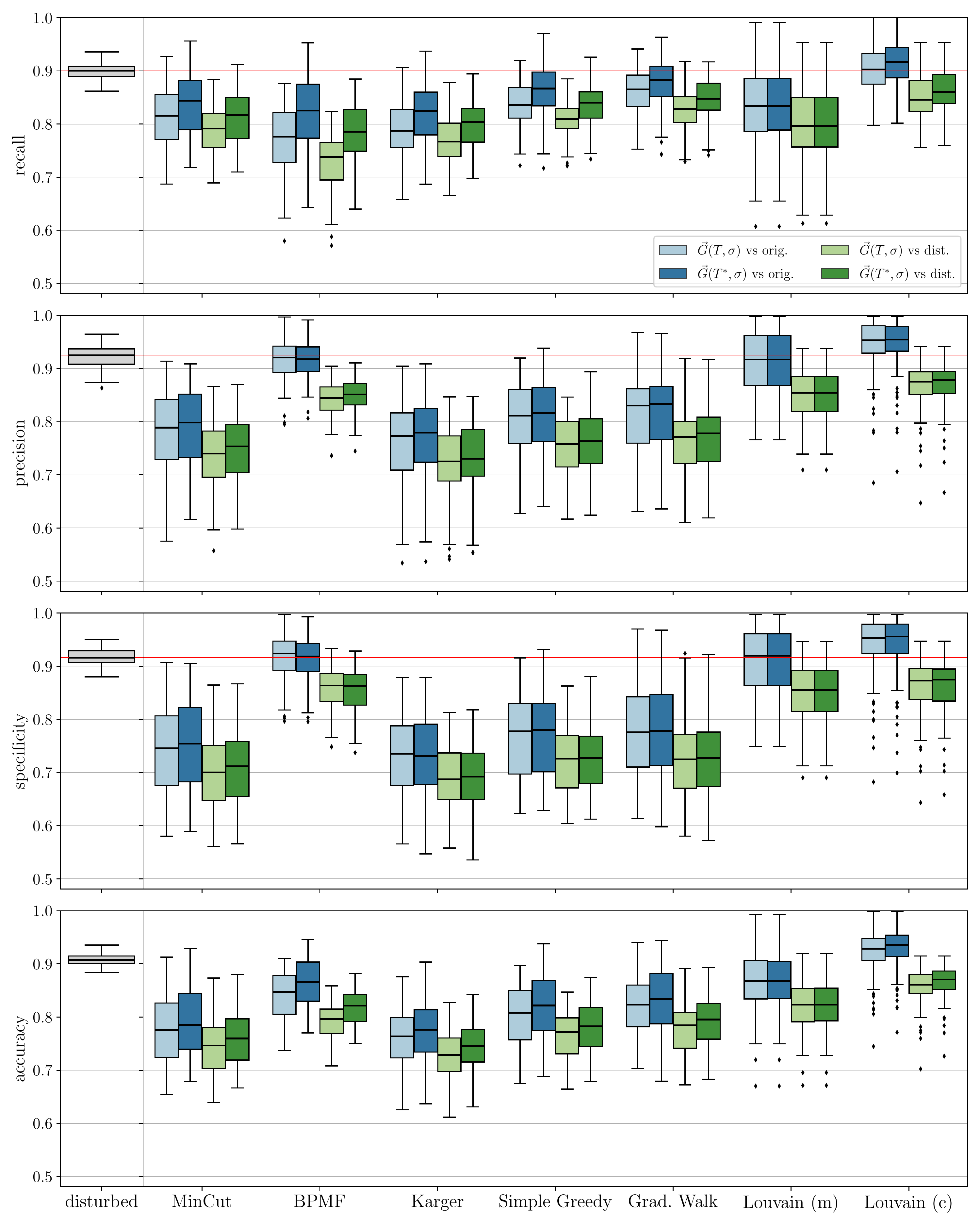}
  \caption{Performance comparison of several BMG editing heuristics based
    on recall, precision, specificity, and accuracy (rows 1 to 4).  The
    l.h.s.\ panels show the respective measure for the perturbed graph
    w.r.t.\ the original random BMG.  The red lines marks the median values
    of the latter.  The r.h.s.\ panels show the results for the edited
    graphs w.r.t.\ the original BMGs (blue) and the perturbed graphs
    (green).  The light colors indicate the ``direct'' performance of each
    method, i.e., the graph $\G(T,\sigma)$ where $T$ is the tree that is
    directly constructed by each method.  The darker colors indicate the
    results if the methods are used as heuristic for \PROBLEM{MaxRTC} in
    Alg.~\ref{alg:simple}.  Example plot for $|V|=30$ vertices and
    $|\sigma(V)|=10$ colors in each graph, insertion and deletion
    probability 0.1, and 100 graphs.}
  \label{fig:edit_comparison2}
\end{figure}

To assess the performance of the various heuristics, we consider the
differences between the editing result $(G^*,\sigma)$ from both the
original BMG $(G_{orig},\sigma)$ and the perturbed input graphs
$(G,\sigma)$. In Fig.~\ref{fig:edit_comparison1}, we summarize the absolute
values of the symmetric differences of the arc sets
$d_{orig} \coloneqq |E(G^*)\symdiff E(G_{orig})|$ and
$d \coloneqq |E(G^*)\symdiff E(G)|$, respectively.  These results are
translated to usual normalized performance indicators (recall, precision,
specificity, and accuracy; all defined in terms of the arc sets) in
Fig.~\ref{fig:edit_comparison2}.

Comparing the distances $d_{orig}$ (blue boxplots) and $d$ (green 
boxplots) of the editing result
$(G^*,\sigma)$ to original unperturbed BMG and the input graph, resp., we
find that, for the methods investigated here, on average $d_{orig}$ is
smaller than $d$. This indicates that all methods are able to capture the
underlying tree structure of the original BMG at least to some extent.  The
discrepancy between $d_{orig}$ and $d$ tends to increase with the level of
perturbation, a trend that is most pronounced for \emph{Louvain (c)}.  This
result is encouraging for practical applications of BMG modification to
correcting noisy best match data, where the eventual goal is to obtain a
good estimate of the underlying true BMG.

Intriguingly, the extraction of consistent informative triples
$\mathscr{R}^*$ from the reconstructed tree $T$ and rerunning
\texttt{BUILD}, i.e., using $\G(T^*,\sigma)$, in general improves the
estimation results for the majority of methods.  In particular, this
increases the recall without a notable negative impact on precision and
specificity (cf.\ Fig.~\ref{fig:edit_comparison2}).  A better recall,
corresponding to a higher proportion of correctly inferred arcs, is not
surprising in this context, since this additional step in essence reduces
the number of triples.  We therefore expect the tree
$T^*=\Aho(\mathscr{R}^*,V(\G))$ to be on average less resolved than $T$.
The BMGs of less resolved trees tend to have more arcs than BMGs of highly
resolved tree (cf.\ \cite[Lemma~8]{Schaller:20x}).  In good accordance with
this prediction, \emph{BPMF}, which shows a strong increase of recall,
always constructs a binary, i.e., fully-resolved, tree $T$ -- whereas the
corresponding tree $T^*$ in general is much less resolved.

Somewhat surprisingly, a good or bad performance for minimizing the
\ur-cost in individual steps apparently does not directly translate to the
performance in the overall editing procedure.  In particular, the
modularity-based \emph{Louvain (m)} method seems to be a better choice than
the \emph{Simple Greedy} approach.  The methods \emph{MinCut} and
\emph{Karger} do not seem to be suitable components for
Alg.~\ref{alg:general-local-optimal}, with the exception of the case where
perturbations are arc deletions only (Fig.~\ref{fig:edit_comparison1},
bottom row). Here, \emph{MinCut} produces reasonable estimates that compare
well with other methods.  The bottom-up method for the \PROBLEM{MaxRTC}
problem \emph{BPMF} also produces relatively good results. It appears to be
robust at high levels of perturbation.  For most of the parameter
combinations, we obtain the best results with the \ur-cost-based Louvain
method (\emph{Louvain~(c)}).  Here, we often observe a symmetric difference
(w.r.t.\ the arcs sets) that is better than the difference between the
original and the perturbed graph.  This trend is illustrated by the red
median lines in Fig.~\ref{fig:edit_comparison1}
and~\ref{fig:edit_comparison2}.  Hence, we achieve two goals of BMG
editing: (i) the resulting graph $(\G^*,\sigma)$ is a BMG, i.e., it
satisfies Def.~\ref{def:BMG}, and (ii) it is closer to the original BMG
than the perturbed graph.  We note that we observed similar trends across
all investigated combinations for the numbers of leaves $N$ (ranging from
$10$ to $40$) and of colors $\ell$ ($\ell<N$ ranging from $2$ to at most
$20$).

Our results show that minimization of the \ur-cost in each step is not the
best approach to BMG editing because this often produces very unbalanced
partitions. As a consequence, more recursion steps are needed in
Alg.~\ref{alg:general-local-optimal} resulting in higher accumulated number
of arc edits.  Fig.~\ref{fig:louvain_better} shows that better solutions to
the BMG editing problem are not necessarily composed of vertex partitions
with minimal \ur-cost in each step. The perturbed graph $(\G,\sigma)$ in
Fig.~\ref{fig:louvain_better} was obtained from the randomly simulated BMG
$(\G_\textrm{orig},\sigma)$ as described above using equal insertion and
deletion probabilities of $0.1$.  As an example, the partitions
$\mathscr{V}_1$ and $\mathscr{V}_2$ as constructed by the \emph{MinCut} and
the \emph{Louvain~(c)} method in the first iteration step of
Alg.~\ref{alg:general-local-optimal} are shown as pink and green frames,
respectively.  \emph{MinCut} produces a single-leaf split $\mathscr{V}_1$
with an isolated vertex $b_2$ and \ur-cost $c(\G,\mathscr{V}_1)=1$
deriving from $U_1(\G,\mathscr{V}_1)=\{(b_2,a_2)\}$. \emph{Louvain~(c)}
identifies the partition $\mathscr{V}_2$ with $c(\G,\mathscr{V}_2)=3$
originating from
$U_2(\G,\mathscr{V}_1)=\{(b_3,a_1),(c_2,a_1), (c_2,b_1)\}$, which
corresponds to the connected components of the Aho graph $H_\textrm{orig}$
of the unperturbed BMG and thus identifies the split in the original tree
$(T,\sigma)$. Here, the correct partition $\mathscr{V}_2$ has a strictly
larger \ur-cost than the misleading choice of $\mathscr{V}_1$. However,
\emph{MinCut} results in a higher total edit cost than \emph{Louvain~(c)}
for $(\G,\sigma)$.

\begin{figure}[t]
  \centering
  \includegraphics[width=0.85\textwidth]{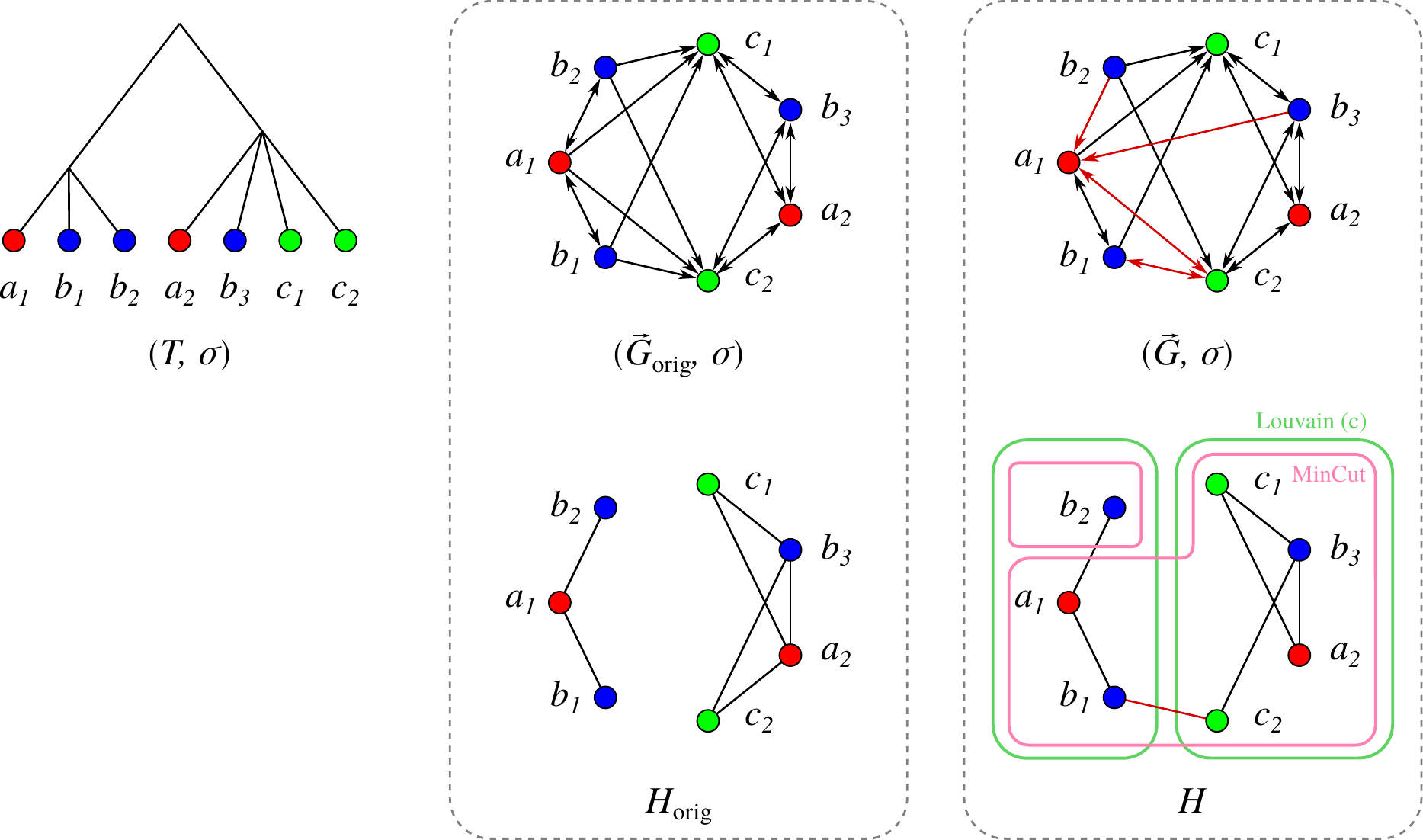}
  \caption{Example of an instance where the Louvain method performs
    better due to more balanced partitions.  The (least resolved) tree
    $(T,\sigma)$ explains the BMG $(\G_\textrm{orig},\sigma)$ with vertex
    set $V$.  The graph
    $H_\textrm{orig}=[\mathscr{R}(\G_\textrm{orig},\sigma), V]$ is the Aho
    graph corresponding to the informative triple set
    $\mathscr{R}(\G_\textrm{orig},\sigma)$.  The perturbed graph
    $(\G,\sigma)$ is obtained from $(\G_\textrm{orig},\sigma)$ by inserting
    the arcs $(b_3,a_1)$, $(c_2,a_1)$, and $(c_2, b_1)$ and deletion of
    $(a_1,b_2)$. The corresponding Aho graph
    $H=[\mathscr{R}(\G,\sigma), V]$ is connected because the perturbation
    introduced the additional informative triple $c_2b_1|b_2$.  The green
    and pink frames correspond to the partitions $\mathscr{V}_1$ and
    $\mathscr{V}_2$ of $V$ constructed by the methods \emph{Louvain~(c)} and
    \emph{MinCut}, respectively.}
  \label{fig:louvain_better}
\end{figure}

In order to account for the issue of unbalanced partitions, we performed a
cursory analysis on maximizing a gain function rather than minimizing the
\ur-cost. In analogy to $c(\G,\mathscr{V})$, we defined $g(\G,\mathscr{V})$
as the number of arcs and non-arcs that are satisfied by the BMGs of
\emph{all} trees in $\mathscr{T}(\mathscr{V})$.  Recapitulating the
arguments in the proof of Lemma~\ref{lem:unsat-rel-charac}, one can show
that these relations can also be determined as the union of three sets by
replacing ``$(x,y)\in E$'' with ``$(x,y)\notin E$'' and \textit{vice versa}
in the definitions of $U_1(\G,\mathscr{V})$, $U_2(\G,\mathscr{V})$, and
$U_3(\G,\mathscr{V})$.  The gain function $g(\G,\mathscr{V})$ can be used
instead of the \ur-cost with \emph{Karger}, \emph{Simple Greedy},
\emph{Gradient Walk}, and in a gain-function-based \emph{Louvain}
method. For all these algorithms, however, maximizing $g(\G,\mathscr{V})$
leads to partitions that appear to be \emph{too} balanced, and a
performance for BMG editing that is worse than the use of the \ur-cost.  A
possible explanation for both unbalanced and too balanced partitions as
produced with a cost and gain function, resp., is the fact that
$U_1(\G,\mathscr{V})$ and $U_2(\G,\mathscr{V})$ (and their gain function
counterparts) contain pairs of vertices $(x,y)$ that lie in distinct sets
of $\mathscr{V}$. Hence, both single-leaf splits and perfectly balanced
partitions minimize (maximize, resp.)  the number of pairs that could
potentially be contained in these arc sets.

\smallskip

\begin{figure}[t]
  \centering
  \includegraphics[width=0.85\textwidth]{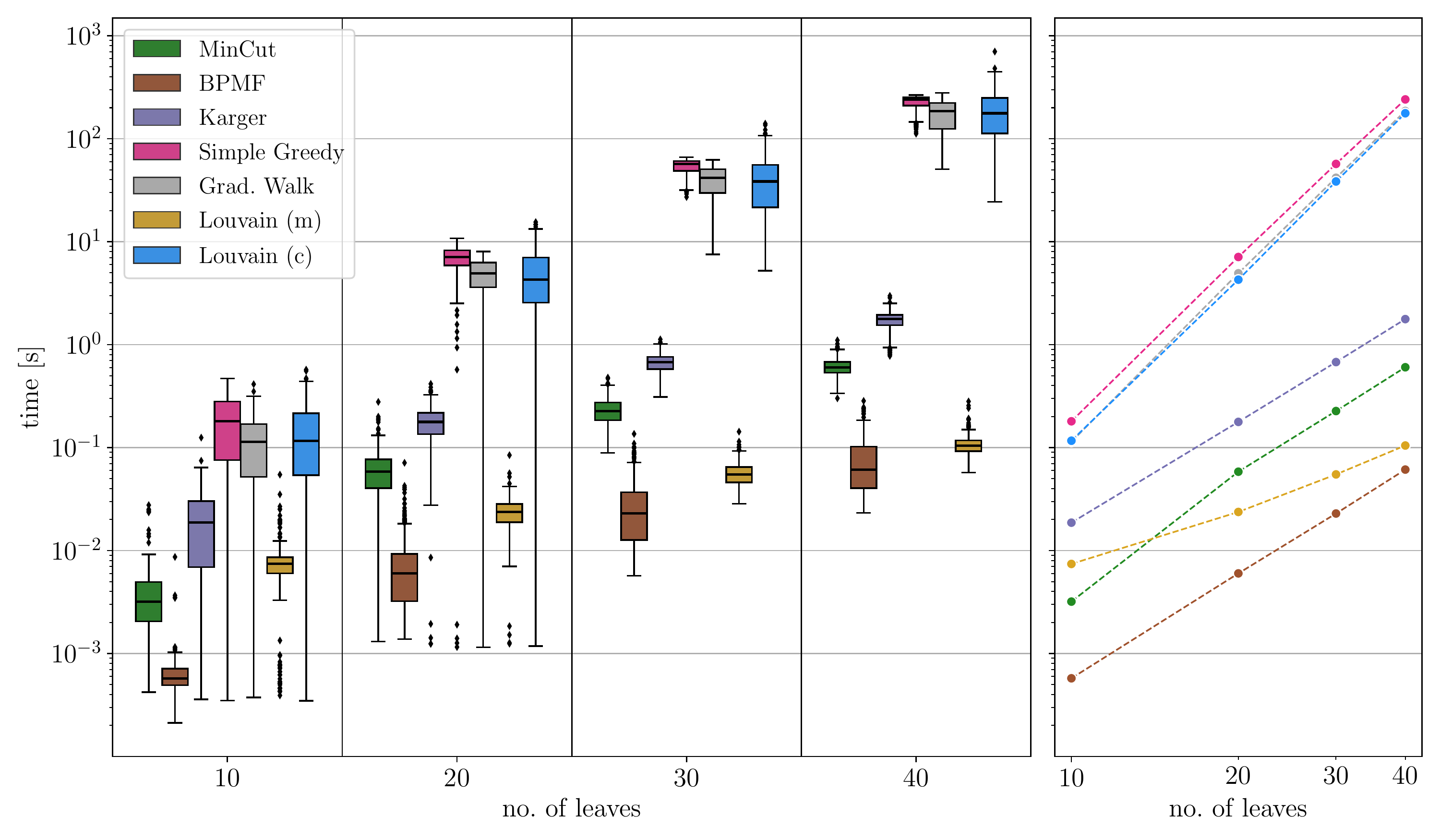}
  \caption{Running times of the different methods for BMG editing.  The
    time only includes the construction of the tree $T$, i.e.,
    Alg.~\ref{alg:general-local-optimal} or \emph{BPMF}, resp., but not the
    extraction of the triple set $\mathscr{R}^*$ followed by rerunning
    \texttt{BUILD}.  For each number of leaves $N\in\{10,20,30,40\}$ and
    each number of colors $\ell$ (taken from $\{2,5,10,20\}$ such that
    $\ell<N$), 100 perturbed BMGs were generated using equal insertion and
    deletion probabilities of $0.1$.  In the right panel, the median values
    are shown with logarithmic axes.}
  \label{fig:runtime}
\end{figure}

All methods for BMG editing were implemented and compared using Python on
an off-the-shelf laptop. Fig.~\ref{fig:runtime} summarizes the running
times. The right panel shows that all methods appear to scale polynomially
in the size $|V|$ of the vertex set of the input graph.  The methods that
explicitly rely on the \ur-cost are much slower than the other methods.  We
suspect that this is largely due to the repeated $O(|V'|^2)$-computation of
$c(\G,\mathscr{V})$ whenever a vertex is moved between the sets/communities
in $\mathscr{V}$. This could possibly be improved by an incremental
algorithm.

\subsection{Heuristics for binary-explainable BMG Editing}
\label{sect:beBMG-editing-results}

In order to test the heuristics for the slightly different task of
obtaining a binary-explainable BMG $(\G^*,\sigma)$, we constructed a
similar set of test instances. The only difference is that we ensured that
$T_\textrm{orig}$ is binary by modifying the attachment procedure (cf.\
Section \ref{subsec:exp_1}) such that in each growth step we only choose
among the vertices that are currently leaves for attaching two new
leaves. Thus, $(\G_\textrm{orig},\sigma)=\G(T_\textrm{orig},\sigma)$ is
binary-explainable.  The editing heuristics are analogous, with two
straightforward modifications:
\begin{itemize}
  \item In the Aho graphs, $\Rbin(\G,\sigma)$ is used instead of
  $\mathscr{R}(\G,\sigma)$.
  \item If we encounter a partition $\mathscr{V}$ of cardinality greater than
  two in some recursion step, we use a coarse-graining $\mathscr{V}'$ of
  $\mathscr{V}$ such that $|\mathscr{V}'|=2$ instead. This modification is
  necessary whenever $[\Rbin(\G,\sigma)[V'], V']$ itself has more than two
  connected components, and for the partitions with $|\mathscr{V}|\ge 3$
  returned by the Louvain method.
\end{itemize}

By Thm.~\ref{thm:algo-consistent-binary}, this procedure is consistent for
binary-explainable BMGs.  Thm.~\ref{thm:algo-consistent-binary}, moreover,
guarantees some freedom in the choice of a coarse-graining
$\mathscr{V}'=\{V_1, V_2\}$ whenever $\mathscr{V}$ is not a bipartition.
We therefore aim to produce (locally) balanced trees in such situations,
i.e., we seek to minimize the difference of $|V_1|$ and $|V_2|$.  Formally,
this corresponds to the well-known \PROBLEM{Number Partitioning} problem
with the multiset $\{|V_i| \,\mid\, V_i \in \mathscr{V}\}$ as input.  We
use the efficient heuristic described by \citet{Karmarkar:83}, which in
general appears to yield very good solutions of the \PROBLEM{Number
  Partitioning} problem \cite{Boettcher:08}.

To construct the second binary tree $T^*$ based on subset of triples
$\mathscr{R}^*\subseteq\Rbin(\G,\sigma)$ that are displayed by $T$, we
employ an analogous coarse-graining in an otherwise unmodified
\texttt{BUILD} algorithm.  We note, however, that one could incorporate
more sophisticated approaches which e.g.\ use some greedy coarse-graining
method based on the \ur-cost.

The results for beBMG editing in essence recapitulate the observations
for general BMG editing, see Appendix~\ref{sect:app-beBMG}:
Alg.~\ref{alg:general-local-optimal} in combination with \emph{Louvain~(c)}
appears to be the best choice for the majority of parameter combinations.
However, it is outperformed by the \emph{BPMF} heuristic at high levels of
perturbation (insertion and deletion probability $0.2$). As in the general
case, construction of $T^*$ and using
$(\G^*,\sigma)\coloneqq\G(T^*,\sigma)$ as editing result appears to be
advantageous. Moreover, the difference of the editing result
$(\G^*,\sigma)$ to the original beBMG $(\G_\textrm{orig},\sigma)$ is on
average smaller than the difference of $(\G^*,\sigma)$ to the perturbed
graph $(\G,\sigma)$.

\section{Summary and discussion}

In this contribution, we have described a large class of heuristics for BMG
editing that operate in a recursive top-down fashion to (at least
implicitly) construct a tree $(T,\sigma)$ capturing the underlying
BMG-structure of an arbitrary input graph $(\G,\sigma)$.  We have shown
that this is closely related to a specific notion of locally good edits,
which we assess using the \ur-cost.  The \ur-cost counts the minimum number
of arc insertions and deletions of the BMG-editing for $(\G,\sigma)$ that
are linked to each inner node (and thus to their corresponding leaf
partitions) in $(T,\sigma)$ and cannot be reversed in subsequent recursion
steps. In particular, we showed that an optimal solution among all possible
partitions guarantees consistency of this class of heuristics (cf.\
Thm.~\ref{thm:algo-consistent} and~\ref{thm:algo-consistent-binary}).
Unfortunately, the corresponding problem \PROBLEM{BPURC} is itself
NP-complete.

We therefore suggested a number of approximation methods for finding
suitable partitions, and compared their performances in the context of
Alg.~\ref{alg:general-local-optimal}.  We find that, even though good
solutions for \PROBLEM{(B)PURC} alone do not seem to be the most adequate
approach, the value of the \ur-costs appears most clearly in a combination
with a method for community detection, more precisely, a modification of
the Louvain method \cite{Blondel:08}.

For all of the methods investigated here, we found that the Aho graph
$H\coloneqq [\mathscr{R}(\G,\sigma)[V'], V']$ serves as a useful starting
point for finding a suitable partition.  This choice is based on the idea
that, due to the properties of BMGs and in particular the construction of
the tree $(T,\sigma)$ from informative triples of the BMG
$(\G,\sigma)=\G(T,\sigma)$, arc insertions and deletions in
$(\G,\sigma)$ should not add too many new edges between the connected
components of the originally disconnected Aho graph of
$\mathscr{R}(\G,\sigma)$ (cf.\
Fig.~\ref{fig:introduce-inconsistency}). Therefore, we suggest that there
is a correlation between good partitions $\mathscr{V}$ of $V'$, i.e.\
partitions linked to few edits, and the minimization of the number of edges
in $H$ connecting vertices in distinct sets of $\mathscr{V}$.

For the general BMG editing problem, we did not make use of the information
contained in the set of forbidden triples $\mathscr{F}(\G,\sigma)$ of the
input graph $(\G,\sigma)$. It might be possible to adapt the algorithm
\texttt{MTT} \cite{He:06}, which identifies consistent pairs
$(\mathscr{R},\mathscr{F})$, instead of \texttt{BUILD}. \texttt{MTT}
constructs a coarse-graining $\mathscr{V}_\textrm{MTT}$ of the set of
connected components of the Aho graph (on $\mathscr{R}$) in order to
account for the forbidden triples in $\mathscr{F}$ in each recursion step.
Possibly, $\mathscr{V}_\textrm{MTT}$ (or some suitable graph
representation) yields a further improvement.  However, in case of beBMG
editing, the extended triple set $\Rbin(\G,\sigma)$ and thus the
corresponding Aho graphs by construction already cover the information
contained in $\mathscr{F}(\G,\sigma)$. Since no substantial improvement over 
the general case was observed in this case (cf.\
Fig.~\ref{fig:edit-comparison-binary}), we opted against more detailed
benchmarking of $\mathscr{V}_\textrm{MTT}$ in comparison to partitions
based on the Aho graph.

The purpose of this contribution is to establish a sound theoretical
foundation for practical approaches to BMG editing and to demonstrate that
the problem can solved for interestingly large instances at reasonable
accuracy. In computational biology, however, much larger problems than
the ones considered here would also be of interest. Less emphasis has been
placed here on computational efficiency and scalability of different
variants. We leave this as topic for future research. Given the
performance advantage of community detection over minimization of the
\ur-cost in each step, it seems most promising to focus on community
detection methods that scale well for very large system. The Louvain
method seems to be a promising candidate, since it has been applied 
successfully to large networks in the past \cite{Blondel:08}. This is largely 
due to the fact that the change of modularity in response to moving a
vertex between modules can be computed efficiently. We suspect that a
comparably fast computation of the \ur-cost may also be possible; this
does not appear to be trivial, however.  Moreover, the method could
probably be accelerated by moving vertices into the community of the first
neighbor such that this results in a (not necessarily optimal) improvement
of the \ur-cost.  A similar randomization approach has already shown to
only slightly affect the clustering quality in terms of modularity
\cite{Traag:2015}.

Since the restriction of a (be)BMG to a subset of colors is again a
(be)BMG, it may also be possible to remove large parts of the noise by
editing induced subgraph on a moderate number of colors, possibly using
information of the phylogeny of the species to select species (= color)
sets. Presumably, color sets with sufficient overlaps will need to be
considered. A systematic analysis of this idea, however, depends on
scalable BMG editing for large instances and goes beyond the scope of
this contribution.

A potential shortcoming of the empirical analysis in
Sec.~\ref{sec:heur-computational-exp} is the simplistic error model, i.e.,
the independent perturbation of arcs (and non-arcs). Better models will
depend on the investigation of BMGs derived from real-life sequence data.
Such data is often burdened with systematic errors arising e.g.\ from the
fact that a common ancestry often cannot be detected for very large
evolutionary distances and from unequal mutation rates during the evolution
of gene families, see e.g.\ \cite{Rost:99,Lafond:18,Stadler:20a} for more
in-depth discussions of these issues.  Benchmarking using real-life data,
however, is a difficult task because the ground truth is unknown and large,
well-curated data sets are not available. Our results so far suggest that a
good performance w.r.t.\ the input graph is also an indicator for a good
performance w.r.t.\ the true graph (cf.\ Fig.~\ref{fig:edit_comparison1}
and~\ref{fig:edit_comparison2}, green vs.\ blue boxplots). Moreover, they
at least suggest that realistic BMG data can be processed with
sufficient accuracy and efficiency to make BMGs an attractive alternative
to classical phylogenetic methods. The construction of bioinformatics
workflows to process best hit data, e.g.\ at the first processing stage of
\texttt{ProteinOrtho} \cite{Lechner:14a}, is a logical next step.

\subsection*{Acknowledgements}
This work was supported in part by the \emph{Deutsche
Forschungs\-gemeinschaft}, the Austrian Federal Ministries
BMK and BMDW, and the Province of Upper Austria in the frame of the COMET
Programme managed by FFG.

\subsection*{Conflict of interest}

The authors declare that they have no conflict of interest.

\begin{appendix}
  
  \section{Proofs}
  
  \subsection*{Proof of Lemma~\ref{lem:unsat-rel-charac}}
  
  \begin{clemma}{\ref{lem:unsat-rel-charac}}
    Let $(\G=(V,E),\sigma)$ be a properly vertex-colored digraph and let
    $\mathscr{V}=\{V_1,\dots,V_k\}$ be a partition of $V$ with
    $|\mathscr{V}|=k\ge2$.  Then
    \begin{equation*}
      U(\G,\mathscr{V}) = U_1(\G,\mathscr{V}) \;\cupdot\;
      U_2(\G,\mathscr{V}) \;\cupdot\; U_3(\G,\mathscr{V})\,.
    \end{equation*}
  \end{clemma}
  \begin{proof}
    We first note that $U_1\coloneqq U_1(\G,\mathscr{V})$,
    $U_2\coloneqq U_2(\G,\mathscr{V})$ and $U_3\coloneqq U_3(\G,\mathscr{V})$
    are pairwise disjoint. Furthermore, we have $x\ne y$ and
    $\sigma(x)\ne\sigma(y)$ for every
    $(x,y)\in U_1 \,\cupdot\, U_2 \,\cupdot\, U_3 $ and every
    $(x,y)\in U(\G,\mathscr{V})$.  Moreover, recall that
    $\mathscr{T}(\mathscr{V})$ is the set of trees $T$ on $V$ that satisfy
    $\mathscr{V} = \{L(T(v)) \mid v\in\child_{T}(\rho_T) \}$. Therefore,
    there is a one-to-one correspondence between the $k\ge 2$ sets in
    $\mathscr{V}$ and the children $\child_{T}(\rho_{T})$ of the root
    $\rho_{T}$ for any $T\in \mathscr{T}(\mathscr{V})$.  We denote by $v_{i}$
    the child corresponding to $V_{i}\in\mathscr{V}$; thus
    $V_{i}=L(T(v_{i}))$.
    
    We first show that $(x,y)\in U_1 \,\cupdot\, U_2 \,\cupdot\, U_3$ implies
    $(x,y)\in U(\G,\mathscr{V})$. Let $T\in \mathscr{T}(\mathscr{V})$
    be chosen arbitrarily, and let $\rho$ be its root.  Suppose that
    $(x,y)\in U_1$.  Thus, we have $(x,y)\in E$,
    $\sigma(y)\in \sigma(L(T(v_{i})))$, $x\preceq_{T} v_{i}$ and
    $y\preceq_{T} v'$ for some $v'\in\child_{T}(\rho)\setminus\{v_{i}\}$.
    Moreover, $\sigma(y)\in \sigma(L(T(v_{i})))$ implies that there is a
    vertex $y'\preceq_{T}v_{i}$ with $\sigma(y')=\sigma(y)$.  Taken together,
    we obtain $\lca_{T}(x,y')\preceq_T v_{i} \prec_T \rho = \lca_{T}(x,y)$,
    and thus $(x,y)\notin E(\G(T,\sigma))$.  If $(x,y)\in U_2$, we have
    $(x,y)\notin E$, $\sigma(y)\notin \sigma(L(T(v_{i})))$,
    $x\preceq_{T}v_{i}$ and $y\preceq_{T} v'$ for some
    $v'\in\child_{T}(\rho)\setminus\{v_{i}\}$.  Moreover,
    $\sigma(y)\notin \sigma(L(T(v_{i})))$ implies that there is no vertex
    $y'\preceq_{T}v_{i}$ with $\sigma(y')=\sigma(y)$.  Thus,
    $\lca_{T}(x,y') = \lca_{T}(x,y)= \rho$ holds for all $y'$ of color
    $\sigma(y')=\sigma(y)$, and thus $(x,y)\in E(\G(T,\sigma))$.  Finally,
    suppose $(x,y)\in U_3$. We have $(x,y)\notin E$ and $y$ is the only leaf
    of its color in $L(T(v_{i}))$.  Therefore, there is no vertex $y'$ with
    $\sigma(y')=\sigma(y)$ and $\lca_{T}(x,y') \prec_{T} \lca_{T}(x,y)$, and
    thus $(x,y)\in E(\G(T,\sigma))$.  In summary, one of the conditions in
    Def.~\ref{def:unsat-relations} is satisfied for $T$ in all three cases.
    Since $T$ was chosen arbitrarily, we conclude
    $(x,y)\in U(\G,\mathscr{V})$ for any
    $(x,y)\in U_1 \,\cupdot\, U_2 \,\cupdot\, U_3$.
    
    In order to show that $(x,y)\in U(\G,\mathscr{V})$ implies
    $(x,y)\in U_1 \,\cupdot\, U_2 \,\cupdot\, U_3$, we distinguish
    \textit{Case (a)}: $(x,y)\in E$ and $(x,y)\notin E(\G(T,\sigma))$ holds
    for all $T\in\mathscr{T}(\mathscr{V})$, and \textit{Case (b)}:
    $(x,y)\notin E$ and $(x,y)\in E(\G(T,\sigma))$ holds for all
    $T\in\mathscr{T}(\mathscr{V})$.
    \par\noindent\textit{Case (a).}
    $(x,y)\in E$ implies $\sigma(x)\ne\sigma(y)$. Moreover,
    there is a vertex $y'$ with $\sigma(y')=\sigma(y)$, and
    $\lca_{T}(x,y')\prec_{T}\lca_{T}(x,y)$ for every
    $T\in\mathscr{T}(\mathscr{V})$ because $(x,y)\notin E(\G(T,\sigma))$.
    Since this is true for all trees in $\mathscr{T}(\mathscr{V})$, there
    must be a set $V_{i}\in \mathscr{V}$ such that $x, y'\in V_{i}$, and in
    particular $\sigma(y')=\sigma(y)\in\sigma(V_{i})$.  Now suppose, for
    contradiction, that $y\in V_{i}$ and thus $x,y,y'\in V_{i}$.  In this
    case, we can choose a tree $T\in\mathscr{T}(\mathscr{V})$ such that
    $x,y\prec_{T}v_{i}$ for some child $v_{i}\in\child_{T}(\rho_T)$ and
    $\lca_{T}(x,y)\preceq\lca_{T}(x,y')$ hold for all $y'$ of color
    $\sigma(y')=\sigma(y)$. Hence, we obtain $(x,y)\in E(\G(T,\sigma))$ for
    this tree; a contradiction. Therefore, we conclude that
    $y\in V\setminus V_{i}$.  In summary, all conditions for $U_1$ are
    satisfied, and thus $(x,y)\in U_1$.
    \par\noindent\textit{Case (b).} We have $(x,y)\notin E$ and
    $(x,y)\in E(\G(T,\sigma))$ for all $T\in\mathscr{T}(\mathscr{V})$.  Let
    $V_{i}\in \mathscr{V}$ such that $x\in V_{i}$.  We distinguish the two
    cases (i) $y\notin V_{i}$, and (ii) $y\in V_{i}$.  In Case~(i), suppose,
    for contradiction, that $\sigma(y)\in\sigma(V_{i})$.  Then, for every
    tree $T\in\mathscr{T}(\mathscr{V})$, there must be a vertex $y'$ of color
    $\sigma(y)$ such that
    $\lca_{T}(x,y')\preceq_{T} v_{i}\prec_{T}\rho_T=\lca_{T}(x,y)$,
    contradicting $(x,y)\in E(\G(T,\sigma))$.  Therefore, we conclude
    $\sigma(y)\notin\sigma(V_{i})$.  It follows that $(x,y)\in U_2$.  In
    Case~(ii), assume, for contradiction, that there is a vertex
    $y'\in V_{i}$ of color $\sigma(y)$ such that $y\ne y'$. This together
    with $\sigma(y')=\sigma(y)\ne\sigma(x)$ implies that all three vertices
    $x,y,y'$ are pairwise distinct.  Since in addition $x,y,y'\in V_{i}$, we
    can choose a tree $T\in\mathscr{T}(\mathscr{V})$ such that
    $x,y,y'\prec_{T}v_{i}$ for some child $v_{i}\in\child_{T}(\rho_T)$ and
    $\lca_{T}(x,y')\prec_{T}\lca_{T}(x,y)$; a contradiction to
    $(x,y)\in E(\G(T,\sigma))$ for all $T\in\mathscr{T}(\mathscr{V})$.
    Therefore, we conclude that $y$ is the only vertex of its color in
    $V_{i}$.  It follows that $(x,y)\in U_3$.  \qed
  \end{proof}

  \subsection*{Proof of Lemma~\ref{lem:independent-U}}
  
  In order to prove Lemma~\ref{lem:independent-U}, we first need the
  following technical result which shows that the editing of an arc in
  Alg.~\ref{alg:general-local-optimal} will not be reversed in the subsequent
  recursion step.
  
  \begin{lemma}
    \label{lem:hierarchical-edits-independent}
    Let $(\G=(V,E),\sigma)$ be a properly vertex-colored digraph,
    $\mathscr{V}=\{V_1,\dots,V_k\}$ a partition of $V$ with
    $|\mathscr{V}|=k\ge2$, and $\mathscr{V}_i=\{V_{i,1},\dots,V_{i,l}\}$,
    $1\le i\le k$, a partition of $V_i$ with
    $|\mathscr{V}_i|=l\ge2$. Moreover, let
    $(\G'\coloneqq \G \triangle U(\G,\mathscr{V}),\sigma)$ be the colored 
    digraph
    that is obtained by applying the edits in $U(\G,\mathscr{V})$ to
    $(\G,\sigma)$.  Then
    $U(\G,\mathscr{V})\cap U(\G'[V_i],\mathscr{V}_i)=\emptyset$.
  \end{lemma}
  \begin{proof}
    Let $\G'_i\coloneqq \G'[V_i]$.  The sets of unsatisfiable relations
    $U(\G,\mathscr{V})$ and $U(\G'_i, \mathscr{V}_i)$ are given by the
    (disjoint) unions
    $U_1(\G,\mathscr{V})\cupdot U_2(\G,\mathscr{V})\cupdot
    U_3(\G,\mathscr{V})$ and
    $U_1(\G'_i,\mathscr{V}_i)\cupdot U_2(\G'_i,\mathscr{V}_i)\cupdot
    U_3(\G'_i,\mathscr{V}_i)$, respectively (cf.\
    Lemma~\ref{lem:unsat-rel-charac}). First, let
    $(x,y)\in U_1(\G,\mathscr{V})$.  Since, by definition of
    $U_1(\G,\mathscr{V})$, $x$ and $y$ are contained in different sets of the
    partition $\mathscr{V}$, they cannot be both contained in $V_i$ and thus,
    $U_1(\G,\mathscr{V}) \cap U(\G'_i, \mathscr{V}_i)=\emptyset$. One
    analogously argues that
    $U_2(\G,\mathscr{V}) \cap U(\G'_i, \mathscr{V}_i)=\emptyset$. Now, assume
    for contradiction that
    $(x,y)\in U_3(\G,\mathscr{V})\cap U(\G'_i, \mathscr{V}_i)$.  By
    definition of $U(\G'_i, \mathscr{V}_i)$, this implies $x,y\in
    V_i$. Moreover, by definition of $U_3(\G,\mathscr{V})$, we have
    $(x,y)\notin E$, which immediately implies $(x,y)\in E(\G'_i)$. By
    Lemma~\ref{lem:unsat-rel-charac}, we therefore conclude
    $(x,y)\in U_1(\G'_i,\mathscr{V}_i)$. Let $V_{i,j}$ be the set of the
    partition $\mathscr{V}_i$ which contains $x$. Then, by definition of
    $U_1(\G'_i,\mathscr{V}_i)$, the color of $y$ is contained in both
    $V_{i,j}$ and $V_i\setminus V_{i,j}$, i.e., $V_i$ contains at least two
    vertices of color $\sigma(y)$. However, $(x,y)\in U_3(\G,\mathscr{V})$
    and $y\in V_i$ together imply that $y$ is the only vertex of its color in
    $V_i$; a contradiction.  \qed
  \end{proof}
  
  We are now in the position to prove the more general
  \begin{clemma}{\ref{lem:independent-U}}
    All edit sets $U(\G^*[V'],\mathscr{V})$ constructed in
    Alg.~\ref{alg:general-local-optimal} are pairwise disjoint.
  \end{clemma}
  \begin{proof}
    First note that, by Lemma~\ref{lem:unsat-rel-charac}, we have
    $\sigma(x)\ne\sigma(y)$ for all $(x,y)\in U(\G^*[V'],
    \mathscr{V})$. Hence, the graph $(\G^*,\sigma)$ remains properly colored
    during the whole recursion.  Moreover, recursive calls on a set $V'$ with
    $|V'|=1$ trivially contribute with a \ur-cost of zero.
    
    By construction, the partitions in consecutive calls of \texttt{Edit()}
    form a hierarchical refinement such that in each recursive call a single
    element of $V_i\in\mathscr{V}$ is refined.  Clearly edit sets encountered
    in independent branches of the recursion tree are disjoint because they
    pertain to disjoint vertex sets. For directly consecutive calls of
    \texttt{Edit()}, Lemma~\ref{lem:hierarchical-edits-independent} states
    that the edits sets are disjoint.  Now consider two recursive call on
    $V'$ and $V''$ with $V''\subset V'$ that are not directly consecutive.
    Let $\mathscr{V}'$ and $\mathscr{V}''$, resp., be the partitions chosen
    for the vertex sets $V'$ and $V''$ of $\G'$ and $\G''$ at the beginning
    of the two recursion steps.  We can apply the same arguments as in the
    proof of Lemma~\ref{lem:hierarchical-edits-independent} to conclude that
    $U_i(\G'[V'],\mathscr{V}') \cap U(\G''[V''],\mathscr{V}'')=\emptyset$,
    $i\in \{1,2\}$.  Finally, assume, for contradiction, that
    $(x,y)\in U_3(\G'[V'],\mathscr{V}') \cap U(\G''[V''],\mathscr{V}'')$.  By
    definition of $U(\G''[V''], \mathscr{V}'')$, this implies $x,y\in
    V''$. Moreover, by definition of $U_3(\G'[V'],\mathscr{V}')$, we have
    $(x,y)\notin E(\G')$, which immediately implies
    $(x,y)\in E(\G'\triangle U(\G'[V'],\mathscr{V}'))$, i.e., $(x,y)$ is an
    arc after the editing in this step. Since both $x,y$ are contained in
    $V''$, it follows from Lemma~\ref{lem:unsat-rel-charac} that all edit
    steps on the way from $\G'[V'']$ to $\G''[V'']$ must be performed by the
    set $U_3$, i.e., they exclusively correspond to arc
    insertions. Therefore, $(x,y)$ is still an arc in $\G''[V'']$.  By
    Lemma~\ref{lem:unsat-rel-charac}, we therefore conclude that
    $(x,y)\in U_1(\G''[V''],\mathscr{V}'')$. Let $V_{x}$ be the set of the
    partition $\mathscr{V}''$ that contains $x$. Then, by definition of
    $U_1(\G''[V''],\mathscr{V}'')$, the color of $y$ is contained in both
    $V_{x}$ and $V''\setminus V_{x}$, i.e., $V''$ contains at least two
    vertices of color $\sigma(y)$. However,
    $(x,y)\in U_3(\G'[V'],\mathscr{V}')$ and $y\in V''\subset V_{x,y}$ for
    some $V_{x,y}\in \mathscr{V}'$ together imply that $y$ is the only vertex
    of its color in $V''$; a contradiction.  \qed
  \end{proof}
  
  \subsection*{Proof of Theorem~\ref{thm:BPURC-NPc}}
  
  In this section we show that \PROBLEM{(B)PURC} is NP-hard by reduction from
  \PROBLEM{Set Splitting}.
  
  \begin{problem}[\PROBLEM{Set Splitting}]\ \\
    \begin{tabular}{ll}
      \emph{Input:}    & A collection $\mathfrak{C}$ of subsets of a finite set 
      $S$, denoted by $(\mathfrak{C},S)$.\\
      \emph{Question:} & Is there a bipartition of $S$ into two subsets $S_1$ 
      and 
      $S_2$ such that\\
      &no subset in $\mathfrak{C}$ is entirely contained in either $S_1$ or 
      $S_2$?
    \end{tabular}
  \end{problem}
  
  \begin{proposition}{\cite{Lovasz:73}}
    \label{prop:set-splitting-NPc}
    \PROBLEM{Set Splitting} is NP-complete.
  \end{proposition}
  
  \begin{ctheorem}{\ref{thm:BPURC-NPc}}
    \PROBLEM{BPURC} is NP-complete.
  \end{ctheorem}
  \begin{proof}
    Given a properly vertex-colored digraph $(\G=(V,E),\sigma)$ and a
    bipartition $\mathscr{V}$ of $V$, the set $U(\G,\mathscr{V})$ and thus
    the \ur-cost $c(\G,\mathscr{V})=|U(\G,\mathscr{V})|$ can be computed in
    polynomial time according to Cor.~\ref{cor:U-polytime}.  Therefore,
    \PROBLEM{BPURC} is contained in NP.  To show NP-hardness, we use reduction
    from \PROBLEM{Set Splitting}.
    
    Let $(\mathfrak{C},S)$ be an instance of \PROBLEM{Set Splitting}. We may
    assume w.l.o.g.\ that $|C|\ge 2$ holds for all $C\in \mathfrak{C}$, since
    otherwise there is no solution at all for \PROBLEM{Set Splitting}. In
    addition, we assume that $\bigcup_{C\in \mathfrak{C}}C = S$. To see that
    this does not yield a loss of generality, suppose that
    $\bigcup_{C\in \mathfrak{C}}C =S' \subsetneq S$. If $\{S'_1,S'_2\}$ is a
    solution for $(\mathfrak{C},S')$ then no subset in $\mathfrak{C}$ is
    entirely contained in either $S'_1$ or $S'_2$. Therefore, we can
    construct a solution $(S_1,S_2)$ for $(\mathfrak{C},S)$ by arbitrarily
    adding the elements in $S\setminus S'$ to either $S'_1$ or $S'_2$.  In
    contrast, $\{S_1\cap S',S_2\cap S'\}$ is a solution for
    $(\mathfrak{C},S')$ provided that $\{S_1,S_2\}$ is a solution for
    $(\mathfrak{C},S)$.
    
    Now, let $(\mathfrak{C},S)$ be an instance of \PROBLEM{Set Splitting} and
    define, for all $s\in S$, the set
    $\mathfrak{C}(s) \coloneqq \{C\mid C\in \mathfrak{C}, s\in C\}$ as the
    subset of $\mathfrak{C}$ that comprises all elements $C\in \mathfrak{C}$
    that contain $s$.  Note that $\mathfrak{C}(s) \neq \emptyset$ for all
    $s\in S$, since we have assumed $\bigcup_{C\in \mathfrak{C}}C = S$, i.e.,
    every $s\in S$ is contained in some element of $\mathfrak{C}$.
    
    We construct a graph $(\G=(V,E),\sigma)$ that serves as input for
    \PROBLEM{BPURC} as follows:
    \begin{description}
      \item[\emph{Step 1:}] For all $s\in S$, construct an $s$-gadget
      $\mathscr{G}_s$ as follows:
      \begin{enumerate}
        \item[(i)] For all $C\in \mathfrak{C}(s)$, add four new vertices to
        $\mathscr{G}_s$ of which two are colored with $(C,1)$ and the other
        two with $(C,2)$.
        \item[(ii)] Add arcs $(x,y),(y,x)$ between all $x,y\in V(\mathscr{G}_s)$
        with $\sigma(x)\neq \sigma(y)$.
      \end{enumerate} \smallskip
      \item[\emph{Step 2:}] Set
      $V \coloneqq \bigcupdot_{s\in S}\; V(\mathscr{G}_s)$,
      $E \coloneqq \bigcupdot_{s\in S}\; E(\mathscr{G}_s)$ and preserve the
      coloring of the vertices within the $s$-gadgets to obtain the graph
      $(\G,\sigma)$.
    \end{description}
    By construction, $|V(\mathscr{G}_s)| = 4|\mathfrak{C}(s)|$ and
    $|V| = \sum_{s\in S}\; 4|\mathfrak{C}(s)|\le 4|\mathfrak{C}||S|$.  Hence,
    the construction of $(\G,\sigma)$ can be achieved in polynomial time.
    Moreover, by construction, $\{V(\mathscr{G}_s)\mid s\in S\}$ forms a
    partition of $V$ and there are no arcs between vertices of distinct
    $s$-gadgets.  Furthermore, $\sigma(V) = \mathfrak{C}\times \{1,2\}$.  An
    illustrative example of such a constructed graph $(\G,\sigma)$ is
    provided in Fig.~\ref{fig:BPURC-reduction}.
    
    \begin{figure}[htb]
      \centering
      \includegraphics[width=0.85\textwidth]{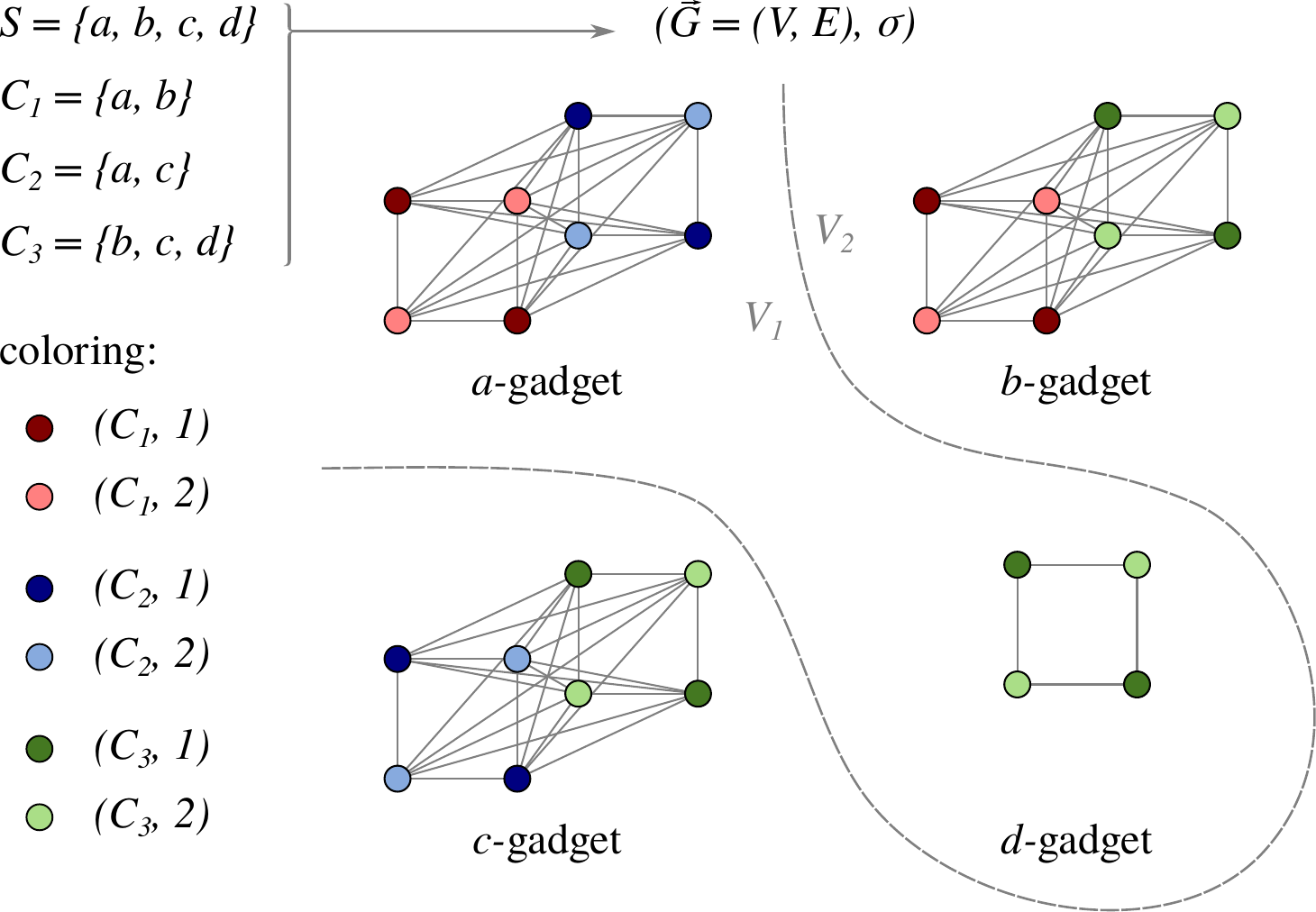}
      \caption[]{Example for the reduction from an instance
        $(\mathfrak{C},S)$ of \PROBLEM{Set Splitting} to an instance
        $(\G,\sigma)$ with $k=0$ of \PROBLEM{BPURC}, as specified in the proof
        of Thm.~\ref{thm:BPURC-NPc}.  In this example, we have $S=\{a,b,c,d\}$
        and $\mathfrak{C}=\{C_1,C_2,C_3\}$.  By construction, all arcs are
        bidirectional and thus, arrow heads are omitted in the drawing of
        $(\G,\sigma)$. A solution for $(\mathfrak{C},S)$ is $S_1=\{a,d\}$ and
        $S_2=\{b,c\}$. The latter is equivalent to a solution of
        \PROBLEM{BPURC} by ``separating'' the $a$- and $d$-gadget from the
        $b$- and $c$-gadget as indicated by the dashed line. The latter
        yields a bipartition $\mathscr{V}=\{V_1,V_2\}$ of $V(\G)$
        that solves \PROBLEM{BPURC} with input $(\G,\sigma,k=0)$.\\
        Note, slight changes of the input $(\mathfrak{C},S)$ to
        $S'=S\setminus\{d\}$ and
        $\mathfrak{C}'=\{C_1,C_2,C_3\setminus \{d\}\}$ would yield an
        instance of \PROBLEM{Set Splitting} that has no yes-answer. In this
        case, the $d$-gadget would disappear from $(\G,\sigma)$ resulting in
        the graph $(\G',\sigma')$.  It is easy to see that there is no
        bipartition $\mathscr{V}=\{V_1,V_2\}$ of $V(\G')$ such that
        $\sigma(V_1)=\sigma(V_2)=\sigma(V(\G'))$ and no gadget gets split up
        between $V_1$ and $V_2$; two necessary properties to obtain a
        solution for \PROBLEM{BPURC} with input $(\G',\sigma')$ and $k=0$
        (cf.\ proof of Thm.~\ref{thm:BPURC-NPc}).}
      \label{fig:BPURC-reduction}
    \end{figure}
    
    We continue by showing that an instance $(\mathfrak{C},S)$ of
    \PROBLEM{Set Splitting} has a yes-answer if and only if \PROBLEM{BPURC}
    has a yes-answer for the input graph $(\G=(V,E),\sigma)$ constructed
    above and $k=0$.  In particular, we will show that $\{S_1,S_2\}$ is a
    solution of $(\mathfrak{C},S)$ if and only if $\mathscr{V} = \{V_1,V_2\}$
    with $V_i=\cup_{s\in S_i} V(\mathscr{G}_s)$, $i\in \{1,2\}$ is a solution
    for $(\G,\sigma)$ where $c(\G,\mathscr{V}) = 0$.
    
    Recall that the set of unsatisfiable relations $U(\G,\mathscr{V})$ of a
    bipartition $\mathscr{V}$ of $V$ is given by the (disjoint) union
    $U_1\cupdot U_2\cupdot U_3$ of the three sets
    $U_1 \coloneqq U_1(\G,\mathscr{V})$, $U_2 \coloneqq U_2(\G,\mathscr{V})$
    and $U_3 \coloneqq U_3(\G,\mathscr{V})$ (cf.\
    Lemma~\ref{lem:unsat-rel-charac}).
    
    First suppose that \PROBLEM{Set Splitting} with input $(\mathfrak{C},S)$
    has a yes-answer and let $\{S_1,S_2\}$ be one of its solutions.  Hence,
    no subset in $\mathfrak{C}$ is entirely contained in either $S_1$ or
    $S_2$, and both sets must be non-empty.  Consider the set
    $\mathscr{V} = \{V_1,V_2\}$ with
    $V_i=\cup_{s\in S_i}\, V(\mathscr{G}_s)$, $i\in \{1,2\}$.  Since
    $\{S_1,S_2\}$ is a bipartition of $S$ and
    $\{V(\mathscr{G}_s)\mid s\in S\}$ is a partition of $V$, we conclude that
    $\mathscr{V}$ is a bipartition of $V$ and that $V(\mathscr{G}_s)$ is
    entirely contained in either $V_1$ or $V_2$ for all $s\in S$.  Together
    with the fact that there are no arcs in $\G$ between vertices of distinct
    $s$-gadgets this implies that $U_1 =\emptyset$.
    
    In order to verify that $U_2 = U_3 = \emptyset$, we first show that
    $\sigma(V_1)=\sigma(V_2) = \sigma(V)$ and that $V_1$ and $V_2$ contain at
    least two vertices of every color, respectively. Consider two arbitrary
    pairs $(C,1), (C,2)\in \sigma(V)=\mathfrak{C}\times\{1,2\}$.  Since
    $\{S_1,S_2\}$ is a solution for \PROBLEM{Set Splitting} with input
    $(\mathfrak{C},S)$, there are vertices $s\in C\cap S_1$ and
    $s'\in C\cap S_2$ and thus, $V(\mathscr{G}_s)\subseteq V_1$ and
    $V(\mathscr{G}_{s'})\subseteq V_2$.  By construction, each of the sets
    $V(\mathscr{G}_s)$ and $V(\mathscr{G}_{s'})$ contains two vertices of
    color $(C,1)$ and two vertices of color $(C,2)$.  Since
    $V(\mathscr{G}_s)\subseteq V_1$ and $V(\mathscr{G}_{s'})\subseteq V_2$,
    the sets $V_1$ and $V_2$ each contain two vertices of both colors $(C,1)$
    and $(C,2)$.  Since $(C,1), (C,2)\in \sigma(V)$ are arbitrary and
    $\sigma(V)= \mathfrak{C}\times\{1,2\}$, we can conclude that
    $\sigma(V_1)=\sigma(V_2) = \sigma(V)$, and that $V_1$ and $V_2$ contain
    at least two vertices of every color.  Now, $\sigma(V_1)=\sigma(V_2)$
    implies that $U_2=\emptyset$. Moreover, since $V_1$ and $V_2$ contain at
    least two vertices of every color, we also have that $U_3=\emptyset$.  In
    summary, we have
    $U(\G,\mathscr{V})=U_1\cupdot U_2\cupdot U_3 = \emptyset$, and thus,
    $c(\G,\mathscr{V})=0$.  Therefore, \PROBLEM{BPURC} with input
    $(\G,\sigma,k=0)$ has a yes-answer.
    
    Now suppose \PROBLEM{BPURC} with input $(\G,\sigma,k=0)$ has a yes-answer
    and thus, a solution $\mathscr{V}=\{V_1,V_2\}$.  Consequently,
    $U(\G,\mathscr{V})=U_1\cupdot U_2\cupdot U_3 =\emptyset$.  We first show
    that both $V_1$ and $V_2$ must contain a vertex of every color in
    $\sigma(V) = \mathfrak{C}\times\{1,2\}$.  To this end, we assume for
    contradiction that w.l.o.g.\ $V_1$ contains no vertex of color $(C,1)$
    for some $C\in\mathfrak{C}$.  Since $|C|\ge 2$, $C$ contains two distinct
    elements $s, s'\in S$.  Note that $C\in \mathfrak{C}(s)$ and
    $C\in \mathfrak{C}(s')$.  By construction in Step 1, there are vertices
    $y\in V(\mathscr{G}_s)$ and $y'\in V(\mathscr{G}_{s}')$ of color
    $\sigma(y)=\sigma(y')=(C,1)$.  Since $(C,1)\notin\sigma(V_1)$, it must
    hold that $y,y'\in V_2$.  Now consider an arbitrary vertex $x\in V_1$.
    Note that $(C,1)\notin\sigma(V_1)$ implies $\sigma(x)\ne(C,1)$.  Since
    $\mathscr{G}_s$ and $\mathscr{G}_{s'}$ are, by construction, vertex
    disjoint, $x$ cannot belong two both gadgets $\mathscr G_s$ and
    $\mathscr{G}_{s'}$.  Therefore, we can choose $\tilde{y}\in\{y,y'\}$ such
    that $x$ and $\tilde{y}$ belong to distinct gadgets, and we obtain
    $(x,\tilde{y})\notin E$ by construction.  This together with $x\in V_1$,
    $\tilde{y}\in V_2=V\setminus V_1$ and
    $\sigma(\tilde{y})=(C,1)\notin\sigma(V_1)$ implies
    $(x,\tilde{y})\in U_2$. Hence, $U_2\neq \emptyset$; a contradiction.
    Therefore, we conclude that both $V_1$ and $V_2$ contain vertices of all
    colors in $\sigma(V)=\mathfrak{C}\times\{1,2\}$.
    
    We continue by showing that $V(\mathscr{G}_s)$ is entirely contained in
    either $V_1$ or $V_2$ for all $s\in S$.  To this end, assume for
    contradiction that there is a gadget $\mathscr{G}_s$ such that
    $W_1\coloneqq V_1\cap V(\mathscr{G}_s)$ and
    $W_2\coloneqq V_2\cap V(\mathscr{G}_s)$ are both non-empty.  Since
    $V(\mathscr{G}_s)$ forms a connected component in $(\G,\sigma)$ and all
    arcs are bidirectional by construction, we can find two vertices
    $x\in W_1$ and $y\in W_2$ such that $(x,y)\in E$.  This together with the
    facts that $x$ and $y$ are in distinct sets $V_1$ and $V_2$ and that both
    $V_1$ and $V_2$ contain all colors of $\sigma(V)$, implies that
    $(x,y)\in U_1$.  Hence, $U_1\neq \emptyset$; a contradiction.  Therefore,
    the vertex set of each $s$-gadget is entirely contained in either $V_1$
    or $V_2$.
    
    We can construct a well-defined partition $\{S_1,S_2\}$ of $S$ such that
    $s\in S_i$ if and only if $V(\mathscr G_s)\subseteq V_{i}$,
    $i\in\{1,2\}$.  By construction, there are vertices of color $(C,1)$ and
    $(C,2)$ in $\mathscr{G}_s$ if and only if $s\in C$.  This together with
    the fact that both $V_1$ and $V_2$ contain vertices of all colors
    $\mathfrak{C}\times\{1,2\}$ implies that $S_1\cap C$ and $S_2\cap C$ are
    both non-empty for every $C\in\mathfrak{C}$.  Hence, $\{S_1, S_2\}$ is a
    solution for \PROBLEM{Set Splitting} with input $(\mathfrak{C},S)$.  \qed
  \end{proof}
  
  \clearpage
  
  \section{BPMF is not a consistent heuristic for \PROBLEM{MaxRTC}} 
  \label{sect:app-BPMF}
  
  The example in Fig.~\ref{fig:BPMF-not-consistent} shows that 
  Alg.~\ref{alg:simple} in combination with the Best-Pair-Merge-First (BPMF) 
  heuristic \cite{Wu:04,Byrka:10} is not a consistent heuristic for BMG editing.
  In particular, BPMF is not consistent for \PROBLEM{MaxRTC}.
  The input graph $(\G_\textrm{orig}=(V,E),\sigma)$ is a BMG and explained by 
  $(T_\textrm{orig},\sigma)$.
  Therefore, its set of informative triples 
  \begin{equation*}
    \mathscr{R}\coloneqq 
    \mathscr{R}(\G_\textrm{orig},\sigma) =
    \{ ab_1|b_2,\, ac_1|c_2,\, ac_1|c_3,\, 
    b_1c_1|c_2,\, b_1c_1|c_3,\, b_2c_1|c_2,\, b_2c_1|c_3,\, c_1b_2|b_1\}
  \end{equation*}
  is consistent (cf.\ Prop.~\ref{prop:BMG-charac}).  On the right-hand side
  of Fig.~\ref{fig:BPMF-not-consistent}, the first three cluster merging
  steps in BPMF with input $\mathscr{R}$ are shown where the numbers are the
  scores $score(S_i, S_j)$ for each pair of clusters $S_i$ and $S_j$ as defined
  in \cite{Byrka:10}.  The pink arrows link inner vertices of the resulting
  binary tree $(T,\sigma)$ and the corresponding cluster merging step based
  on the maximal score.  The tree $(T,\sigma)$ does not display the triple
  $ab_1|b_2$.  As a consequence, its BMG $\G(T,\sigma)$ contains the
  additional arc $(a,b_2)$, and the triple set $\mathscr{R}^*$ extracted from
  $T$ in Alg.~\ref{alg:simple} is a proper subset of $\mathscr{R}$.  In
  particular, the final editing result $\G(T^*,\sigma)$ with
  $T^*=\Aho(\mathscr{R}^*,V)$ also contains the arc $(a,b_2)$ which was
  not present in the original BMG.
  
  \begin{figure}[htb]
    \centering
    \includegraphics[width=0.85\textwidth]{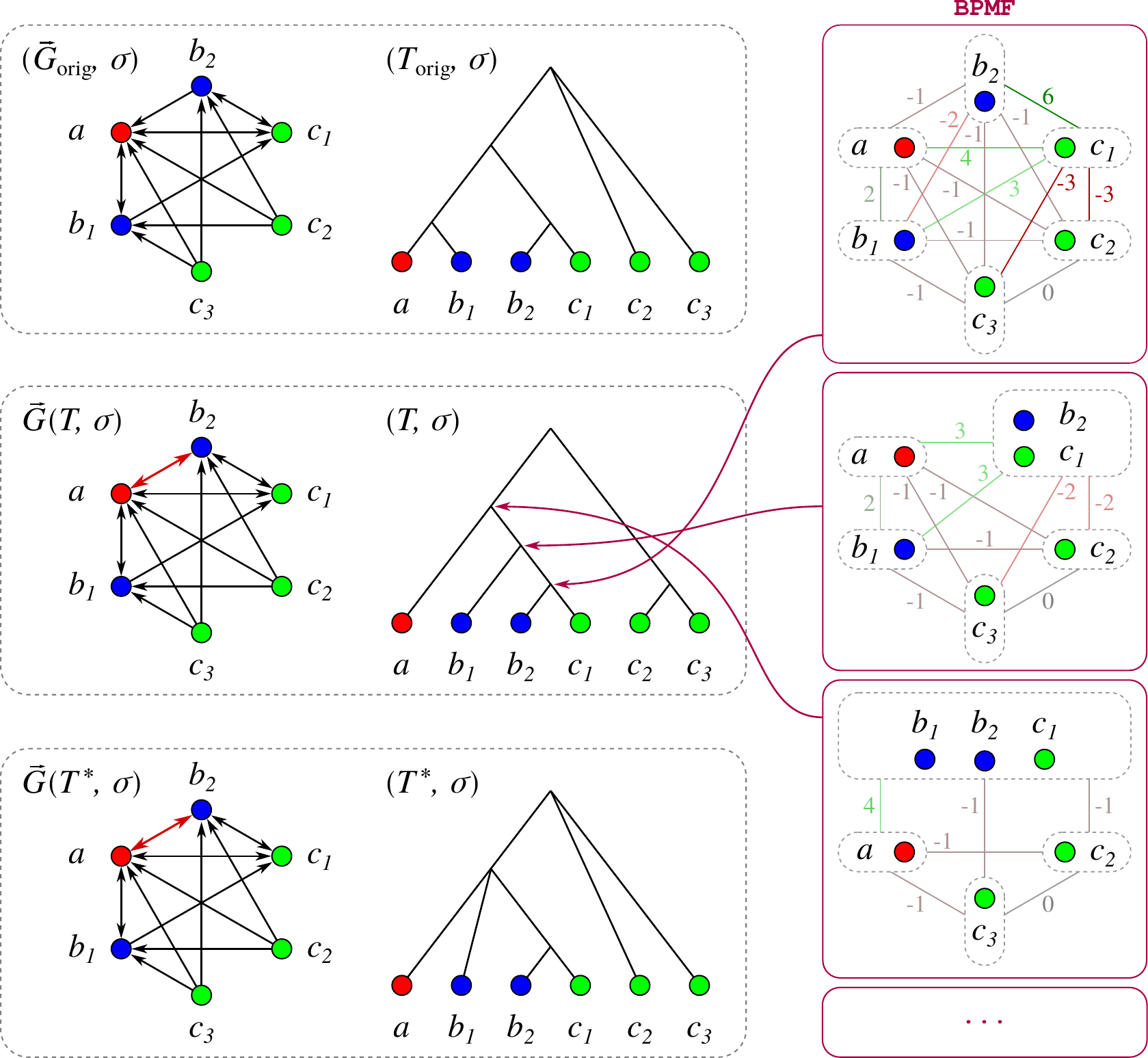}
    \caption{Example showing that BPMF is not a consistent heuristic for
      \PROBLEM{MaxRTC}, and that Alg.~\ref{alg:simple} with BPMF is not a
      consistent heuristic for BMG editing. See the text for a detailed
      description.}
    \label{fig:BPMF-not-consistent}
  \end{figure}
  
  \clearpage
  
  \section{Analysis of single-leaf splits} 
  \label{sect:app-single}
  
  Fig.~\ref{fig:single-vertex-cut} quantifies the abundance of single-leaf
  splits on the same instances as in Fig.~\ref{fig:part-quality}. We
  distinguish between single-leaf splits that are correct w.r.t.\ the Aho
  graph $H_\textrm{orig}$ of the orginal unperturbed graph, and single-leaf
  splits that are not present in the unperturbed target. MinCut, Karger,
  Simple Greedy and Gradient Walk frequently produce single-leaf splits that
  are not present in $H_\textrm{orig}$. The modularity-based Louvain method,
  in contrast, never returned a single-leaf split, even if it was present in
  $H_\textrm{orig}$. The modified Louvain method is most often in good
  agreement with $H_\textrm{orig}$ as far as single-leaf splits are
  concerned, at least for pertubation levels of 10\% of insertions and
  deletions.
  \begin{figure}[htb]
    \centering
    \includegraphics[width=0.85\textwidth]{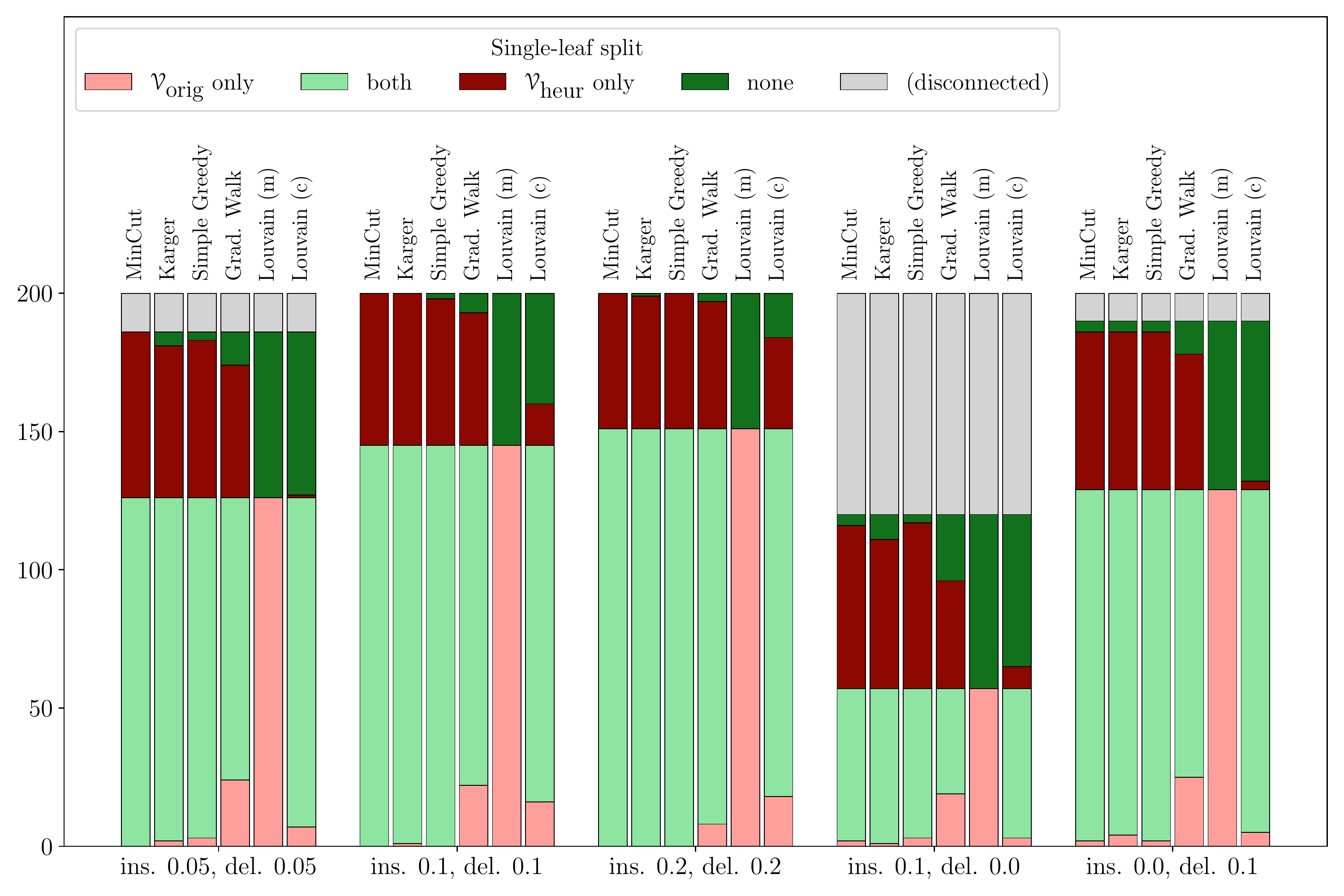}
    \caption{Abundance of single-leaf splits for pairs of BMGs
      $(\G_{\textrm{orig}},\sigma)$ and disturbed graphs $(\G,\sigma)$ (both
      with vertex set $V$).  The partition $\mathscr{V}_{\textrm{orig}}$
      corresponds to the connected components of the Aho graph
      $H_\textrm{orig}\coloneqq [\mathscr{R}(\G_{\textrm{orig}},\sigma), V]$
      and, hence, to the partition induced by the subtrees of the children of
      the root of the LRT $(T,\sigma)$ of $(\G_{\textrm{orig}},\sigma)$ (cf.\
      Prop.~\ref{prop:BMG-charac}).  The partition $\mathscr{V}_\textrm{heur}$
      corresponds to the partition of $V$ as determined by one of the
      partitioning methods (based on $H\coloneqq[\mathscr{R}(\G,\sigma), V]$).
      The gray parts of the bars comprise those instances for which $H$ is
      disconnected. The light and dark red bars indicate the amount of graphs
      for which only $\mathscr{V}_{\textrm{orig}}$ or
      $\mathscr{V}_\textrm{heur}$, resp., is a single-leaf split, while light
      and dark green bars represent instances for which both and none of the
      two partitions, resp., are single-leaf splits.  Note that the partitions
      were not compared explicitly, in particular, the identified singletons
      in $\mathscr{V}_\textrm{heur}$ in the light green instances may deviate
      from those in $\mathscr{V}_{\textrm{orig}}$ in some cases.  Example plot
      for $|V|=30$ vertices and $|\sigma(V)|=10$ colors in each graph.  200
      generated graph pairs per combination of arc insertion (ins.) and
      deletion (del.) probabilities.}
    \label{fig:single-vertex-cut}
  \end{figure}
  
  \clearpage
  
  \section{Performance analysis in beBMG editing}
  \label{sect:app-beBMG}
  
  \begin{figure}[htb]
    \centering
    \includegraphics[width=0.85\textwidth]{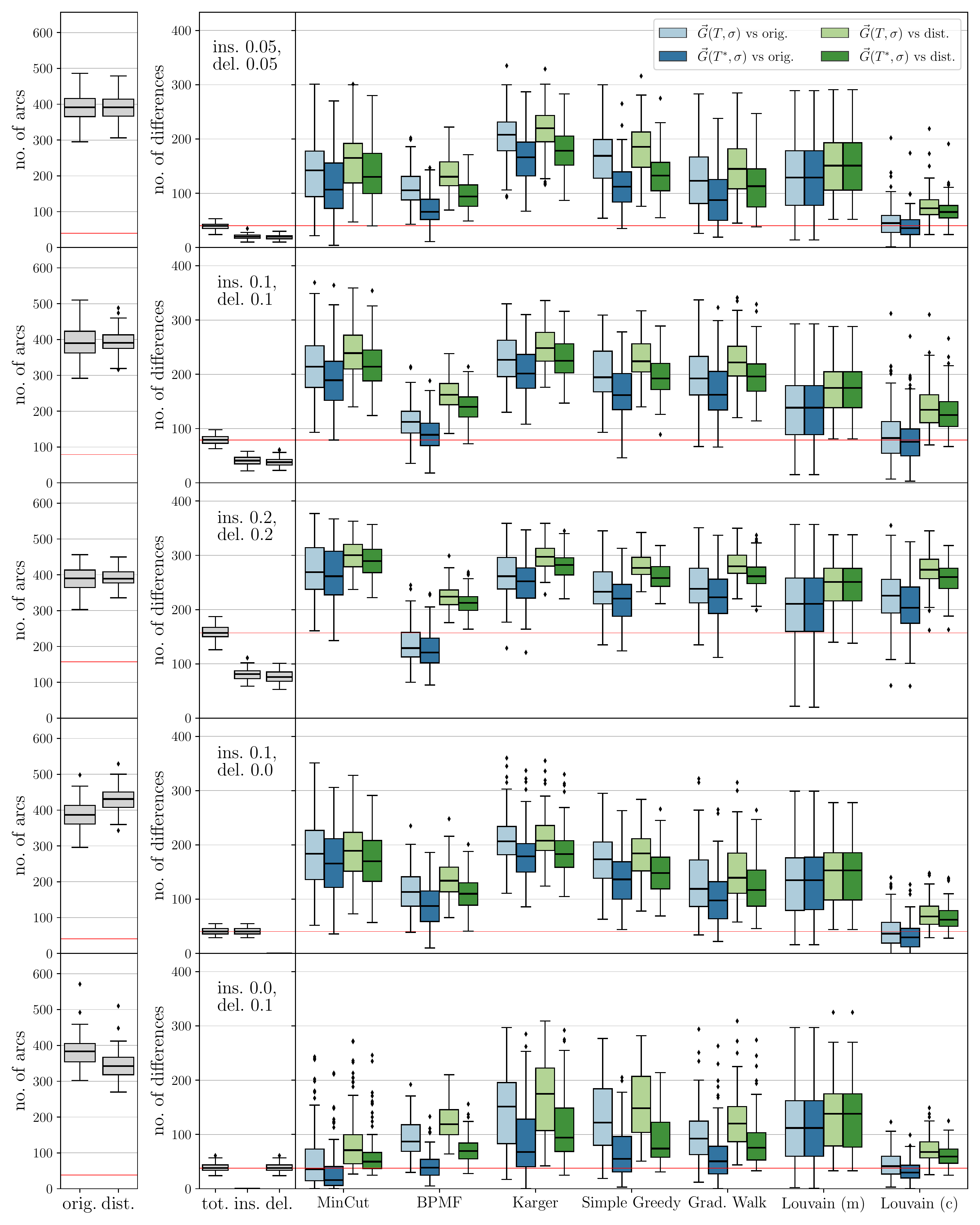}
    \caption{Performance comparison of several beBMG editing heuristics based
      on the no.\ of arc differences. See Fig.~\ref{fig:edit_comparison1} for
      further description.  Example plot for $|V|=30$ vertices and
      $|\sigma(V)|=10$ colors in each graph, 100 graphs per row.}
    \label{fig:edit-comparison-binary}
  \end{figure}

\end{appendix}

\bibliography{preprint-bmg-heuristics}


\end{document}